\documentclass[a4paper, 10pt, american]{amsart}

\usepackage{times,latexsym,amssymb}
\usepackage{amsmath,amsthm,bm}
\usepackage{xcolor}
\usepackage[colorlinks,pdfpagelabels,pdfstartview = FitH,bookmarksopen
= true,bookmarksnumbered = true,linkcolor = blue,plainpages =
false,hypertexnames = false,citecolor = red,pagebackref=false]{hyperref}

\usepackage{amsbsy}
\usepackage{amstext}
\usepackage{amssymb}
\usepackage{esint}
\usepackage{stmaryrd}
\usepackage[pdftex]{graphicx}
\usepackage{floatflt}
\usepackage{comment}

\usepackage{ifthen}
\newcount\full
\allowdisplaybreaks
\newtheorem{theorem}{Theorem}
\newtheorem{lemma}[theorem]{Lemma}
\newtheorem{definition}[theorem]{Definition}
\newtheorem{proposition}[theorem]{Proposition}
\newtheorem{corollary}[theorem]{Corollary}
\newtheorem{remark}[theorem]{Remark}
\numberwithin{theorem}{section}
\numberwithin{equation}{section}

\newcommand{\mint}{- \mskip-19,5mu \int}
\newcommand{\tmint}{- \mskip-16,5mu \int}
\def\N{\mathbb{N}}
\def\R{\mathbb{R}}

\renewcommand{\d}{\mathrm{d}}
\newcommand{\dx}{\mathrm{d}x}
\newcommand{\dy}{\mathrm{d}y}

\newcommand{\dt}{\mathrm{d}t}
\newcommand{\ds}{\mathrm{d}s}

\newcommand{\drho}{\mathrm{d}\rho}

\renewcommand{\epsilon}{\varepsilon}

\newcommand{\al}{\alpha}

\newcommand{\be}{\beta}

\newcommand{\gm}{\gamma}

\newcommand{\sig}{\sigma}

\newcommand{\om}{\omega}
\newcommand{\Om}{\Omega}

\DeclareMathOperator{\spt}{spt}
\DeclareMathOperator{\Div}{div}

\DeclareMathOperator{\loc}{loc}

\DeclareMathOperator{\Tail}{Tail}
\renewcommand{\epsilon}{\varepsilon}
\newcommand{\eps}{\varepsilon}
\renewcommand{\rho}{\varrho}

\def\eqn#1$$#2$${\begin{equation}\label#1#2\end{equation}}

\newcommand{\btau}{\boldsymbol{\tau}}

\def\Xint#1{\mathchoice
    {\XXint\displaystyle\textstyle{#1}}%
    {\XXint\textstyle\scriptstyle{#1}}%
    {\XXint\scriptstyle\scriptscriptstyle{#1}}%
    {\XXint\scriptscriptstyle\scriptscriptstyle{#1}}%
    \!\int}
\def\XXint#1#2#3{\setbox0=\hbox{$#1{#2#3}{\int}$}
    \vcenter{\hbox{$#2#3$}}\kern-0.5\wd0}

\def\dashint{\Xint{\raise4pt\hbox to7pt{\hrulefill}}}
\def\tmint{\Xint{\raise0pt\hbox to6pt{\hrulefill}}}

\def\XXiint#1#2#3{\setbox0=\hbox{$#1{#2#3}{\iint}$}
    \vcenter{\hbox{$#2#3$}}\kern-0.5\wd0}


\author[V. B\"ogelein]{Verena B\"{o}gelein}
\address{Verena B\"ogelein\\
Fachbereich Mathematik, Universit\"at Salzburg\\
Hellbrunner Str. 34, 5020 Salzburg, Austria}
\email{verena.boegelein@plus.ac.at}

\author[F. Duzaar]{Frank Duzaar}
\address{Frank Duzaar\\
Fachbereich Mathematik, Universit\"at Salzburg\\
Hellbrunner Str. 34, 5020 Salzburg, Austria}
\email{frankjohannes.duzaar@plus.ac.at}

\author[N. Liao]{Naian Liao}
\address{Naian Liao\\
Fachbereich Mathematik, Universit\"at Salzburg\\
Hellbrunner Str. 34, 5020 Salzburg, Austria}
\email{naian.liao@plus.ac.at}

\author[G. Molica Bisci]{Giovanni Molica Bisci}
\address{Giovanni Molica Bisci\\ 
Dipartimento di Scienze Pure e Applicate (DiSPeA), University of Urbino Carlo Bo\\
Piazza della Repubblica, 13, 61029 Urbino, Italy}
\email{giovanni.molicabisci@uniurb.it}

\author[R. Servadei]{Raffaella Servadei}
\address{Raffaella Servadei\\ 
Dipartimento di Scienze Pure e Applicate (DiSPeA), University of Urbino Carlo Bo\\
Piazza della Repubblica, 13, 61029 Urbino, Italy}
\email{raffaella.servadei@uniurb.it}

\begin{document}

\title[Regularity for the fractional $p$-Laplace equation]{Regularity for the fractional $p$-Laplace equation}

\date{\today}

\begin{abstract}

Higher Sobolev and H\"older regularity is studied for local weak solutions of the fractional $p$-Laplace equation of order $s$ in the case $p\ge 2$. Depending on the regime considered, i.e. $$0<s\le\tfrac{p-2}{p}\quad \mbox{or} \quad\tfrac{p-2}{p}<s<1,$$
precise local estimates are proven. The relevant estimates are stable if the fractional order  $s$ reaches $1$; the known Sobolev regularity estimates for the local $p$-Laplace are recovered. The case $p=2$ reproduces the  almost $W^{1+s,2}_{\rm loc}$-regularity for the fractional Laplace equation of  any order $s\in(0,1)$.

\end{abstract}
\maketitle
\tableofcontents

\newpage

\section{Introduction}
In this paper we study higher Sobolev and H\"older regularity of locally bounded, local weak solutions of the fractional $p$-Laplace equation of order $s\in (0,1)$
and $p\ge 2$ on a bounded domain $\Omega\subset \R^N$ with dimension $N\ge 2$: 
\begin{equation}\label{PDE}
    (-\Delta_p)^s u :=\text{p.v.}\int_{\R^N}\frac{2|u(x)-u(y)|^{p-2}(u(x)-u(y))}{|x-y|^{N+sp}}\,\dx= 0.
\end{equation}
For the precise notion of weak solution we refer to Definition \ref{def:loc-sol}. 
Recently, much attention has been paid to this kind of nonlocal operators. The interest stems from their challenging, mathematical structures and their connections with concrete applications, such as continuum mechanics, phase transition, population dynamics, optimal control and game theory. 
To our knowledge, operators of this type were first introduced in \cite{Andreu-Mazon-Rossi-Toledo, Ishii-Nakamura}.

Our {\bf main results} are divided into two parts according to the regime of $s$, namely either $s\in (\frac{p-2}{p}, 1)$ or $s\in (0,\frac{p-2}{p}]$. In the {\bf first part}, we establish that $\nabla u$ belongs to the fractional Sobolev space $W^{\beta ,q}_{\rm loc}(\Om)$ for any $q\ge p$ and any $\beta\in (0,\frac{p}{q}(s-\frac{p-2}{p}))$. In particular,  $\nabla u$ belongs to $L^q_{\loc}(\Om)$ for any $q\ge p$; a direct consequence of this result is that $u\in C^{0,\al}_{\loc}(\Om)$ for any $\al\in(0,1)$.
Precise local estimates will be presented in Theorems~\ref{thm:W1p}, \ref{thm:W1q} and \ref{*thm:beta-q} regarding the claimed gradient regularity properties and Theorem~\ref{thm:Hoelder s>} regarding the almost Lipschitz continuity of $u$. All these estimates are stable as $s\uparrow1$.
 Whereas in the {\bf second part}, it is unknown if $\nabla u$ exists in the Sobolev sense. However, the fractional differentiability order $s$ has been improved to any number less than $\frac{sp}{p-2}$, whereas the integrability order $q$ can be any number larger than $p$, namely $u\in W_{\loc}^{\gm,q}(\Om)$ for any $q\ge p$ and $\gm\in [s,\frac{sp}{p-2})$. As a corollary, we have $u\in C^{0,\gm}_{\loc}(\Om)$ for any $\gm\in(0,\frac{sp}{p-2})$. Precise estimates regarding the regularity properties of the second part are given in Theorems~\ref{thm:Wgq} and \ref{thm:Hoelder-subcritical}.

From a variational point of view, the fractional $p$-Laplace operator \eqref{PDE}
can be considered as a non-local cousin of the classical $p$-Laplace operator
\begin{equation}\label{p-Lapl}
        -\Div\big(|\nabla u|^{p-2}\nabla u\big)=0.
\end{equation}
Classical results of Uraltseva \cite{Uraltseva} (for  equations) and Uhlenbeck \cite{Uhlenbeck} (for  systems) state that the gradient of local weak solutions of \eqref{p-Lapl} is locally H\"older continuous. Such a regularity result lays the foundation for a number of further theories.
Nevertheless, up to now it is still elusive whether an analogue of this result holds true for the fractional $p$-Laplace operator \eqref{PDE}. In fact, it is even far from trivial to assert that $\nabla u$ exists in the Sobolev sense. To our best knowledge, this was confirmed by Brasco \& Lindgren; they established that when   $s\in (\frac{p-1}{p}, 1)$, $\nabla u \in L^p_{\loc}(\Om)$ for globally bounded solutions,  cf.~\cite[Corollary~1.8]{Brasco-Lindgren}; moreover, the range of $s$ can be improved to   $s\in (\frac{p-1}{p+1}, 1)$ if $u$ is a solution of a certain Dirichlet problem, cf.~\cite[Corollary~1.9]{Brasco-Lindgren}.

Our main contribution significantly improves this result and establishes the higher integrability of $\nabla u$ under a wider range of $s$, namely $\nabla u \in L^q_{\loc}(\Om)$ for any $q\ge p$ and any $s\in (\frac{p-2}{p}, 1)$. Moreover, this improvement is achieved under the mere notion of local solution, and no additional assumption is imposed on the solution's global behavior.

Another classical result 
states that any local solution of \eqref{p-Lapl} satisfies that $|\nabla u|^\frac{p-2}2\nabla u\in W^{1,2}_{\rm loc}(\Omega)$; see \cite{Bojarski:1987, Uhlenbeck, Uraltseva}.  This higher differentiability can be converted into fractional differentiability  by a standard argument, that is  $\nabla u\in W^{\beta,p}_{\rm loc} (\Omega,\R^N)$ for any $0<\beta<\frac{2}{p}$, cf.~Remark~\ref{Rmk:stability}. 
Our result indicates particularly that when $s\in (\frac{p-2}{p}, 1)$, local solutions of the fractional $p$-Laplacian \eqref{PDE} satisfies $\nabla u\in W^{\beta ,p}_{\rm loc}(\Om)$ for any $\beta\in (0,s-\frac{p-2}{p})$, and thus formally recovers the classical result in the limit $s\uparrow 1$. In this sense our range of $\be$ is sharper than the one obtained previously by Brasco \& Lindgren \cite[Corollaries~1.8 \& 1.9]{Brasco-Lindgren}. Moreover, our approach dispenses with any additional assumption on the solution's global behavior and relies solely on the notion of local solution.

Last but not least, we have substantially improved the higher H\"oder regularity  for solutions of the fractional $p$-Laplace equation \eqref{PDE} as well. Indeed, the H\"older exponent has been improved to any number less than $\min\{1,\frac{sp}{p-2}\}$ in contrast to the known $\min\{1,\frac{sp}{p-1}\}$. In particular, the ``almost Lipschitz" regularity improves the one obtained by Brasco \& Lindgren \& Schikorra \cite[Theorem~5.2]{Brasco-Lindgren-Schikorra} in the sense that the admissible range of $s$ has been extended from $ [\frac{p-1}{p}, 1)$ to $[\frac{p-2}{p}, 1)$.
In the particular case $p=2$ we thus recover the whole range $s\in(0,1)$, which is in perfect accordance with known regularity theory for the fractional Laplacian.

Examining the effect of an inhomogeneous term on the right-hand side of \eqref{PDE} has been considered in \cite{Brasco-Lindgren,Brasco-Lindgren-Schikorra}. Interesting though it is, we decide to concentrate on the homogeneous equation in this manuscript. We believe our new techniques can also be applied to such a case.

The effort poured in this manuscript induces further study regarding the gradient regularity for the fractional $p$-Laplace equation \eqref{PDE}. Our next step is to have a more complete picture concerning the higher regularity theory. In particular, we would like to generalize the results to equations with more general kernels,  and to examine the case $p<2$ as well. We expect similar results to hold in the whole range $s\in (0,1)$ and $p\in (1,2)$.

\subsection{Statement of the main results}
The main results differ depending on which regime of the fractional differentiability order $s$ is considered.  It turns out that the  case $p>2$ and $s\in (0,\frac{p-2}{p})$ deviates significantly from the case $p\ge 2$ and $s\in (\frac{p-2}{p},1)$.  First, we present the main result for the range $s\in (0,\frac{p-2}{p}]$. This guarantees that for  locally bounded, weak solutions of the fractional $(s,p)$-Laplace equation, the integrability can be improved to any $q\ge p$ as well as the fractional differentiability to any $\gamma\in [s,\frac{sp}{p-2})$. The precise statement is as follows; for the definition of $\Tail(u;R)$ we refer to~\eqref{Eq:tail} below.

\begin{theorem}[Almost $W^{\frac{sp}{p-2},q}$-regularity]\label{thm:Wgq}
Let $p\in(2,\infty)$, and $s\in(0,\frac{p-2}{p}]$. Then, for any locally bounded, local weak solution $u\in W^{s,p}_{\rm loc}(\Omega)\cap L^{p-1}_{sp}(\R^N)$ of~\eqref{PDE} in the sense of Definition~\ref{def:loc-sol}, we have
\[
    u\in W^{\gamma,q}_{\loc}(\Om)
    \qquad 
    \mbox{for any $q\in [p,\infty)$, and $\gamma\in \big[s,\frac{sp}{p-2}\big)$.}
\]
Moreover, there exists a universal constant $C=C(N,p,s,q,\gamma)\ge 1$, such that for any ball $B_{R}\equiv B_{R}(x_o)\Subset \Omega$ we have
\begin{align*}
    [u]_{W^{\gamma,q}(B_{\frac12 R})}
    \le 
    \frac{C}{R^{\gamma}}\Big[
    R^{s-N(\frac{1}{p}-\frac{1}{q})}[u]_{W^{s,p}(B_R)} +
    R^{\frac{N}{q}} \big(\|u\|_{L^\infty(B_{R})} + 
    \Tail(u;R)\big) \Big] .
\end{align*}
The constant $C$ blows up as $\gm\uparrow \frac{sp}{p-2}$.
\hfill $\Box$
\end{theorem}

Using the previous theorem and the Morrey-type embedding for
fractional Sobolev spaces, we immediately obtain that locally bounded local weak solutions of the fractional $(s,p)$-Laplace equation in the  range $ s\in (0,\frac{p-2}{p}]$ are locally Hölder continuous with any exponent $\gamma \in (0,\frac{sp}{p-2})$. The precise statement is as follows.

\begin{theorem}[Almost $C^{0, \frac{sp}{p-2}}$-regularity]\label{thm:Hoelder-subcritical}
Let $p\in(2,\infty)$ and $s\in(0,\frac{p-2}{p}]$. Then, for any locally bounded, local weak solution $u\in W^{s,p}_{\rm loc}(\Omega)\cap L^{p-1}_{sp}(\R^N)$ of~\eqref{PDE} in the sense of Definition~\ref{def:loc-sol}, we have
\begin{equation*}
    u\in C^{0,\gamma}_{\loc}(\Om)\qquad
    \mbox{for any $\gamma \in \big(0,\frac{sp}{p-2}\big)$.}
\end{equation*} Moreover, there exists a universal constant 
$C=C(N,p,s,\gamma) \ge 1$, such that for any ball $B_{R}\equiv B_{R}(x_o)\Subset \Omega$ we have
\begin{equation*}
    [u]_{C^{0,\gamma}(B_{\frac12R})}
    \le 
    \frac{C}{R^{\gamma}}
    \Big[ R^{s-\frac{N}p}[u]_{W^{s,p} (B_{R})} + 
    \|u\|_{L^{\infty}(B_{R})} +
    \mathrm{Tail}(u;R)\Big].
\end{equation*}
The constant $C$ blows up as $\gamma\uparrow\frac{sp}{p-2}$.\hfill $\Box$
\end{theorem}

In the case of $p\ge 2$ and $s\in (\frac{p-2}{p}, 1)$,  better regularity properties of local weak solutions can be achieved. The first result concerns the gradient regularity in $L^p$.  Roughly speaking, it states that for locally bounded, local weak solutions $u$ of the $(s,p)$-Laplace equation, the weak gradient $\nabla u$ exists and is locally in $L^p(\Omega,\R^N)$.

\begin{theorem}[$L^{p}$-gradient regularity]\label{thm:W1p}
Let $p\ge 2$ and $s\in(\frac{p-2}{p},1)$. Then, for any locally bounded, local weak solution $u\in W^{s,p}_{\rm loc}(\Omega)\cap L^{p-1}_{sp}(\R^N)$ of~\eqref{PDE} in the sense of Definition~\ref{def:loc-sol}, we have
$$
    u\in W^{1,p}_{\rm loc}(\Omega).
$$
Moreover, there exists a universal constant $C=C(N,p,s)$, such that for any ball $B_{R}\equiv B_{R}(x_o)\Subset \Omega$ we have
\begin{align*}
    \|\nabla u\|_{L^p(B_{\frac12 R})}
    &\le 
    \frac{C}{R} \Big[  R^{s}(1-s)^{\frac1p}[u]_{W^{s,p}(B_{R})} +
    R^{\frac{N}{p}} \big(\|u\|_{L^\infty(B_{R})} + \Tail(u;R)\big)\Big].
\end{align*}
The constant $C$ is stable as  $s\uparrow 1$ and blows up as $s\downarrow \frac{p-2}{p}$.
\hfill$\Box$
\end{theorem}
At this point the question naturally arises whether the $L^p$-gradient regularity  of a locally bounded, local weak solutions to the fractional $(s,p)$-Laplace equation can be improved. This can in fact be answered positively. It turns out that for any $q>p$ the $L^q$-gradient regularity holds. The higher gradient regularity is the core of the following theorem.

\begin{theorem}[$L^q$-gradient regularity]\label{thm:W1q}
Let $p\in[2,\infty)$ and  $s\in(\frac{p-2}{p},1)$. Then, for any locally bounded, local weak solution $u\in W^{s,p}_{\rm loc}(\Omega)\cap L^{p-1}_{sp}(\R^N)$ of~\eqref{PDE} in the sense of Definition~\ref{def:loc-sol}, we have
$$
    u\in W^{1,q}_{\rm loc}(\Omega)\qquad \mbox{for any $q\in [p,\infty)$.}
$$
Moreover, there exists a universal constant  $C=C(N,p,s,q)$, such that for any ball $B_{R}\equiv B_{R}(x_o)\Subset \Omega$ the quantitative $L^q$-gradient estimate 
\begin{align*}
    \|\nabla u\|_{L^q(B_{\frac12 R})}
    \le
    C R^{\frac{N}{q}-1} 
    \Big[R^{s-\frac{N}{p}} (1-s)^{\frac1p}[u]_{W^{s,p}(B_{R})} +
    \|u\|_{L^{\infty}(B_{R})}+\mathrm{Tail}(u;R)\Big]
\end{align*}
holds true. The constant $C$ is stable in the limit  $s\uparrow 1$ and blows up as $s\downarrow \frac{p-2}{p}$.
\end{theorem}

At this stage, the classical  Morrey-type embedding for
the Sobolev space $W^{1,q}$ with $q>N$, implies that locally bounded, local weak solutions to the fractional $(s,p)$-Laplace equation in  the regime $s\in (\frac{p-2}{p},1)$ are locally Hölder continuous with exponent $1-\frac{N}{q}$. Since $q$ can be chosen arbitrarily large by Theorem \ref{thm:W1q}, this means H\"older continuity for any H\"older exponent in $ (0,1)$. 

\begin{theorem}[Almost Lipschitz continuity]\label{thm:Hoelder s>}
Let $p\in[2,\infty)$ and $s\in(\frac{p-2}{p},1)$.  Then,  for any locally bounded, local weak solution $u\in W^{s,p}_{\rm loc}(\Omega)\cap L^{p-1}_{sp}(\R^N)$ of~\eqref{PDE} in the sense of Definition~\ref{def:loc-sol}, we have
$$
    u\in C_{\loc}^{0,\gamma}(\Omega)\qquad\mbox{for any $\gamma\in(0,1)$.}\
$$
Moreover,  there exists a universal constant  $C=C(N,p,s,\gamma)$, such for any ball $B_R\equiv B_R(x_o)\Subset\Omega$ we have
\begin{align*}
     [u]_{C^{0,\gamma}(B_{\frac12 R})}
     \le 
     \frac{C}{R^{\gamma}}
    \Big[ R^{s-\frac{N}p}(1-s)^{\frac1p}[u]_{W^{s,p} (B_{R})} +
    \|u\|_{L^{\infty}(B_{R})} +
    \mathrm{Tail}(u;R)\Big] .
\end{align*}
The constant $C$ is stable as $s\uparrow1$ and blows up as $s\downarrow\frac{p-2}{p}$.
\hfill $\Box$
\end{theorem}
In analogy to the local $p$-Laplace equation, it can also be shown that the gradient $\nabla u$ of a locally bounded, local weak solution $u$ to the fractional $(s,p)$-Laplace equation admits  a certain fractional differentiability. The precise result is as follows.

\begin{theorem}[Almost $W^{\frac{sp-(p-2)}{q},q}$-gradient regularity]\label{*thm:beta-q}
Let $p\in\![2,\infty)$ and $s\in\!(\frac{p-2}{p},1)$. Then,  for any locally bounded,  local weak solution $u\in W^{s,p}_{\rm loc}(\Omega)\cap L^{p-1}_{sp}(\R^N)$ of~\eqref{PDE} in the sense of Definition~\ref{def:loc-sol}, we have
$$
    \nabla u\in W^{\alpha,q}_{\rm loc}(\Omega)\qquad 
    \mbox{for any $q\in [p,\infty)$, and $\alpha \in\big(0,\tfrac{sp-(p-2)}{q}\big)$.}
$$
Moreover, there exists a universal constant $C=C(N,p,s,q,\alpha)$ such that for any ball $B_{R}\equiv B_R(x_o)\Subset \Omega$ we have 
\begin{align*}
    [\nabla u]_{W^{\alpha, q}(B_{\frac12 R})}
    &\le 
    \frac{C R^\frac{N}{q}}{R^{1+\alpha}} 
    \Big[R^{s-\frac{N}{p}} (1-s)^{\frac{1}{p}}[u]_{W^{s,p}(B_{R})} +
    \|u\|_{L^{\infty}(B_{R})}+\mathrm{Tail}(u;R) \Big],
\end{align*}
Moreover, the constant $C$ is stable as $s\uparrow 1$ and blows up as $s\downarrow \frac{p-2}{p}$ and $\alpha\uparrow \frac{sp-(p-2)}{q}$.
\end{theorem}

\subsection{Brief summary of state of the art}\label{state-of-art}
Since our results concern an elliptic problem, we refrain from discussing time-dependent problems which deserve an independent treatment.
We start with an incomplete overview of the results that are known in the case of linear non-local equations. The theory here is far less fragmented than for the fractional $p$-Laplace operator. 
The topic of inner regularity of the fractional Laplacian of order $s$ and  linear fractional operators with more general kernels is still a very active field of research. It has attracted a lot of attention in recent years, as can be seen from the large and fast growing number of results. We refer, without being exhaustive at all, to the inner regularity results in \cite{Abatangelo-Ros-Oton, Abels-Grubb,Bass-1,Bass-2,Cozzi-1,Dipiero-Savin-Valdinoci,Dipiero-Ros Oton-Serra-Valdinoci,Dyda-Kassmann,Fall,Imbert-Silvestre,Jin-Xiong,Kassmann,Ros-Oton-Serra,Serra-1,Serra-2,Silvestre-06}. In these works, the equations differ in terms of the assumptions made with respect to the integral kernel. However, a common feature is that essentially all of them considered integral operators of linear growth. Weak Harnack inequalities were studied in \cite{Dyda-Kassmann, Kassmann}. Global regularity, i.e.~$C^s$-H\"older continuity up to the boundary, was studied \cite{Ros-Oton-Serra:bd1,Ros-Oton-Serra:bd2}; see also \cite{Grubb-1, Grubb-2}. Higher-order boundary regularity in the sense that $\frac{u}{{\rm dist}^s (\cdot,\partial\Omega)}$ is H\"older continuous on the closure of the domain $\Omega$, was established in \cite{Abatangelo-Ros-Oton, Abels-Grubb, Grubb-1,Grubb-2,Ros-Oton-Serra:bd1,Ros-Oton-Serra:bd2}. 
Gradient potential estimates  have been obtained in \cite{Kuusi-Nowak-Sire} in the regime $s\in (\frac12,1)$; see also \cite{Diening-Kim-Lee-Nowak} for more general nonlocal  equations of linear growth.  
On the other hand, regularity results on the scale of $L^p$-spaces  can be found in \cite{Bass-Ren, Kuusi-Mingione-Sire-1, Kuusi-Mingione-Sire-2}; see also \cite{Auscher-Bortz-Saari}.  For a similar result in case of the fractional $p$-Laplace with $p\ge2$ we refer to \cite{Schikorra-1}. More precisely, a nonlocal self-improving property
(fractional Gehring lemma) has been established. 
An  interior $L^p$-regularity theory -- Calder\'on-Zygmund theory --  has been developed in
\cite{Fall-Mengesha-Schikorra-Yeepo, Mengesha-Schikorra-Yeepo}; see also \cite{Byun-Kim, Diening-Nowak}.
For a Wiener-type criterion for boundary regularity we refer
to \cite{Kim-Lee-Lee}. A Harnack inequality is established in \cite{DKP-2}, see also~\cite{Cozzi-1}.
Lastly, for nonlinear fractional equations (including the fractional $p$-Lapacian), an existence, regularity and potential theory with measure data was developed in \cite{Kuusi-Mingione-Sire-3}.

As already indicated, the regularity theory  for the fractional $p$-Laplacian with $p\neq2$ is far from well developed, and many of the basic questions are still unanswered. Concerning weak solutions, local boundedness and H\"older regularity with a qualitative
H\"older exponent in the interior and at the boundary have been established in \cite{DiCastro-Kuusi-Palatucci, Korvenpaa-Kuusi-Palatucci}. 
The results cover the whole range of $s\in (0,1)$ and $p\in (1,\infty)$. For a similar result in the framework of viscosity solutions we refer to \cite{Lindgren-1}. Higher order boundary regularity has been achieved for $p\ge 2$ in \cite{Iannizzotto-Mosconi-Squassina}. In the super-quadratic case $p\ge 2$ the fractional differentiability of weak solutions of the fractional $(s,p)$-Laplace equation has been improved quantitatively, whereas in the regime $s\in (\frac{p-1}{p},1)$, the gradient of weak solutions exists in $L^p$ and additionally exhibits a certain fractional differentiability; these are achieved in~\cite{Brasco-Lindgren}. Still in the superquadratic case $p\ge 2$, interior higher H\"older regularity with an explicit H\"older exponent  was established in \cite{Brasco-Lindgren-Schikorra}. More precisely,
in the regime $s\in (0,\frac{p-1}{p}]$ weak solutions are almost $C^{0,\frac{sp}{p-1}}_{\rm loc}$, while in the range $s\in (\frac{p-1}{p}, 1)$ solutions are almost Lipschitz continuous. This kind of higher H\"older regularity has been recently extended in \cite{Lindgren-Garain} to the subquadratic case $1<p<2$ with the same threshold with respect to $s$ and $p$. To our knowledge, these are the state of the art regarding the higher H\"older regularity.

\subsection{Novel techniques}
To our knowledge, a method of difference quotients in the context of fractional $(s,p)$-Laplace equations has first been implemented in the pioneering work \cite{Brasco-Lindgren} of Brasco \& Lindgren, which deals with the higher Sobolev regularity; see  \cite{Caffarelli-Silvestre} for a different nonlocal equation. Later, a separate yet similar program is carried out to deal with the higher H\"older regularity in \cite{Brasco-Lindgren-Schikorra} by Brasco \& Lindgren \& Schikorra. Their program features finite iterations in Besov-type spaces. Although our approach also relies on a difference quotient technique, the overall strategy differs dramatically from the one used in \cite{Brasco-Lindgren, Brasco-Lindgren-Schikorra}. 

Amongst other things, a novel {\it tail estimate} plays a pivotal role in our approach. It nicely captures the long-range behavior of solutions in the finite difference scheme and permits us to efficiently run iterations at various stages. More importantly, together with our new iteration technique, it allows us to concentrate on the higher Sobolev regularity, whereas the higher H\"older regularity is deduced upon applying the Morrey-type imbedding in the last step. This technical route, granted by the tail estimate, sets our approach apart from the existing one. In fact, 
to understand our approach, it is instructive to keep in mind the following imbedding:
\[
u\in W^{\gm, q}_{\loc}(\Om),\> q> N\quad \implies \quad  u\in C_{\loc}^{\gm-\frac{N}q}(\Om).
\]
It holds true no matter whether $\gm\in(0,1)$ or $\gm=1$.

The novel tail estimate also gives birth to a natural dichotomy that distinguishes the existence of $\nabla u$ or not.
Indeed, it ensures sufficient amount of gain in fractional differentiability for each step of our iteration and allows us to approach the limiting order $\frac{sp}{p-2}$. If $s\in(\frac{p-2}p,1)$, the limiting order exceeds $1$, and hence the $W^{1,p}$-regularity follows. We refer the reader to the beginning of \S~\ref{sec:W1p} for more explanation on this iteration technique.

The next step hinges upon improving the integrability exponent from $p$ to any number $q>p$. The arguments for $s\in(0,\frac{p-2}p]$ and for $s\in(\frac{p-2}p,1)$ are different. For the former case, we refer to the beginning of \S~\ref{sec:frac-smalls} for an explanation. Whereas the underlying idea of  the latter case is somewhat similar to a Moser-type iteration scheme. 
Indeed, assuming that $\nabla u\in L^{q}_{\loc}(\Om)$ for some $q\ge p$, we establish an improved estimate for the second-order finite differences,  which joint with a result of Stein gives $\nabla u\in W^{\alpha,q}_{\loc}(\Om)$ for a small differentiability order $\alpha$. In turn, this implies $\nabla u\in L^{\frac{Nq}{N-\al q}}_{\loc}(\Om)$ by the Sobolev-type embedding for fractional spaces. Such a procedure is then iterated finite times until the desired integrability exponent is reached.
Differently from Moser's iteration for the local $p$-Laplace equation, our argument cannot be iterated infinitely many times.

Finally, we stress that, to our knowledge, the tail estimate is completely new and has the potential to be one of the key ingredients also for other regularity results in the context of fractional differential equations.

\medskip
 
\noindent
{\bf Acknowledgments.}  N.~Liao is supported by the FWF-project P36272-N \emph{On the Stefan type problems.} 
G.~Molica Bisci and ~R.~Servadei have
been funded by the European Union - NextGenerationEU within the framework
of PNRR  Mission 4 - Component 2 - Investment 1.1 under the Italian
Ministry of University and Research (MUR) program PRIN 2022 - grant number
2022BCFHN2 - Advanced theoretical aspects in PDEs and their applications -
CUP: H53D23001960006 and partially supported by the INdAM-GNAMPA Research
Project 2024: Aspetti geometrici e analitici di alcuni problemi locali e
non-locali in mancanza di compattezza - CUP E53C23001670001. 

\section{Preliminaries}

\subsection{Notation and definitions}
Throughout the manuscript $C$ denotes a generic constant which can change from line to line. In the statements and also in the proofs, we trace the dependencies of the constants in terms of the data.    We indicate the dependencies by writing, for instance, $C=C(N,p,s)$ if $C$ depends on $N,p$ and $s$. Next, we denote $B_R(x_o)\subset \R^N$ to be the ball of radius $R$ and center $x_o$ in $\R^N$. Whereas we define
$$
    K_R(x_o):=B_R(x_o)\times B_R(x_o).
$$
We use this notation at various points to give the double integrals a more compact form.

Let $\Om\subset\R^N$ be a bounded  open set and $\al\in(0,1)$. For a function $w\colon\Om\to \R$ we understand by $w\in C_{\loc}^{0,\al}(\Om)$ that for any ball $B_R(x_o)
\Subset \Om$ we have that $w\in C^{0,\al}\big(\overline{B_R(x_o)}\big)$ holds. 
Moreover, we denote the semi-norm
\[
[w]_{C^{0,\al}(B_R(x_o))}:=\sup_{x\not=y\in B_R(x_o)}\frac{|w(x)-w(y)|}{|x-y|^\al}.
\]
We also introduce the {\bf fractional Sobolev space} $W^{\gamma,q}(\Omega)$ with some $q\in [1,\infty )$ and
$\gamma\in (0,1)$. A measurable function $w\colon\Omega\to\R$ belongs to the fractional Sobolev space $W^{\gamma, q}(\Omega)$ if and only if
\begin{align*}
    \| w\|_{W^{\gamma ,q}(\Omega)}
    &:=
    \| w\|_{L^q(\Omega)} +
    [w]_{W^{\gamma ,q}(\Omega)}
    <\infty,
\end{align*}
where the semi-norm is defined as 
\begin{align*}
    [w]_{W^{\gamma ,q}(\Omega)}
    &:=
    \bigg[
    \iint_{\Omega\times\Omega}
    \frac{|w(x)-w(y)|^q}{|x-y|^{N+\gamma q}}\,\dx\dy
    \bigg]^\frac{1}{q} .
\end{align*}
Some useful results concerning fractional Sobolev spaces are collected in \S \ref{sec:fractional};
for more information we refer to \cite{Hitchhikers-guide}.

For $q>0$ and $\gamma>0$ the {\bf Tail space} $L^q_\gamma (\R^N)$ consist of all  $w\in L^q_{\rm loc}(\R^N)$  that satisfy
$$
    \int_{\R^N}\frac{|w|^q}{(1+|x|)^{N+\gamma}}\, \dx <\infty.
$$
The following quantity -- called the {\bf tail} -- measures the global behavior of a function  $w$ belonging to the tail space $L^{p-1}_{sp}(\R^N)$. It  plays an essential role in our quantitative estimates.
\begin{equation}\label{Eq:tail}
{\rm Tail}(u; x_o, R):= \bigg[R^{sp}\int_{\R^N\setminus B_R(x_o)}\frac{|u(x)|^{p-1}}{|x-x_o|^{N+sp}}\,\dx \bigg]^{\frac1{p-1}}.
\end{equation}
It is not difficult to show that $\Tail (u;x_o,R)<\infty$ whenever
$u\in L^{p-1}_{sp}(\R^N)$. If $x_o=0$ or if the center point is clear from the context, we omit it in the notation.

Our notion of solution is local in nature.
\begin{definition}[Local weak solution]\label{def:loc-sol}
 Let $\Omega\subset\R^N$ be  bounded open set, $p\in(1,\infty)$ and $s\in(0,1)$. A function $u\in W^{s,p}_{\rm loc}(\Omega)\cap L^{p-1}_{sp}(\R^N)$ is a local weak solution of~\eqref{PDE} in $\Omega$ if and only if
\begin{equation}\label{weak-sol}
        \iint_{\R^N\times\R^N}\frac{|u(x)-u(y)|^{p-2}(u(x)-u(y))(\varphi(x)-\varphi(y))}{|x-y|^{N+sp}}\,\dx\dy =0
\end{equation}
for every $\varphi\in W^{s,p}(\Omega)$ compactly supported in $\Omega$ and extended to $0$ outside $\Omega$.
\end{definition}

\subsection{Algebraic inequalities}
In this section we will summarize the algebraic inequalities that will be used in the rest of the paper. In regularity theory, the most optimal algebraic inequalities are crucial, because  
they usually provide optimal results.
For $\gamma\in(0,\infty)$ we define
$$
    V_\gamma(a)
    :=
    |a|^{\gamma-1}a,
    \qquad\mbox{for $a\in\R$.}
$$
If $a=0$ we set $V_\gamma(a)=0$ also for $\gamma\in(0,1)$. The basic algebraic inequality relating the difference $|V_\gamma(b) - V_\gamma(a)|$ to $|a-b|$ can be found in  \cite[Lemma~2.1]{Acerbi-Fusco} for $\gamma\in(0,1)$, and \cite[Lemma~2.2]{GiaquintaModica:1986} for $\gamma\in(1,\infty)$. The stated values of the constants can be derived by a careful inspection of the proofs.

\begin{lemma}\label{lem:Acerbi-Fusco}
For any $\gamma>0$, and for all $a,b\in\R$, we have
\begin{align*}
	C_1(|a| + |b|)^{\gamma-1}|b-a|
	\le
	|V_\gamma(b) - V_\gamma(a)|
	\le
	C_2(|a| + |b|)^{\gamma-1}|b-a|,
\end{align*}
where 
\begin{equation*}
    C_1
    =
    \left\{
    \begin{array}{ll}
        \gamma, & \mbox{if $\gamma\in(0,1]$,} \\[5pt]
        2^{1-\gamma}, & \mbox{if $\gamma\in[1,\infty)$,}
    \end{array}
    \right.
    \qquad
    C_2
    =
    \left\{
    \begin{array}{ll}
         2^{1-\gamma}, & \mbox{if $\gamma\in(0,1]$,} \\[5pt]
         \gamma, & \mbox{if $\gamma\in[1,\infty)$.}
    \end{array}
    \right.
\end{equation*}
\end{lemma}

\ifnum\full=1

\begin{proof}
Since we are interested in the most simple explicit constants possible, we give the proof for the sake of completeness.
We note that the case $a=0=b$ is trivial and therefore will not be considered in the following. 
We first compute 
\begin{align*}
	|V_\gamma(b) - V_\gamma(a)|
	=
	\bigg|\int_0^1 \frac{\d}{\ds} V_\gamma(sa + (1-s)b) \,\ds \bigg|
    =
    \gamma \int_0^1 |sa + (1-s)b|^{\gamma-1} \,\ds.
\end{align*}
In the following it remains to estimate the integral on the right-hand side from above and below. By symmetry we may assume $|a|\le|b|$. Next, we observe that 
\begin{align*}
    \int_0^1 |sa + (1-s)b|^{\gamma-1} \,\ds
    \ge 
    |b|^{\gamma-1} 
    \ge 
    (|a|+|b|)^{\gamma-1} ,
    \qquad\mbox{if $\gamma\in(0,1]$,}
\end{align*}
and 
\begin{align*}
    \int_0^1 |sa + (1-s)b|^{\gamma-1} \,\ds
    \le 
    |b|^{\gamma-1} 
    \le 
    (|a|+|b|)^{\gamma-1} ,
    \qquad\mbox{if $\gamma\in [1,\infty)$.}
\end{align*}
For the other cases, we first consider the situation where either $0<a\le b$ or $b\le a<0$. If $\gamma\in(0,1]$, we find that
\begin{align*}
    \int_0^1 |sa + (1-s)b|^{\gamma-1} \,\ds
    \le 
    \int_{0}^1 |(1-s)b|^{\gamma-1} \,\ds 
    =
    \frac{1}{\gamma} |b|^{\gamma-1} 
    \le 
    \frac{2^{1-\gamma}}{\gamma} (|a|+|b|)^{\gamma-1},
\end{align*}
while in the case $\gamma\in [1,\infty)$ we have
\begin{align*}
    \int_0^1 |sa + (1-s)b|^{\gamma-1} \,\ds
    \ge 
    \int_{0}^1 |(1-s)b|^{\gamma-1} \,\ds 
    =
    \frac{1}{\gamma} |b|^{\gamma-1} 
    \ge 
    \frac{2^{1-\gamma}}{\gamma} (|a|+|b|)^{\gamma-1}.
\end{align*}
Now, it remains to consider the situation where either $0<a\le b$ or $b\le a<0$ (excluding the case $a=0=b$). 
Letting $s_o:=\frac{b}{b-a}$, we decompose the integral 
\begin{align*}
    \int_0^1 |sa + (1-s)b|^{\gamma-1} \,\ds
    =
    \int_0^{s_o} |sa + (1-s)b|^{\gamma-1} \,\ds +
    \int_{s_o}^1 |sa + (1-s)b|^{\gamma-1} \,\ds.
\end{align*}
For the first integral we compute
\begin{align*}
    \int_0^{s_o} |sa + (1-s)b|^{\gamma-1} \,\ds 
    &=
    s_o \int_0^1 \big|s_o ta + (1- s_o t)b\big|^{\gamma-1} \,\dt \\
    &=
    s_o|b|^{\gamma-1} \int_0^1 (1-t)^{\gamma-1} \,\dt
    =
    \frac{s_o}{\gamma} |b|^{\gamma-1}
\end{align*}
and for the second one
\begin{align*}
    \int_{s_o}^1 |sa + (1-s)b|^{\gamma-1} \,\ds
    &=
    (1-s_o) \int_0^1 \big|(1+(1-s_o)t)a + (1-s_o)tb\big|^{\gamma-1} \,\ds \\
    &=
    (1-s_o) |a|^{\gamma-1} \int_0^1 t^{\gamma-1} \,\dt 
    =
    \frac{1-s_o}{\gamma} |a|^{\gamma-1}.
\end{align*}
Summing up, we find
\begin{align*}
    \int_0^1 |sa + (1-s)b|^{\gamma-1} \,\ds
    =
    \frac{1}{\gamma} \bigg(\frac{|b|^{\gamma-1}b}{b-a} - \frac{|a|^{\gamma-1}a}{b-a}\bigg)
    =
    \frac{|b|^\gamma + |a|^\gamma}{\gamma(|b|+|a|)}.
\end{align*}
In the case $\gamma\in(0,1]$ this yields
\begin{align*}
    \int_0^1 |sa + (1-s)b|^{\gamma-1} \,\ds
    \le
    \frac{2^{1-\gamma}}{\gamma} (|a|+|b|)^{\gamma-1},
\end{align*}
while in the case $\gamma\in [1,\infty)$ we obtain
\begin{align*}
    \int_0^1 |sa + (1-s)b|^{\gamma-1} \,\ds
    \ge
    \frac{2^{1-\gamma}}{\gamma} (|a|+|b|)^{\gamma-1}.
\end{align*}
This finishes the proof.
\end{proof}

\fi

The next algebraic inequality somehow represents the ellipticity contained in the fractional $p$-Laplace equation. In contrast to \cite[Lemma A.6]{Brasco-Lindgren-Schikorra}, we work directly with the expression obtained by testing. More precisely, this means that in the proof of the energy inequality $e$ and $f$ take over the role of $\eta (x+h)$ and $\eta (x)$. The advantage of this algebraic inequality in comparison with \cite[Lemma A.6]{Brasco-Lindgren-Schikorra} is that the essential steps in the estimation of the fractional $p$-Laplacian from below are outsourced  into the algebraic inequality.
\begin{lemma}\label{lem:algebraic-1}
Let $p\ge 2$ and $\delta\ge1$. Then there exists a constant $C=\widetilde{C}(p) 2^\delta$, such that whenever   $a,b,c,d\in \R$ and $e,f\in\R_{\ge 0}$, we have
\begin{align*}
    \big( V_{p-1}(a&-b)- V_{p-1}(c-d)\big)
    \big( V_\delta (a-c)e^2 -V_\delta (b-d)f^2\big)\\
    &\quad\ge 
    \tfrac1{C} \mathbf I
    -C
     \big(|a-b|+|c-d|\big)^{p-2} \big( |a-c|+|b-d|\big)^{\delta+1}|e-f|^2,
\end{align*}
where
\begin{equation*}
    \mathbf I:= \big( |a-b|+|c-d|\big)^{p-2}
    \big( |a-c|+|b-d|\big)^{\delta-1}
    \big| (a-c)-(b-d)\big|^2 \big(e^2+f^2\big).
\end{equation*}
\end{lemma}

\begin{proof}
We first rewrite the left-hand side of the desired inequality in the form
\begin{align*}
    &
    \tfrac12
    \big( V_{p-1}(a-b)- V_{p-1}(c-d)\big) \big( V_\delta (a-c)-V_\delta (b-d)\big)
    \big(e^2+f^2\big)\\
    &\qquad
    +\tfrac12
    \big( V_{p-1}(a-b)- V_{p-1}(c-d)\big) \big( V_\delta (a-c)+V_\delta (b-d)\big)
    \big(e^2-f^2\big).
\end{align*}
The first summand is non-negative due to the monotonicity; cf.~\cite[Lemma A.5]{Brasco-Lindgren-Schikorra}. This means that the multiplicative factors $ V_{p-1}(a-b)- V_{p-1}(c-d)$ and $V_\delta (a-c)-V_\delta (b-d)$ must have the same sign. Without loss of generality, we assume that both are non-negative. Indeed, if they are both negative, we simply switch the order of the summands in the two factors. To each factor we apply Lemma \ref{lem:Acerbi-Fusco} (first with exponent $p-1$ and then with $\delta$) to conclude 
\begin{align*}
    C_1(p-1) 
    \le   
    \frac{V_{p-1}(a-b)- V_{p-1}(c-d)}{\big( |a-b|+|c-d|\big)^{p-2} | (a-c)-(b-d)|}
    \le  
    C_2(p-1) 
\end{align*}
and
\begin{align*}
    C_1(\delta) 
    \le 
    \frac{V_{\delta}(a-c)- V_{\delta}(b-d)}{\big( |a-c|+|b-d|\big)^{\delta-1}| (a-c)-(b-d)|}
    \le 
    C_2(\delta) . 
\end{align*}
This leads to the following lower bound for the above-mentioned first term
\begin{align*}
   \tfrac12 \big( V_{p-1}(a-b)- V_{p-1}(c-d)\big) & \big( V_\delta (a-c)-V_\delta (b-d)\big)
    \big(e^2+f^2\big)\\
    &\qquad\qquad\quad \ge 
    \tfrac12 C_1(p-1)C_1(\delta)\mathbf I.
\end{align*}
Similarly, using the upper bound for $V_{p-1}$  we estimate the second term from above
\begin{align*}
    \tfrac12
    \big( V_{p-1}(a&-b)- V_{p-1}(c-d)\big) \big( V_\delta (a-c)+V_\delta (b-d)\big)
    \big(e^2-f^2\big)\\
    &\le
     \tfrac12 C_2(p-1) \big( |a-b|+|c-d|\big)^{p-2}| (a-c)-(b-d)|\\
     &\phantom{\le\,}\cdot
     \big( |a-c|^\delta +|b-d|^\delta\big)|e-f|(e+f)\\
     &\le
     \tfrac12 C_2(p-1) \big( |a-b|+|c-d|\big)^{p-2}| (a-c)-(b-d)|\\
     &\phantom{\le\,}\cdot
     \big( |a-c| +|b-d|\big)^\delta|e-f|(e+f)\\
      &=: 
        \tfrac12 C_2(p-1) \mathbf{II}.
\end{align*}
To proceed further, we use Young's inequality to estimate
\begin{align*}
    &C_2(p-1)| (a-c)-(b-d)|(e+f)\big( |a-c| +|b-d|\big)|e-f|\\
    &\qquad\le
    \tfrac12\eps | (a-c)-(b-d)|^2 (e+f)^2 +\tfrac{C_2(p-1)^2}{2\eps} 
    \big( |a-c| +|b-d|\big)^2(e-f)^2\\
    &\qquad\le
    \eps | (a-c)-(b-d)|^2\big( e^2 +f^2\big)+\tfrac{C_2(p-1)^2}{2\eps} 
    \big( |a-c| +|b-d|\big)^2(e-f)^2.
\end{align*}
Therefore, we get
\begin{align*}
    & C_2(p-1)\mathbf{II} \\
    &\quad\le
     \big( |a-b|+|c-d|\big)^{p-2}
    \big( |a-c| +|b-d|\big)^{\delta-1}\\
    &\quad\phantom{\le\,}\cdot
    \Big[ \eps | (a-c)-(b-d)|^2\big(e^2+f^2\big) +\tfrac{C_2(p-1)^2}{2\eps} 
    \big( |a-c| +|b-d|\big)^2(e-f)^2\Big]\\
    &\quad=
     \eps\mathbf I+
     \tfrac{C_2(p-1)^2}{2\eps}\big( |a-b|+|c-d|\big)^{p-2}
    \big( |a-c| +|b-d|\big)^{\delta+1}(e-f)^2.
\end{align*}
Joining the preceding estimates finally results in
\begin{align*}
    &\big( V_{p-1}(a-b)- V_{p-1}(c-d)\big)
    \big( V_\delta (a-c)e^2 -V_\delta (b-d)f^2\big)\\
    &\qquad\ge 
    \tfrac12 \big( C_1(p-1)C_1(\delta) - \eps\big) \mathbf I \\
    &\qquad\quad-
    \tfrac{C_2(p-1)^2}{4\eps}\big( |a-b|+|c-d|\big)^{p-2}
    \big( |a-c| +|b-d|\big)^{\delta+1}(e-f)^2.
\end{align*}
Here, we choose $\eps:= \frac12 C_1(p-1)C_1(\delta)$ and obtain 
\begin{align*}
    &\big( V_{p-1}(a-b)- V_{p-1}(c-d)\big)
    \big( V_\delta (a-c)e^2 -V_\delta (b-d)f^2\big)\\
    &\qquad\ge 
     \tfrac{C_1(p-1)C_1(\delta)}{4}\mathbf I  \\
    &\qquad\quad-
    \tfrac{C_2(p-1)^2}{2C_1(p-1)C_1(\delta)}\big( |a-b|+|c-d|\big)^{p-2}
    \big( |a-c| +|b-d|\big)^{\delta+1}(e-f)^2.
\end{align*}
This proves the desired estimate. 
\end{proof}

\begin{lemma}\label{lem:algebraic-2}
Let $\gamma\ge 1$, $A,B\in\R$ and $e,f\in\R_{\ge 0}$. Then we have
\begin{equation*}
     | A-B|^\gamma\big( e^\gamma+f^\gamma\big)
     \ge 
     2^{2-\gamma}| Ae-Bf|^\gamma -2^{1-\gamma}|A+B|^\gamma |e-f|^\gamma.
\end{equation*}
\end{lemma}
\begin{proof}
By using the convexity of $t\mapsto t^\gamma$ twice, we obtain
\begin{align*}
    | Ae-Bf|^\gamma
    &=
    \big|\tfrac12 (A-B)(e+f)+\tfrac12 (A+B)(e-f)\big|^\gamma\\
    &\le
    \tfrac12 |A-B|^\gamma(e+f)^\gamma
    +
    \tfrac12 |A+B|^\gamma|e-f|^\gamma\\
    &\le
    2^{\gamma -2}|A-B|^\gamma\big(e^\gamma+f^\gamma \big)
     +
    \tfrac12 |A+B|^\gamma |e-f|^\gamma.
\end{align*}
Multiplication by $2^{2-\gamma}$ yields the assertion.
\end{proof}

\subsection{Some integral estimates}
The first result ensures that a certain integral exists.
\begin{lemma}\label{int-sing}
Let $0<\beta<N$, and $\Omega\subset\R^N$ measurable with $|\Omega|<\infty$. Then, for any $x\in\R^N$ we have
\begin{equation*}
    \int_\Omega\frac{1}{|x-y|^{N-\beta}}\dy
    \le 
    \frac{\omega_N}{\beta}\Big( \frac{N|\Omega|}{\omega_N}\Big)^\frac{\beta}{N},
\end{equation*}
where $\omega_N:=|S_1^{N-1}|=N|B_1|$ denotes the $(N{-}1)$-dimensional surface measure of the unit sphere in $\R^N$. In the case of a ball $B_R(x_o)=\Omega$ , we have \begin{equation*}
    \int_{B_R(x_o)}\frac{1}{|x-y|^{N-\beta}}\dy
    \le 
    \frac{\omega_N}{\beta} R^\beta.
\end{equation*}
\end{lemma}
\begin{remark}\label{no-int-sing}\upshape
If $\beta\ge N$, i.e. $N-\beta\le 0$, the result is simpler. In fact, 
for $x\in B_R(x_o)$ we have
\begin{align*}
        \int_{B_R(x_o)}\frac{1}{|x-y|^{N-\beta}}\dy
        \le
        \int_{B_{2R}(x)}\frac{1}{|x-y|^{N-\beta}}\dy
        =
        \frac{\omega_N}{\beta} (2R)^{\beta}.
\end{align*}
\end{remark}
The following lemma can be inferred from \cite[Lemma 2.3]{Brasco-Lindgren-Schikorra}. 
\begin{lemma}\label{lem:t}
Let $p\in(1,\infty)$ and $s\in(0,1)$. Then, for any $u\in L^{p-1}_{sp}(\R^N)$, any ball $B_{R}\equiv B_{R}(x_o)\Subset \Omega$ and any $r\in(0,R)$, we have 
\begin{align*}
    \Tail (u; x_o,r)^{p-1}
    \le
    C(N)\Big(\frac{R}{r}\Big)^N
    \big( \Tail (u, R)+ \| u\|_{L^\infty (B_R)}\big)^{p-1}.
\end{align*}
\end{lemma}

\begin{proof}
Applying \cite[Lemma 2.3]{Brasco-Lindgren-Schikorra} on two concentric balls $B_{r}(x_o)\Subset B_{R}(x_o)$, we obtain 
\begin{align*}
    \Tail (u; x_o,r)^{p-1}
    &\le
    \Big(\frac{r}{R}\Big)^{sp} \Tail (u; x_o, R)^{p-1} +
    r^{-N} \int_{B_R(x_o)} |u|^{p-1} \,\dx \\
    &\le
    \Big(\frac{r}{R}\Big)^{sp} \Tail (u; x_o, R)^{p-1} +
    |B_1| \Big(\frac{R}{r}\Big)^N \| u\|_{L^\infty (B_R)}^{p-1}\\
    &\le
    C(N)\Big(\frac{R}{r}\Big)^N
    \Big( \Tail (u, R)^{p-1}+ \| u\|_{L^\infty (B_R)}^{p-1}\Big) \\
    &\le
    C(N)\Big(\frac{R}{r}\Big)^N
    \big( \Tail (u, R)+ \| u\|_{L^\infty (B_R)}\big)^{p-1},
\end{align*}
which proves the claim.
\end{proof}

\subsection{Fractional Sobolev spaces}\label{sec:fractional}
In the following, we summarize some statements concerning fractional Sobolev spaces that are  useful for our purposes. We intend to avoid more function spaces, such as  Nikol’skii  and Besov spaces, 
and to limit ourselves to the essential functional estimates. For further information on this topic, we refer to \cite{Adams, Brasco-Lindgren, Brasco-Lindgren-Schikorra}.
We start with an embedding that ensures that $W^{1,q}$-functions also belong to $W^{\gamma ,q}$ for any $0<\gamma <1$. 
\ifnum\full=0
Since this is well-known, we omit the proof.
\fi
\ifnum\full=1
This is well-known; however, for the convenience of the reader we give a proof.
\fi

\begin{lemma}[Embedding $W^{1,q}\hookrightarrow W^{\gamma,q}$]
\label{lem:FS-S}
Let $q\ge 1$ and $\gamma\in (0,1)$. Then  for any  $w\in W^{1,q}(B_R)$ we have
\begin{align*}
    \iint_{B_R\times B_R} 
    \frac{|w(x)-w(y)|^q}{|x-y|^{N+\gamma q}} \,\dx\dy
    \le 
    8\omega_N \frac{R^{(1-\gamma)q}}{(1-\gamma)q} 
    \int_{B_R} |\nabla w|^q \,\dx. 
\end{align*}
\end{lemma}
\begin{remark}\upshape
For $\gamma\uparrow 1$ the following stability result holds
\begin{align*}
    \lim_{\gamma \uparrow 1} (1-\gamma)
    \iint_{B_R\times B_R}
     \frac{|w(x)-w(y)|^q}{|x-y|^{N+\gamma q}}\,\dx\dy
     &=
     C(N,q)\int_{B_R}|\nabla w|^q\, \dx,
\end{align*}
whenever $w\in W^{1,q}(B_R)$;
see \cite{BBM-2}.
\end{remark}

\ifnum\full=1

\begin{proof}
For fixed $x\in B_R$ we introduce polar coordinates centered at $x$, i.e.~we write $y=x+\rho\omega$ with $\omega\in\R^N$, $|\omega |=1$, and $\rho\in [0, d(x,\omega)]$, where $d(x,\omega)$ denotes the  distance of $x$ to $(x+\omega\R_+)\cap \partial B_R$. In this way we get
\begin{align*}
       \iint_{B_R\times B_R} &
    \frac{|w(x)-w(y)|^q}{|x-y|^{N+\gamma  q}} \,\dx\dy\\
    &=
    \int_{B_R} \int_{|\omega|=1}
    \int_0^{d(x,\omega)}
    \frac{|w(x)-w(x+\rho\omega)|^q}{\rho^{N+\gamma  q}}\rho^{N-1}\drho\d \omega\dx.
\end{align*}
Here, $\d \omega$ denotes the $(N-1)$-dimensional surface measure on $S^{N-1}$.
Next we write
\begin{align*}
    w(x)-w(x+\rho\omega)&=-\int_{0}^\rho \nabla w(x+s\omega)\cdot\omega\, \ds,
\end{align*}
so that
\begin{align*}
    |w(x)-w(x+\rho\omega)|^q
    &\le 
    \int_{0}^\rho |\nabla w(x+s\omega)|^q\ds \rho^{q-1}.
\end{align*}
Inserting this above and using Fubini's theorem  we get
\begin{align}\label{in-polar-coord}\nonumber
    \mathbf{I}
    &=
    \int_{B_R}\int_{B_R}
    \frac{|w(x)-w(y)|^q}{|x-y|^{N+\gamma  q}} \,\dx\dy\\\nonumber
    &=
    \int_{B_R} \int_{|\omega|=1}
    \int_0^{d(x,\omega)}
    \frac{ \int_{0}^{q}|\nabla w(x+s\omega)|^q\ds \rho^{q-1}}{\rho^{N+\gamma  q}}\rho^{N-1}\drho \d\omega\dx\\\nonumber
     &=
    \int_{B_R} \int_{|\omega|=1}
    \int_0^{d(x,\omega)}\int_0^\rho |\nabla w(x+s\omega)|^q\ds\, \rho^{(1-\gamma )q-2}\d\omega
    \drho\dx\\\nonumber
    &=
    \int_{B_R} \int_{|\omega|=1}
    \int_0^{d(x,\omega)}
    \int_s^{d(x,\omega)}|\nabla w(x+s\omega)|^q \rho^{(1-\gamma )q-2}\drho\ds
    \d\omega
    \drho\dx\\
    &=
    \int_{B_R} \int_{|\omega|=1}
    \int_0^{d(x,\omega)}
    \frac{d(x,\omega)^{(1-\gamma )q}- s^{(1-\gamma )q}}{(1-\gamma )q-1}
    |\nabla w(x+s\omega)|^q
    \, \ds\d\omega\dx.
\end{align}
We now distinguish cases. First, we consider the case $(1-\gamma )q<\frac12$. Then, 
\begin{align*}
    \mathbf I
    &\le
    \frac1{1-(1-\gamma )q}
    \int_{B_R} \int_{|\omega|=1}\int_0^{d(x,\omega)}
    |\nabla w(x+s\omega)|^qs^{(1-\gamma )q}\,
    \ds\d\omega\dx\\
    &=
    \frac1{1-(1-\gamma )q}
    \int_{B_R} \int_{|\omega|=1}\int_0^{d(x,\omega)}
    \frac{|\nabla w(x+s\omega)|^q}{s^{N-(1-\gamma )q}}
    s^{N-1}\,
    \ds\d\omega\dx\\
    &=
    \frac1{1-(1-\gamma )q}
    \int_{B_R}\int_{B_R}
    \frac{|\nabla w(y)|^q}{|x-y|^{N-(1-\gamma )q}}\dy\dx\\
    &\le
   \frac{2\omega_NR^{(1-\gamma )q}}{(1-\gamma )q}\int_{B_R}|\nabla w|^q\dx.
\end{align*}
To obtain the last line, we first used Fubini's theorem to interchange the order of integrations and then applied Lemma \ref{int-sing} with $\beta=(1-\gamma)q$. 

Next, we will deal with the case $(1-\gamma )q\in [\frac12,\frac32]$. Here we start with the expression for $I$ from the second last line of \eqref{in-polar-coord}. We write the exponent of $s$, i.e.~$(1-\gamma)q-2$, in the form $-(N-\alpha) +(1-\beta) +(N-1)$, for some $0<beta<1$, which will be specified later. Then, we estimate $s^{1-\beta}$ by $\rho^{1-\beta}$ and interchange the order of integration with respect to $\rho$ and $\omega$. Finally, we apply again Lemma \ref{int-sing} and compute the $\rho$-integral explicitly.  This way we get
\begin{align*}
    \mathbf I
    &=
    \int_{B_R} \int_{|\omega|=1}
    \int_0^{d(x,\omega)}\int_0^\rho
    |\nabla w(x+s\omega)|^q \ds\rho^{(1-\gamma )q-2}\drho
    \d\omega\dx\\
    &=
     \int_{B_R} \int_{|\omega|=1}
    \int_0^{d(x,\omega)}\int_0^\rho
    \frac{|\nabla w(x+s\omega)|^q}{s^{N-\beta}} s^{1-\beta} s^{N-1}\ds\rho^{(1-\gamma )q-2}\drho
    \d\omega\dx\\
    &\le
    \int_{B_R} \int_{|\omega|=1}
    \int_0^{d(x,\omega)}\int_0^\rho
    \frac{|\nabla w(x+s\omega)|^q}{s^{N-\beta}}  s^{N-1}\ds\rho^{(1-\gamma )q-1-\beta}\drho
    \d\omega\dx\\
    &\le
    \int_{B_R}\int_{B_R}\frac{|\nabla w(y)|^q}{|x-y|^{N-\beta}}
    \dy\dx\int_0^{d(x,\omega)}
    \rho^{(1-\gamma )q-1-\beta}\drho\\
    &=
    \frac{\omega_NR^{(1-\gamma )q}}{\beta ((1-\gamma )q-\beta)}
  \int_{B_R}|\nabla w|^q\dy
\end{align*}
Choosing $\beta =\frac12 (1-\gamma )q\in [\frac14,\frac34]$ we get
\begin{align*}
    \mathbf I
    &\le
    \frac{8\omega_N R^{(1-\gamma )q}}{(1-\gamma )q}\int_{B_R}|\nabla w|^q\dy
\end{align*}
Finally, we consider the remaining case $(1-\gamma )q\ge \frac32$.  Starting from the last line of \eqref{in-polar-coord}, we first neglect the negative contribution of $-s^{}$, and then write
$1=s^{-(N-1}s^{N-1}$.  Then we estimate $d(x,\omega)$ by $R$ and interchange by Fubini's theorem the order of integration between $s$ 
and $\omega$, so that we can apply Lemma \ref{int-sing} with $\beta =1$.
This gives
\begin{align*}
    \mathbf I
    &\le
     \int_{B_R} \int_{|\omega|=1}  \frac{d(x,\omega)^{(1-\gamma )q-1}}{(1-\gamma )q-1}
     \int_0^{d(x,\omega)}\frac{|\nabla w(x+s\omega)|^q}{s^{N-1}}
     s^{N-1}\ds\d \omega\dx\\
     &\le
      \frac{ R^{(1-\gamma )q-1}}{(1-\gamma )q-1} 
      \int_{B_R} \int_{B_R}
      \frac{|\nabla w(y)|^q}{|x-y|^{N-1}}\dy\dx\\
      &\le
      \frac{ \omega_N R^{(1-\gamma )q}}{(1-\gamma )q-1} 
      \int_{B_R}|\nabla w|^q\dy\\
      &\le
      \frac{ 3 \omega_N R^{(1-\gamma )q}}{(1-\gamma )q} 
      \int_{B_R}|\nabla w|^q\dy
\end{align*}
To obtain the last line we used $(1-\gamma)q-1\ge \frac13 (1-\gamma)q$.
Joining the three cases we get
$$
      \iint_{B_R\times B_R} 
    \frac{|w(x)-w(y)|^q}{|x-y|^{N+\gamma  q}} \,\dx\dy
    \le
    \frac{ 8 \omega_N R^{(1-\gamma )q}}{(1-\gamma )q} 
      \int_{B_R}|\nabla w(y)|^q\dy,
$$
which proves the claim.
\end{proof}

\fi

Next, we provide a fractional Sobolev-Poincar\'e inequality, which can be retrieved from \cite[Theorem 6.7]{Hitchhikers-guide}. 
\begin{lemma}[Fractional Sobolev-Poincar\'e inequality]\label{lem:frac-Sob-2}
Let $q\ge 1$, $\gamma\in (0,1)$, such that $\gamma q<N$.
Then, for any  $w\in W^{\gamma,q}(B_R)$ we have 
\begin{equation*}
    \bigg[\mint_{B_R} \big| w-(w)_{B_R}\big|^\frac{Nq}{N-q\gamma}\,\dx
    \bigg]^\frac{N-q\gamma}{Nq}
    \le
     CR^\gamma \bigg[\int_{B_R}\mint_{B_R}
    \frac{|w(x)-w(y)|^q}{|x-y|^{N+q\gamma}} \,\dx\dy\bigg]^\frac1q,
\end{equation*}
where  $C=C(N,q,\gamma)$.\hfill $\Box$
\end{lemma}

\begin{remark}\upshape
To trace the stability as $\gamma\uparrow 1$, the precise dependence of $C$ in terms of $\gamma$ is crucial. From \cite[Theorem 1]{BBM-1}, we have for $\gamma\in [\frac12 ,1)$ that
$$
    C(N,q,\gamma)= \bigg[\frac{C(N)(1-\gamma)}{(N-\gamma q)^{q-1}}\bigg]^\frac1q.
$$
\end{remark}

\begin{remark}\upshape 
The above Fractional Sobolev-Poincar\'e inequality easily yields a variant without mean value in the left-hand side integral. In fact, under the same assumptions as in Lemma \ref{lem:frac-Sob-2}, we have 
\begin{align*}
    \bigg[\mint_{B_R} & |w|^\frac{Nq}{N-q\gamma} \,\dx
    \bigg]^\frac{N-q\gamma}{Nq}
    \!\!\le
    2^{q-1}\bigg[ C^q R^{q\gamma}  \int_{B_R}\mint_{B_R}
    \frac{|w(x)-w(y)|^q}{|x-y|^{N+q\gamma}} \dx\dy +
    \mint_{B_R} |w|^q \dx\bigg]
\end{align*}
for any $w\in W^{\gamma,q}(B_R)$,
where $C=C(N,q,\gamma)$ is the constant from Lemma \ref{lem:frac-Sob-2}.
\hfill $\Box$
\end{remark}

\subsection{Finite differences and fractional Sobolev spaces}
For an open set $\Omega\subset \R^N$, and a direction vector $h\in\R^N$, define $\Omega_{h}:=\{ x\in\Omega\colon x+h\in\Omega\}$. For measurable $w\colon \Omega\to\R$,  we denote by $\btau_h\colon L^1(\Omega)\to L^1(\Omega_{h}) $ the finite difference operator
\begin{equation*}
    \btau_hw(x):=w(x+h)-w(x),
\end{equation*}
whenever $h\in\R^N$ and $x\in \Omega_h$. If the direction is a fixed  unit  vector
$e\in\R^N$, we write 
$$
    \btau_h^{(e)}w(x):=w(x+he)-w(x),
$$ 
where $h\in\R$ now is a real number. If $e$ is a canonical basis vector $e_i$ we write $\btau_h^{(i)}w$ instead of $ \btau_h^{(e_i)}w$. At several points we will use two elementary properties of finite differences. These are summarized in \cite[Lemma 7.23 \& 7.24]{Gilbarg-Trudinger}
and \cite[Chap.~5.8.2]{Evans:book}. 

\begin{lemma}\label{lem:diff-quot-1}
Let $1<q<\infty$, $M>0$, and $0<d<R$.  Then, any $w\in L^q(B_R)$ that satisfies 
\begin{equation}\label{bound-tauh-w}
    \int_{B_{R-d}} \big|\btau_h^{(i)}w\big|^q\,\dx
    \le
    M^q|h|^q \quad\mbox{for any $0<|h|\le d$}
\end{equation}
is weakly differentiable in direction $x_i$ on $B_{R-d}$, and moreover
$$
    \int_{B_{R-d}} |D_iu|^q\,\dx\le M^q.
$$
If $w$ satisfies \eqref{bound-tauh-w} for any direction $e_i$, then $w\in W^{1,q}(B_{R-d})$.\hfill$\Box$
\end{lemma}

\begin{lemma}\label{lem:diff-quot-2}
Let  $1\le q<\infty$ and $0<d<R$. Then, for any 
$w\in W^{1,q}(B_R)$, any $i\in\{1,\dots ,N\}$, and any $0<|h|\le d$, we have
$$
    \big\| \btau_h^{(i)} w\big\|_{L^q(B_{R-d})} 
    \le 
    |h|\,\| D_i w\|_{L^q(B_R)}.
$$
Moreover, we have
$$
    \lim\limits_{|h|\rightarrow 0} 
    \Big\|
    \frac{\btau^{(i)}_h w}{h} - D_i w\Big\|_{L^q(B_{R-d})} =0.
$$ 
for any direction $e_i$.\hfill $\Box$
\end{lemma}

In the context of fractional Sobolev spaces, $W^{\gamma ,q}$-functions 
fulfill an estimate  for finite differences that is similar to the one from Lemma \ref{lem:diff-quot-2}; see \cite[Proposition\,2.6]{Brasco-Lindgren}.
\begin{lemma}\label{lem:N-FS}
Let $q\in(1,\infty)$, $\gamma \in(0,1)$, and $0<d<R$. Then, there exists a constant $C=C(N,q)$ such that for any $w\in W^{\gamma,q}(B_{R})$, we have
\begin{align*}
    \int_{B_{R-d}} |\btau_h w|^q \,\dx 
    &\le
    C\,|h|^{\gamma q} 
    \Bigg[(1-\gamma)[w]^q_{W^{\gamma,q}(B_{R})} +
    \bigg(\frac{R^{(1-\gamma )q}}{d^q} + \frac{1}{\gamma d^{\gamma q}}\bigg)\|w\|^q_{L^q(B_{R})}
    \Bigg]
\end{align*}
for any  $h\in\R^N\setminus\{0\}$ that satisfies  $|h|\le d$.
\end{lemma}

Finite differences can also be used to identify mappings with certain quantitative properties, such as belonging to a fractional Sobolev space. Such a result can be interpreted as the fractional analogue of Lemma \ref{lem:diff-quot-1}. 
We refer to \cite[7.73]{Adams}. The version given here is from \cite[Lemma 3.1]{DeFilippis-Mingione-Invent}. 
\begin{lemma}\label{lem:FS-N}
Let $q\in [1,\infty)$, $\gamma \in(0,1]$, $M\ge 1$, and $0<d<R$. Then, there exists a constant $C=C(N,q)$ such that whenever $w\in L^{q}(B_{R+d})$ satisfies 
\begin{align*}
    \int_{B_R} |\btau_h w|^q \,\dx 
    \le 
    M^q|h|^{\gamma q} 
    \qquad \mbox{for any $h\in\R^N\setminus\{0\}$ with $|h|\le d$,}
\end{align*}
then $w\in W^{\beta ,q}(B_R)$ whenever $0<\beta <\gamma$. Moreover, we have
\begin{align*}
      \iint_{B_R\times B_R}  
    \frac{|w(x)-w(y)|^q}{|x-y|^{N+\beta q}} \,\dx\dy
    \le 
    C\bigg[\frac{d^{(\gamma-\beta)q}}{\gamma-\beta} M^q +
    \frac{1}{\beta d^{\beta q}} \|w\|_{L^q(B_R)}^q\bigg].
\end{align*}
\end{lemma}

\ifnum\full=1
\begin{proof} 
Decompose the domain of integration $K_R=B_R\times B_R$ into 
$K_R\cap\{|x-y|\le d\}$ and $K_R\cap\{|x-y|> d\}$. Denote the two resulting integrals with $\mathbf I$  and $\mathbf{II}$.
\begin{align*}
        \iint_{K_R}\frac{|w(x)-w(y)|^q}{|x-y|^{N+\beta q}}\,\dx\dy
        &=
        \iint_{K_R\cap\{x-y|\le d|\}}\dots\,\dx\dy
        +
         \iint_{K_R\cap\{x-y|> d|\}}\dots\,\dx\dy\\
         &=:\mathbf I +\mathbf{II}.
\end{align*}
In $\mathbf I$ substitute $y-x=h$ and enlarge the domain of integration with respect to the $h$ variables from $B_R(-x)\cap\{|h|\le d\}$ to $\{|h|\le d\}$. The resulting integral can easily be calculated by introducing spherical coordinates. This results in
\begin{align*}
    \mathbf{I}
    &=
    \int_{B_R}\int_{B_R(-x)\cap\{|h|\le d\}}
    \frac{|\btau_h w(x)|^q}{|h|^{N+\beta q}}\,\d h\dx\\
    &\le
    \int_{B_R}\int_{\{|h|\le d\}}
    \frac{|\btau_h w(x)|^q}{|h|^{N+\beta q}}\,\d h\dx\\
    &=
    \int_{\{|h|\le d\}}\int_{B_R}
    \frac{|\btau_h w(x)|^q}{|h|^{N+\beta q}}\,\dx\d h\\
    &\le
    M^q  \int_{\{|h|\le d\}}|h|^{-N+(\gamma -\beta)q}\d h\\
    &=
    N|B_1| M^q\int_{0}^d s^{-1+(\gamma -\beta)q}\,\ds\\
    &=
    \frac{ N|B_1|}{(\gamma -\beta)q}d^{(\gamma -\beta)q}.
\end{align*}
The numerator in the integral $\mathbf{II}$ is first estimated by $2^{q-1} (|w(x)|^q+|w(y)|^q)$. Then we use the symmetry of the integrand with respect to  $x$ and $y$. In the last step, we introduce polar coordinates to  compute the second integral. This yields
\begin{align*}
    \mathbf{II}
    &\le
    2^{q-1} \iint_{K_R\cap\{x-y|> d|\}}\frac{|w(x)|^q+|w(y)|^q}{|x-y|^{N+\beta q}}\,\dx\dy\\
    &=
     2^{q} \iint_{K_R\cap\{x-y|> d|\}}\frac{|w(x)|^q}{|x-y|^{N+\beta q}}\,\dx\dy\\
     &=
      2^{q}\int_{B_R}|w(x)|^q\int_{B_R\setminus B_d(x)}\frac{1}{|x-y|^{N+\beta q}}\,\dx\\
      &\le
      2^{q}\int_{B_R}|w(x)|^q\int_{\R^N\setminus B_d(x)}\frac{1}{|x-y|^{N+\beta q}}\,\dx\\
      &\le
      \frac{2^{q} N|B_1|}{\beta qd^{\beta q}} d^{-\beta q}\int_{B_R}|w(x)|^q\,\dx.
\end{align*}
This proves the claim.
\end{proof}
\fi
In the following we will introduce two lemmata that will help us to deal with  second order differences in the fractional context. 
The results can essentially be retrieved  from \cite[Chapter 5]{Stein}; see also \cite[Proposition 2.4]{Brasco-Lindgren}. The proof is based on the thermic extension characterization of Besov spaces. A different proof of Lemma~\ref{lem:Domokos} can be found in \cite[Lemma 2.2.1]{Domokos-1} (see also \cite[Theorem 1.1]{Domokos-2})
in the context of the local $p$-Laplace equation on the Heisenberg group and where it serves to  handle the horizontal derivative of weak solutions. 
\ifnum\full=1
Since this argument does not seem to be widely known, we give the  proof for the sake of completeness.  
\fi

\begin{lemma} \label{lem:Domokos}
Let $q\in [1,\infty)$, $\gamma>0$, $M\ge 0$, $0<r<R$, and $0<d\le \tfrac1{2}(R-r)$. Then, there exists a constant $C=C(q)$ such that whenever  $w\in L^q(B_R)$ satisfies
\begin{align}\label{ass:Domokos}
    \int_{B_r} |\btau_h(\btau_h w)|^q \,\dx 
    \le 
    M^q|h|^{\gamma q},
    \qquad \mbox{for any $h\in\R^N\setminus\{0\}$ with $|h|\le d$,}
\end{align}
then in the  case $\gamma\in (0,1)$ we have for any $ 0<|h|\le \tfrac12 d$ that
\begin{align}\label{est-1st-diffquot<1}
     \int_{B_r} |\btau_hw|^q  \,\dx
     &\le
     C(q) |h|^{q\gamma}
     \Bigg[\Big( \frac{M}{1-\gamma}\Big)^q +\frac{1}{d^{q\gamma }}
     \int_{B_R}|w|^q\,\dx
     \Bigg]
     ,
\end{align}
while in the case $\gamma >1$ there holds
\begin{align}\label{est-1st-diffquot>1}  
     \int_{B_r} |\btau_hw|^q  \,\dx
     &\le
     C(q)|h|^q\bigg[
      \Big(\frac{M }{\gamma -1}\Big)^qd^{(\gamma-1)q}
      +
        \frac1{d^q}
        \int_{B_R}|w|^q\,\dx
      \bigg].
\end{align}
In the limiting case $\gamma =1$  we have for any $0<\beta<1$ that
\begin{align}\label{est-1st-diffquot=1}
     \int_{B_r} |\btau_hw|^q  \,\dx
     &\le
      C(q) |h|^{q\beta}
     \bigg[
     \Big(\frac{ M}{1-\beta}\Big)^q d^{(1-\beta)q}+\frac{1}{d^{\beta q}} \int_{B_R}|w|^q\,\dx\bigg].
\end{align}
\end{lemma}

\ifnum\full=1
\begin{proof}
Instead of assumption \eqref{ass:Domokos} we use the equivalent assumption
\begin{align*}
    \int_{B_r} \big|\btau_h^{(e)}(\btau_h^{(e)} w)\big|^q \,\dx 
    \le 
    M^qh^{\gamma  q},
    \qquad \mbox{for any $0<h\le d $ and any $|e|=1$.}
\end{align*}
Since
$$
    \btau_h^{(e)}(\btau_h^{(e)} w) =\btau_{2h}^{(e)}w-2\btau_h^{(e)}w.
$$
we have
\begin{align*}
    \bigg[\int_{B_r} \big|\btau_{2h}^{(e)}w-2\btau_h^{(e)}w \big|^q \,\dx\bigg]^\frac1q
    \le 
    Mh^\gamma 
    \qquad \mbox{for any $0<h\le d $ and any $|e|=1$.}
\end{align*}
For $t\in [\tfrac12 d ,d )$ we replace $h$ by $\frac{t}{2^j}$ with $j\in\N$ which yields
\begin{align*}
    \bigg[\int_{B_r} \Big|\btau_\frac{t}{2^{j-1}}^{(e)}w-2\btau_\frac{t}{2^{j}}^{(e)}w \Big|^q \,\dx\bigg]^\frac1q
    &\le 
    M \Big(\frac{t}{2^j}\Big)^\gamma 
    =
    \frac{M t^\gamma }{2^\gamma 2^{\gamma  (j-1)}} .
\end{align*}
For $k\in\N$ we therefore get
\begin{align}\label{domokos-pre}\nonumber     
    \bigg[
     \int_{B_r} \Big|\btau_t^{(e)}w&-2^k\btau_\frac{t}{2^{k}}^{(e)}w \Big|^q \,\dx\bigg]^\frac1q\\\nonumber
     &=
     \Bigg[
     \int_{B_r} \bigg| \sum_{j=1}^k 2^{j-1} \Big( \btau_\frac{t}{2^{j-1}}^{(e)}w-
     2\btau_\frac{t}{2^{j}}^{(e)}w\Big)\bigg|^q \,\dx\Bigg]^\frac1q\\\nonumber
     &\le
     \sum_{j=1}^k 2^{j-1}\Bigg[\int_{B_r}
     \Big|  \btau_\frac{t}{2^{j-1}}^{(e)}w-
     2\btau_\frac{t}{2^{j}}^{(e)}w\Big|^q \,\dx\Bigg]^\frac1q\\
     &\le
     \frac{M t^\gamma }{2^\gamma } \sum_{j=1}^k \frac{2^{j-1}}{2^{\gamma  (j-1)}}
     =
     \frac{M t^\gamma }{2^\gamma } \sum_{j=0}^{k-1} 2^{j(1-\gamma )}.
\end{align}
Arrived at this stage we consider first the case $0<\gamma <1$ and compute the sum on the right-hand side of \eqref{domokos-pre}. We obtain
\begin{align*}
        \bigg[
     \int_{B_r} \Big|\btau_t^{(e)}w-2^k\btau_\frac{t}{2^{k}}^{(e)}w \Big|^q \,\dx\bigg]^\frac1q
    \leq
     \frac{M t^\gamma }{2^\gamma }\frac{2^{(1-\gamma )k}-1}{2^{1-\gamma }-1}
    =
     Mt^\gamma \frac{2^{(1-\gamma )k}-1}{2-2^\gamma }.
\end{align*}
The previous estimate allows us to bound the $L^q(B_r)$-norm of $\btau_\frac{t}{2^{k}}^{(e)}w$. Indeed, we have
\begin{align*}
    \bigg[
     \int_{B_r} &\Big|\btau_\frac{t}{2^{k}}^{(e)}w\Big|^q  \,\dx\bigg]^\frac1q
     =
     2^{-k}
     \Bigg[
     \int_{B_r} \Big|2^k \btau_\frac{t}{2^{k}}^{(e)}w\Big|^q  \,\dx\Bigg]^\frac1q\\
     &\le
     2^{-k}
     \bigg[
     \int_{B_r} \Big|\btau_t^{(e)}w-2^k\btau_\frac{t}{2^{k}}^{(e)}w \Big|^q \,\dx\bigg]^\frac1q
     +
      2^{-k}
     \bigg[
     \int_{B_r} \big|\btau_t^{(e)}w\big|^q \,\dx\bigg]^\frac1q\\
     &\le
     \underbrace{
     2^{-k} Mt^\gamma \frac{2^{(1-\gamma )k}-1}{2-2^\gamma }}_{\le \frac{M}{2-2^\gamma }(\frac{t}{2^k})^\gamma } +2^{-(k+1)}\bigg[ \int_{B_R}|w|^q\,\dx\bigg]^\frac1q\\
     &\le
     \Bigg[\frac{M}{2-2^\gamma } +\frac{4}{d ^\gamma  }
     \bigg[ \int_{B_R}|w|^q\,\dx\bigg]^\frac1q
     \Bigg]
     \Big(\frac{t}{2^k}\Big)^\gamma .
\end{align*}
In the transition from the penultimate to the last line, we  used 
\begin{align*}
    \Big(\frac{t}{2^k}\Big)^\gamma 
    &=
    \Big(\frac{2t}{d }\cdot\frac{d }{2^{k+1}}\Big)^\gamma 
    \ge 
     \Big(\frac{d }{2^{k+1}}\Big)^\gamma 
     \ge
     \frac{d ^\gamma  }{2^{k+1}}.
\end{align*}
Note that $\frac{2t}{d }\ge 1$.
Now, we consider $h\in (0 ,\frac12 d ]$. 
Then there exist $t\in [\frac12 d ,d )$ and $k\in\N$ such that $h=\frac{t}{2^k}$. This allows us to replace in the inequality above $\frac{t}{2^k}$ by $h$. We obtain that
\begin{align*}
      \bigg[
     \int_{B_r} \big|\btau_h^{(e)}w\big|^q  \,\dx\bigg]^\frac1q
     &\le
     \Bigg[\frac{M}{2-2^\gamma } +\frac{4}{d ^\gamma  }
     \bigg[ \int_{B_R}|w|^q\,\dx\bigg]^\frac1q
     \Bigg]
     h^\gamma ,
\end{align*}
which itself implies
\begin{align*}
     \int_{B_r} \big|\btau_h^{(e)}w\big|^q  \,\dx
     &\le
     2^{q+1}
     \bigg[\Big( \frac{M}{1-\gamma }\Big)^q +\frac{1}{d ^{q\gamma } }
     \int_{B_R}|w|^q\,\dx
     \bigg]
     h^{q\gamma }.
\end{align*}
For $h\in (\tfrac12 d ,d [$ we trivially have
\begin{align*}
     \int_{B_r} \big|\btau_h^{(e)}w\big|^q  \,\dx
     &\le 
     2^{q-1} \int_{B_R} |w|^q  \,\dx
     \le
     \frac{2^{2q-1}}{d ^{q\gamma }} \int_{B_R} |w|^q  \,\dx.
\end{align*}
This proves the claim \eqref{est-1st-diffquot<1}.
Next we consider the case $\gamma  >1$. In this case the series in \eqref{domokos-pre} converges as $k\to\infty$ and we obtain
\begin{align*}  
    \bigg[
     \int_{B_r} \Big|\btau_t^{(e)}w-2^k\btau_\frac{t}{2^{k}}^{(e)}w \Big|^q \,\dx\bigg]^\frac1q
     &\le
     \frac{M t^\gamma }{2^\gamma } \sum_{j=0}^{\infty} 2^{j(1-\gamma )}
     \le
     \frac{M t^\gamma }{2^\gamma  (1-2^{1-\gamma })},
\end{align*}
which leads us to the $L^q(B_r)$-bound
\begin{align*}
    \bigg[
     \int_{B_r} \Big|\btau_\frac{t}{2^{k}}^{(e)}w\Big|^q  \,\dx\bigg]^\frac1q
     &\le
     2^{-k}
     \frac{Mt^\gamma  }{2^\gamma (1-2^{1-\gamma })} +2^{1-k}
     \bigg[ \int_{B_R}|w|^q\,\dx\bigg]^\frac1q.
\end{align*}
As in the case $0<\gamma  <1$ we choose $h\in (0,\frac12 d ]$ in the form $h=\frac{t}{2^k}$ with $t\in [\frac12 d ,d )$ and $k\in\N$ to get
\begin{align*}
    \bigg[
     \int_{B_r} \big|\btau_h^{(e)}w\big|^q  \,\dx\bigg]^\frac1q
     &\le
     \frac{h}{t}
     \frac{Mt^\gamma  }{2^\gamma (1-2^{1-\gamma })} +\frac{2h}{t}
     \bigg[ \int_{B_R}|w|^q\,\dx\bigg]^\frac1q\\
     &\le
     \frac{2h}{d }
     \frac{Md ^\gamma  }{2^\gamma (1-2^{1-\gamma })} +\frac{4h}{d }
     \bigg[ \int_{B_R}|w|^q\,\dx\bigg]^\frac1q\\
     &\le
     \frac{h}{d }\Bigg[
      \frac{Md ^\gamma  }{2^{\gamma -1}-1}
      +
      4\bigg[ \int_{B_R}|w|^q\,\dx\bigg]^\frac1q
      \bigg].
\end{align*}
Taking both sides to the power $q$ we have
\begin{align*}   
     \int_{B_r} \big|\btau_h^{(e)}w\big|^q  \,\dx
     &\le
     2^{3q}\Big(\frac{h}{d }\Big)^q\bigg[
      \Big(\frac{Md ^\gamma  }{2^{\gamma -1}-1}\Big)^q
      +
        \int_{B_R}|w|^q\,\dx
      \bigg],
\end{align*}
which easily implies \eqref{est-1st-diffquot>1}. Finally, we treat the limiting case $\gamma  =1$. In this case from \eqref{domokos-pre} we get
\begin{align*}  
    \bigg[
     \int_{B_r} \Big|\btau_t^{(e)}w-2^k2\btau_\frac{t}{2^{k}}^{(e)}w \Big|^q \,\dx\bigg]^\frac1q
     &\le
     k \frac{M t}{2}, 
\end{align*}
which implies  the $L^q(B_r)$-bound
\begin{align*}
    \Bigg[
     \int_{B_r} \Big|\btau_\frac{t}{2^{k}}^{(e)}w\Big|^q  \,\dx\Bigg]^\frac1q
     &\le
    \frac{k}{2^k}\frac{Mt }{2} +2^{1-k}
     \bigg[ \int_{B_R}|w|^q\,\dx\bigg]^\frac1q.
\end{align*}
Again we choose $h\in (0,\frac12 d ]$ in the form $h=\frac{t}{2^k}$ with $t\in [\frac12 d ,d )$ and $k\in\N$. For $k$ we have
$$
 k=\frac{\ln \tfrac{t}{h}}{\ln 2}\le\frac{\ln\big(\frac{d }{h}\big)}{\ln2}
$$
Inserting this yields
\begin{align*}
    \bigg[
     \int_{B_r} \big|\btau_h^{(e)}w\big|^q  \,\dx\bigg]^\frac1q
     &\le
     \frac{Mh}{2\ln2} \ln\Big(\frac{d }{h}\Big)
     +\frac{4h}{d }
     \bigg[ \int_{B_R}|w|^q\,\dx\bigg]^\frac1q\\
     &\le
     \frac{ Md }{2\ln 2}\frac{h}{d } \ln\Big(\frac{d }{h}\Big)
     +\frac{4h}{d }\bigg[ \int_{B_R}|w|^q\,\dx\bigg]^\frac1q
\end{align*}
To proceed further we consider  for $\beta\in (0,1)$ the function $(0,1]\ni t\mapsto f(t):=t^{1-\beta}\ln \tfrac1{t}$. Then, $\lim_{t\downarrow 0}f(t)=0$. Moreover, we have $f'(t)=-t^{-\beta}\big[ (1-\beta)\ln t+1\big]$ and $f'(t)=0$ for 
$t_o=\exp\big[-\frac{1}{1-\beta}\big]$ with
$f(t_o)=\frac{1}{\mathrm e(1-\beta)}$. Therefore, we have 
$$
    f(t)\le \frac1{\mathrm e(1-\beta)}\quad \forall\, t\in (0,1].
$$
Inserting this above we have
\begin{align*}
    \bigg[
     \int_{B_r} \big|\btau_h^{(e)}w\big|^q  \,\dx\bigg]^\frac1q
     &\le
     \frac{ Md }{2\mathrm e\ln 2(1-\beta)} \Big(\frac{h}{d }\Big)^\beta
     +\frac{4h}{d }\bigg[ \int_{B_R}|w|^q\,\dx\bigg]^\frac1q\\
     &\le
     \Big(\frac{h}{d }\Big)^\beta\bigg[
     \frac{ Md }{2\mathrm e\ln 2(1-\beta)} +\bigg[ \int_{B_R}|w|^q\,\dx\bigg]^\frac1q
     \bigg]\\
     &\le
      h^\beta
     \bigg[
     \frac{ Md ^{(1-\beta)}}{2\mathrm e\ln2(1-\beta)} +\frac{4}{d ^\beta}\bigg[ \int_{B_R}|w|^q\,\dx\bigg]^\frac1q\bigg],
\end{align*}
proving the claim \eqref{est-1st-diffquot=1}.
\end{proof}
\fi
Lemma \ref{lem:Domokos} guarantees in the case $\gamma >1$ that functions $w$ satisfying \eqref{ass:Domokos}
are indeed weakly differentiable. However, from the assumption
on the second differences we loose an
amount of the order $|h|^{\gamma -1}$; roughly speaking, we pass from
second order differences  that are measured in terms of $|h|^\gamma$ to first order  differences that are quantified in terms of $|h|$, which means that we control $\nabla w$.
The loss is of order $|h|^{\gamma -1}$. The $|h|^{\gamma -1}$-part controls the oscillations of $\nabla w$ in certain respects. Therefore, 
this part can be used to show that the gradient $\nabla w$ is fractionally differentiable. This is basically the contend of the next lemma. The version presented here can be inferred from 
\cite[Lemma 2.9]{Diening-Kim-Lee-Nowak} using the quasi-minimality of the mean value for the mapping $\xi\mapsto
\int_{B_R}|w-\xi|^q\,\dx$; see also \cite[Proposition 2.4]{Brasco-Lindgren} and \cite[Lemma 2.6]{Brasco-Lindgren-Schikorra}.

\begin{lemma}\label{lem:2nd-Ni-FS}
Let $q\in [1,\infty)$, $\gamma \in (0,1)$, $M>0$, $R>0$, and $d \in (0,R)$. Then, for any 
 $w\in W^{1,q}(B_{R+6d })$ that satisfies
\begin{equation}\label{ass:W^beta,q-second-diff}
    \int_{B_{R+4d }}\big|\btau_h(\btau_h w)\big|^q\,\dx
    \le
    M^q|h|^{q (1+\gamma )}\qquad \forall\, 0<|h|\in (0,d ],
\end{equation}
we have 
$$
\nabla u\in W^{\beta,q}(B_R)\quad \mbox{for any $\beta\in (0,\gamma )$.}
$$
Moreover, there  exists a constant $C$ depending only on $N$ and $q$,
such that 
\begin{align*}
     [\nabla w]_{W^{\beta,q}(B_R)}^q
    &\le
    \frac{Cd ^{q(\gamma -\beta)}}{(\gamma -\beta) \gamma ^q (1-\gamma )^q}
    \bigg[ M^q + \frac{(R+4d )^{q}}{\beta d ^{q(1+\gamma )}} 
    \int_{B_{R+4d }}|\nabla w|^q\,\dx
    \bigg].
\end{align*}
\end{lemma}

For $\gm q>N$ we recall from~\cite[Theorem~8.2]{Hitchhikers-guide} the Morrey-type embedding
$$
    W^{\gm,q} (B_R) \hookrightarrow C^{0,\gm-\frac{N}{q}}(B_R).
$$
Applying the Morrey embedding on $B_1$ to the rescaled function $\tilde u_R=u_R-(u_R)_{B_1}$, where $u_R(x)=u(Rx+x_o)$ and $B_R(x_o)\subset\R^N$ and subsequently the fractional Poincar\'e inequality leads to the following Lemma; 
cf.~\cite[Proposition~2.2]{Novak:2022}. 

\begin{lemma}\label{Lem:morrey}
Let $q\ge 1$ and $\gm \in (0,1)$ such that $\gm q>N$. Then there exists a constant $C=C(N,q,\gm)$  such that for any $w\in W^{\gm, q}(B_R)$ we have
\begin{align*}
    [w]_{C^{0,\gm-\frac{N}{q}}(B_R)}\le C[w]_{W^{\gm,q}(B_R)}.
\end{align*}
\end{lemma}
Finally, we state the well known Morrey embedding for Sobolev functions \cite{Morrey}.
\begin{lemma}\label{Lem:morrey-classic}
Let $q\ge 1$ such that $q>N$. Then there exists a constant $C=C(N,q)$  such that for any $w\in W^{1, q}(B_R)$ we have
\begin{align*}
    [w]_{C^{0,1-\frac{N}{q}}(B_R)}\le C\|\nabla w\|_{L^{q}(B_R)}.
\end{align*}
\end{lemma}

\section{Energy inequalities}

The aim of this section is to derive energy estimates for finite differences on balls $B_R(x_o)\Subset\Omega$. In the course of the proof we have to control terms involving integrals outside $B_R(x_o)$, the so-called tail terms. These estimates will be derived in the following subsections.

\subsection{Tail estimate for finite differences}\label{sec:tail}

In regularity theory, it is crucial for the proof of almost  optimal statements to have the best possible energy inequalities at hand. For non-local fractional problems such as the fractional $p$-Laplace operator, unavoidable non-local terms, the so-called {\em tail terms}, occur after testing the equation. In our case, we test the equation with $\varphi := V_\delta (\btau_hu)\eta^p$, where $\delta\ge 1$, $h\in\R^N\setminus\{0\}$ and $\eta$ is a suitable cut-off function. This leads to a non-local term in which, among other things, the difference
$$
    V_{p-1}\big(u(x+h)-u(y+h)\big)- V_{p-1}\big(u(x)-u(y)\big)
$$ 
appears. Even if $u$ is neither differentiable nor fractionally differentiable outside the domain $\Omega$, this difference can be quantifiably exploited
in terms of the step size $|h|$ and the finite difference $\btau_hu$. This is precisely the point where we succeed in extending the validity of the previously known regularity statements to the range $s\in (\frac{p-2}p,1)$ instead of $s\in (\frac{p-1}p,1)$. Since all our results are stable in the limit $p\downarrow 2$ (actually we prove them directly for $p\in [2,\infty)$), we obtain in the case $p=2$ that all statements are valid for the whole range $s\in (0,1)$.

\begin{lemma}\label{Lm:tail}
Let $p\in [2,\infty)$ and  $s\in(0,1)$. There exists a constant $C=C(N,p,s)$ such that whenever $u\in L^{p-1}_{sp}(\R^N)$, $x_o\in\R^N$, $R>0$, $r\in(0,R)$, and $d\in (0,\frac 14(R-r)]$, we have for any $x\in B_{\frac12 (R+r)}(x_o)$ and any $h\in\R^N\setminus\{0\}$ with $0<|h|\le d$
that
\begin{align}\label{est:tail-p>2}\nonumber
    \bigg|\int_{\R^N\setminus B_R(x_o)}&
    \frac{ V_{p-1}(u_h(x)-u_h(y))- V_{p-1}(u(x)-u(y)) }{|x-y|^{N+sp}}\,\dy \bigg|\\
    &\le 
    C\frac{|\btau_h u(x)|}{R^{sp}}\Big(\frac{R}{R-r}\Big)^{N+sp}
    \mathcal T^{p-2} +
    C\frac{|h|}{R^{sp+1}} 
    \Big(\frac{R}{R-r}\Big)^{N+sp+1}
    \mathcal T^{p-1}.
\end{align}
where
\begin{equation*}
    \mathcal T:= \|u\|_{L^{\infty}(B_{R+d}(x_o))}+\mathrm{Tail}(u;x_o,R+d).
\end{equation*}
In addition, the constant $C$ has the form $C=\widetilde{C}(N,p)/s$.
\end{lemma}

\begin{proof}

Instead of the center $x_o$ we consider balls centered at $0$, and prove the inequality for the translated function $x\mapsto u(x-x_o)$. However, we still write $u$ for simplicity, while  keeping in mind that $u$ is the translated function. Split the integral on the left side of \eqref{est:tail-p>2} into two terms, and transform the first one using the transformation $z=y+h$ to obtain
\begin{align*}
    &\int_{\R^N\setminus B_R(h)}
    \frac{V_{p-1}(u_h(x)-u(y))}{|x+h-y|^{N+sp}}\,\dy -
    \int_{\R^N\setminus B_R}
    \frac{V_{p-1}(u(x)-u(y))}{|x-y|^{N+sp}}\,\dy \\
    &\quad=
    \int_{\R^N\setminus (B_R(h)\cup B_R)}
    \bigg[\frac{V_{p-1}(u_h(x)-u(y))}{|x+h-y|^{N+sp}} -
    \frac{V_{p-1}(u(x)-u(y))}{|x-y|^{N+sp}} \bigg]\,\dy \\
    &\quad\phantom{=\,}
    +
    \int_{B_R\setminus B_R(h)}
    \frac{V_{p-1}(u_h(x)-u(y))}{|x+h-y|^{N+sp}}\,\dy -
    \int_{B_R(h)\setminus B_R}
    \frac{V_{p-1}(u(x)-u(y))}{|x-y|^{N+sp}}\,\dy.
\end{align*}
Applying absolute values on both sides and then the triangle inequality, it is clear that we end up with three integrals $\mathbf{I}_j$, $j=1,2,3$ on the right-hand side. To proceed, we treat the first of them. Indeed, observe that since $|y|>R$ and $|x|<\frac12(R+r)$, for any $\xi\in B_{|h|}$ with $0<|h|\le d\le\frac14 (R-r)$, we have 
\begin{align}\label{est:xi}
    \frac{|x+\xi-y|}{|y|}\ge 1-\frac{|x|}{|y|}-\frac{|\xi|}{|y|}\ge1-\frac{R+r}{2R}-\frac{d}{R}\ge\frac{R-r}{4R}.
\end{align}
We apply the mean value theorem to the function $[0,1]\ni t\mapsto |x+t h-y|^{-N-sp}$ to 
find some $t\in [0,1]$, such that
\begin{align*}
    \bigg|\frac{1}{|x+h-y|^{N+sp}} -
    \frac{1}{|x-y|^{N+sp}} \bigg|
    &\le \frac{(N+sp)|h|}{|x+th-y|^{N+sp+1}}\\
    &\le
    |h| \Big(\frac{4R}{R-r}\Big)^{N+sp+1} \frac{N+p}{|y|^{N+sp+1}}.
\end{align*}
Here, we used \eqref{est:xi}, which is possible since $th\in B_{|h|}$. Now, we use the above observation together with Lemma~\ref{lem:Acerbi-Fusco} (choosing $b=u_h(x)-u(y)$, $a=u(x)-u(y)$, $\gamma =p-1\ge 1$, $C_2=p-1$) to estimate  the integrand of the first integral by
\begin{align*}
    \mathbf V_h&:=\bigg|\frac{V_{p-1}(u_h(x)-u(y))}{|x+h-y|^{N+sp}} -
    \frac{V_{p-1}(u(x)-u(y))}{|x-y|^{N+sp}} \bigg|\\
    &
    \le
    \big|V_{p-1}(u_h(x)-u(y))\big| \bigg|\frac{1}{|x+h-y|^{N+sp}} -
    \frac{1}{|x-y|^{N+sp}} \bigg| \\
    &\phantom{\le\,}
    +
    \frac{1}{|x-y|^{N+sp}}\big|V_{p-1}(u_h(x)-u(y)) -
    V_{p-1}(u(x)-u(y)) \big| \\
    &
    \le 
   C|h| \Big(\frac{R}{R-r}\Big)^{N+sp+1} \frac{|u_h(x)-u(y)|^{p-1}}{|y|^{N+sp+1}}\\
    &\phantom{\le\,}
    +
    C\frac{|\btau_h u(x)|}{|x-y|^{N+sp}} \big(|u_h(x)-u(y)|+|u(x)-u(y)|\big)^{p-2}
\end{align*}
for a constant $C=C(N,p)$. The second term on the right-hand side is straightforward to estimate. Indeed,
\begin{align*}
    \mathbf V_h
    &
    \le C|h| \Big(\frac{R}{R-r}\Big)^{N+sp+1} \frac{|u_h(x)|^{p-1} + |u(y)|^{p-1}}{|y|^{N+sp+1}}\\
    &\phantom{\le\,}
    + 
    C\Big(\frac{R}{R-r}\Big)^{N+sp}\frac{|\btau_h u(x)|}{|y|^{N+sp}}\big(|u_h(x)|+|u(x)|+2|u(y)|\big)^{p-2}.
\end{align*}
For the last inequality we used that $|x-y|\ge \frac{R-r}{2R}|y|$ for any $|x|\le\frac12 (R+r)$ and $|y|>R$. 
The constant $C$ depends only on $N$ and $p$.
Consequently,  the first integral is estimated by
\begin{align*}
    \mathbf{I}_1 
    &\le  
    C|h| \Big(\frac{R}{R-r}\Big)^{N+sp+1} \int_{\R^N\setminus (B_R(h)\cup B_R)}\frac{|u_h(x)|^{p-1} + |u(y)|^{p-1}}{|y|^{N+sp+1}}\dy\\
    &\phantom{\le\,}
    + 
    C\Big(\frac{R}{R-r}\Big)^{N+sp}|\btau_h u(x)|\int_{\R^N\setminus (B_R(h)\cup B_R)}\frac{\big(|u_h(x)|+|u(x)|+2|u(y)|\big)^{p-2}}{|y|^{N+sp}}\,\dy\\
    &\le  
    C|h| \Big(\frac{R}{R-r}\Big)^{N+sp+1} \int_{\R^N\setminus B_R}\frac{\|u\|_{L^{\infty}(B_R)}^{p-1} + |u(y)|^{p-1}}{|y|^{N+sp+1}}\,\dy\\
    &\phantom{\le\,}
    + 
    C\Big(\frac{R}{R-r}\Big)^{N+sp}|\btau_h u(x)|\int_{\R^N\setminus B_R}\frac{\|u\|_{L^{\infty}(B_R)}^{p-2}+|u(y)|^{p-2}}{|y|^{N+sp}}\,\dy\\
    &
    \le
     C\frac{|h|}{R^{sp+1}}
     \Big(\frac{R}{R-r}\Big)^{N+sp+1}
     \bigg[
     \tfrac1{sp+1}\|u\|_{L^{\infty}(B_R)}^{p-1}
     +
     R^{sp}\int_{\R^N\setminus B_R}\frac{|u(y)|^{p-1}}{|y|^{N+sp}}\dy
     \bigg]
     \\
     &
     \phantom{\le\,}
     +
     C\frac{|\btau_h u(x)|}{R^{sp}}\Big(\frac{R}{R-r}\Big)^{N+sp}
     \bigg[
     \tfrac1{sp}\|u\|_{L^{\infty}(B_R)}^{p-2}
     +
     R^{sp}\int_{\R^N\setminus B_R}\frac{|u(y)|^{p-2}}{|y|^{N+sp}}\,\dy
     \bigg]\\
    &\le
    C\frac{|h|}{R^{sp+1}} \Big(\frac{R}{R-r}\Big)^{N+sp+1}\mathcal{T}^{p-1}
    +
    \frac{C}{s}\frac{|\btau_h u(x)|}{R^{sp}}\Big(\frac{R}{R-r}\Big)^{N+sp}
    \mathcal{T}^{p-2}.
\end{align*}
Here, to obtain the last line, we used the definition of $\mathcal{T}$, H\"older's inequality and Lemma~\ref{lem:t} (to estimate $\Tail (u;R)$ in terms $\mathcal T$) as
\begin{align*}
    \int_{\R^N\setminus B_R}\frac{|u(y)|^{p-2}}{|y|^{N+sp}}\,\dy
    &=
    \int_{\R^N\setminus B_R}\frac{|u(y)|^{p-2}}{|y|^{(N+sp)\frac{p-2}{p-1}}}\frac{1}{|y|^{\frac{N+sp}{p-1}}}\,\dy\\
    &\le
    \bigg[\int_{\R^N\setminus B_R}\frac{|u(y)|^{p-1}}{|y|^{N+sp}}\,\dy\bigg]^{\frac{p-2}{p-1}}
    \bigg[\int_{\R^N\setminus B_R}\frac{1}{|y|^{N+sp}}\,\dy
    \bigg]^{\frac{1}{p-1}}\\
    &=
    \frac{C}{(sp)^\frac1{p-1}R^{sp}}
   \bigg[R^{sp}\int_{\R^N\setminus B_R}\frac{|u(y)|^{p-1}}{|y|^{N+sp}}\,\dy\bigg]^{\frac{p-2}{p-1}}\\
   &=
   \frac{C}{(sp)^\frac1{p-1}R^{sp}}\Tail (u;R)^{p-2} \\
   &\le 
   \frac{C}{(sp)^\frac1{p-1}R^{sp}} 
   \Big(\frac{R+d}{R}\Big)^{\frac{p-2}{p-1}N} \mathcal{T}^{p-2} \\
   &\le 
   \frac{C}{(sp)^\frac1{p-1}R^{sp}} 
   \mathcal{T}^{p-2},
\end{align*}
for a constant $C=C(N,p)$, and $(sp)^{1-\frac1{p-1}}=(sp)^\frac{p-2}{p-1}\le p$. To obtain the last line we used $\frac{R+d}{R}\le 2$.

To deal with the  integral $\mathbf I_2$, we use the fact that $|x|<\frac12(R+r)$, $|y-h|>R$, and hence $|x+h-y|\ge|y-h|-|x|\ge\frac12(R-r)$. Consequently, 
\begin{align*}
    \mathbf{I}_2&=\int_{B_R\setminus B_R(h)}
    \frac{\big|V_{p-1}(u_h(x)-u(y))\big|}{|x+h-y|^{N+sp}}\,\dy\\
    &\le\frac{C}{(R-r)^{N+sp}}\int_{B_R\setminus B_R(h)}\big(|u_h(x)|^{p-1} + |u(y)|^{p-1}\big)\,\dy\\
    &\le \frac{C |h| R^{N-1}}{(R-r)^{N+sp}}\|u\|_{L^{\infty}(B_R)}^{p-1}.
\end{align*}
Here, in the last line we also used that $|B_R\setminus B_R(h)|\le C(N)|h| R^{N-1}$.
Similarly, to deal with the  integral $\mathbf I_3$, we use the fact that $|x|<\frac12(R+r)$ and $|y|>R$, and hence $|x-y|\ge|y|-|x|\ge\frac12(R-r)$. Consequently, 
\begin{align*}
    \mathbf{I}_3&=\int_{B_R(h)\setminus B_R}
    \frac{\big|V_{p-1}(u(x)-u(y))\big|}{|x-y|^{N+sp}}\,\dy\\
    &\le\frac{C}{(R-r)^{N+sp}}\int_{B_R(h)\setminus B_R}\big(|u(x)|^{p-1} + |u(y)|^{p-1}\big)\,\dy\\
    &\le \frac{C |h| R^{N-1}}{(R-r)^{N+sp}}\|u\|_{L^{\infty}(B_{R+d})}^{p-1}.
\end{align*}
Note that the constants in the inequalities for $\mathbf I_2$ and $\mathbf I_3$ depend only on $N$ and $p$.
Collecting all these estimates concludes the proof.
\end{proof}

\subsection{Energy inequalities for finite differences}\label{sec:energy}

As we will see,  the tail estimate from Lemma~\ref{Lm:tail} for finite differences  plays an important role in proving the energy inequality. In a certain sense, the exponents of the increment $|h|$ and the finite difference $|\btau_hu|$ determine the gain in fractional differentiability of $V_{\frac{p+\delta-1}{p}}(\btau_hu)$, where $\delta\ge 1$.
In order to derive local energy inequalities, we need a localized version of this expression, i.e.~instead of $V_{\frac{p+\delta-1}{p}}(\btau_hu)$, we consider $\eta V_{\frac{p+\delta-1}{p}}(\btau_hu)$ with some cut-off function $\eta$. To avoid constantly explaining the choice of
the cut-off function $\eta$, we fix the class of cut-off functions in advance. 
\begin{definition}[The class of cut-off functions]\label{Def:Z}\upshape
Given $x_o\in \R^N$ and radii $0<r<R$, by $\mathfrak Z_{r,R}(x_o)$ we denote the class of functions $\eta\in C^1_0\big(B_{\frac12 (R+r)}(x_o),[0,1]\big)$ that satisfy $\eta=1$ in $B_r(x_o)$ and $\|\nabla\eta\|_{B_{\frac12 (R+r)}(x_o)} \le \frac{4}{R-r}$.
\hfill $\Box$
\end{definition}

The first energy estimate is stated for any order $s\in(0,1)$. However, it will only be used in the  regime $s\in(\frac{p-2}p,1)$, which is considered in \S\,\ref{sec:super}.

\begin{proposition}[First energy inequality]\label{prop:energy}
Let $p\in[2,\infty)$, $s\in(0,1)$, and $\delta\in [1,\infty)$. There exists a constant $C=C(N,p,s,\delta)$ such that whenever
$u\in W^{s,p}_{\rm loc}(\Omega)\cap L^{p-1}_{sp}(\R^N)$
is a locally bounded, local weak solution of~\eqref{PDE} in the sense of Definition~\ref{def:loc-sol} that satisfies
$$
 u\in W^{\frac{sp+\delta-1}{p+\delta-1},p+\delta-1 }_{\rm loc}(\Omega),
$$
then for any $0<r<R$, $d\in (0,\tfrac14 (R-r)]$,
$B_{R+d}\equiv B_{R+d}(x_o)\Subset\Omega$, 
$\eta\in \mathfrak Z_{r,R}(x_o)$, and  any  step size $0<|h|\le d$ we have the energy inequality
\begin{align*}
    \iint_{K_R} &
    \frac{\big|V_{\frac{p+\delta-1}{p}}(\btau_hu(x)) \eta(x) - V_{\frac{p+\delta-1}{p}}(\btau_hu(y)) \eta(y)\big|^{p}}{|x-y|^{N+sp}} \,\dx\dy \\
    &\le 
    \frac{C}{(R-r)^2}
    \bigg[\iint_{K_{R+d}} \frac{|u(x)-u(y)|^{p+\delta-1}}{|x-y|^{N+sp+\delta-1}} \,\dx\dy \bigg]^{\frac{p-2}{p+\delta-1}}
    \\
    &\qquad\qquad\qquad\cdot
    \bigg[\frac{R^{(1-s)p}}{1-s} 
    \int_{B_R} |\btau_h u|^{p+\delta-1} \,\dx\bigg]^{\frac{\delta+1}{p+\delta-1}}\\
    &\phantom{\le\,}+
    \frac{C}{(1-s)R^{sp}}\Big(\frac{R}{R-r}\Big)^{N+sp}\mathcal T^{p-2}\int_{B_{R}} |\btau_h u|^{\delta+1}\,\dx\\ 
    &\phantom{\le\,} +
   \frac{C |h|}{R^{sp+1}} 
    \Big(\frac{R}{R-r}\Big)^{N+sp+1}
    \mathcal T^{p-1}\int_{B_{R}} |\btau_hu|^\delta \,\dx,
\end{align*}
where
\begin{equation*}
    \mathcal T
    :=
    \|u\|_{L^{\infty}(B_{R+d})}+\mathrm{Tail}(u;R+d).
\end{equation*}
The dependence of the constant $C$ on $\delta$ is of the form 
$C=\widetilde{C}(N,p)s^{-1}8^{\delta}\delta^p$.
\end{proposition}

\begin{proof}
Consider $x_o\in \Omega$, $0<r<R$ and $d\in (0,\tfrac14 (R-r)]$ such that $B_{R+d}(x_o)\Subset\Omega$. 
Since $x_o$ is fixed we omit the reference to the center $x_o$ and write $B_\rho$ and $K_\rho$ instead of $B_\rho(x_o)$ and $K_\rho(x_o)$. 
Let $h\in \R^N\setminus\{ 0\}$ with $|h|<d$. Testing \eqref{weak-sol} with $\varphi_{-h}(x):=\varphi(x-h)$ instead of $\varphi$, where $\varphi\in W^{s,p}(B_R)$ with $\spt\varphi\in B_{\frac12 (R+r)}$, we conclude by discrete integration by parts that also $u_h$ satisfies \eqref{weak-sol}. Subtracting \eqref{weak-sol} with $u$ from \eqref{weak-sol} with $u_h$, we obtain
\begin{align}\label{system-h}
    \iint_{\R^N\times\R^N}
    \frac{\big(V_{p-1}(u_h(x){-}u_h(y))- V_{p-1}(u(x){-}u(y))\big)(\varphi (x){-}\varphi(y)) }{|x-y|^{N+sp}}\,\dx\dy=0
\end{align}
for any $\varphi\in W^{s,p}(B_R)$ with $\spt\varphi\in B_{\frac12 (R+r)}$. 
In \eqref{system-h} we now choose
\begin{equation*}
    \varphi := V_\delta (\btau_hu)\eta^p\quad
    \mbox{with $\delta\ge 1$ and $\eta\in \mathfrak Z_{r,R}(x_o)$.}
\end{equation*}
Since $u$ is locally bounded, one can verify that $\varphi\in W^{s,p}(B_R)$.
Decomposing $\R^N$ into $B_R$ and its complement $\R^N\setminus B_R$ we obtain 
\begin{align*}
    0
    &=
     \iint_{K_R}
    \frac{\big(V_{p-1}(U_h(x,y)){-} V_{p-1}(U(x,y))\big)\big(V_\delta (\btau_hu)\eta^p (x){-}V_\delta (\btau_hu)\eta^p(y)\big) }{|x-y|^{N+sp}}\,\dx\dy\\
    &\phantom{=\,}
    +
    \iint_{B_{\frac12 (R+r)}\times (\R^N\setminus B_R)}
    \frac{\big(V_{p-1}(U_h(x,y))- V_{p-1}(U(x,y))\big)V_\delta (\btau_hu)\eta^p (x) }{|x-y|^{N+sp}}\,\dx\dy\\
    &\phantom{=\,}
    -
    \iint_{(\R^N\setminus B_R)\times B_{\frac12 (R+r)}}
    \frac{\big(V_{p-1}(U_h(x,y))- V_{p-1}(U(x,y))\big)V_\delta (\btau_hu)\eta^p (y)}{|x-y|^{N+sp}}\,\dx\dy.
\end{align*}
Here, we abbreviated $U_h(x,y):= u_h(x)-u_h(y)=-U_h(y,x)$ and  $U(x,y):= U_o(x,y)$. Furthermore, the term  $V_\delta (\btau_hu)\eta^p (x)$ is to be understood in such a way that the argument $x$ appears in both factors, i.e.~$V_\delta (\btau_hu)\eta^p (x) =(V_\delta (\btau_hu\eta^p) (x)$. Analogously, the same applies to $V_\delta (\btau_hu)\eta^p (y )$. By interchanging the roles of $x$ and $y$ in the second integral, it can be seen that the second integral coincides with the last integral except for the sign. Therefore, we get 
\begin{equation}\label{eq:I=2T}
      \mathbf I=-2\mathbf T,
\end{equation}
where
\begin{align*}
    \mathbf I
    &:=
      \iint_{K_R}
    \frac{\big(V_{p-1}(U_h(x,y)){-} V_{p-1}(U(x,y))\big)\big(V_\delta (\btau_hu)\eta^p (x){-}V_\delta (\btau_hu)\eta^p(y)\big) }{|x-y|^{N+sp}}\,\dx\dy,\\
    \mathbf T
    &:=
    \iint_{B_{\frac12 (R+r)}\times (\R^N\setminus B_R)}
    \frac{\big(V_{p-1}(U_h(x,y))- V_{p-1}(U(x,y))\big)V_\delta (\btau_hu)\eta^p (x)}{|x-y|^{N+sp}}\,\dx\dy.
\end{align*}
Now, we turn our attention to the {\bf estimation of the local term} $\mathbf I$.  Applying Lemma \ref{lem:algebraic-1} with $a=u_h(x)$, $b=u_h(y)$, $c=  u(x)$, $d=  u(y)$, $e=\eta^\frac{p}2(x)$, and  $f=\eta^\frac{p}2(y)$ we have
\begin{align*}
    \mathbf I
    &\ge \tfrac1C\mathbf I_1-C\mathbf I_2
\end{align*}
for a constant $C=\widetilde{C}(p)2^\delta$. Here, we abbreviated
\begin{align*}
    \mathbf I_1
    &:=
    \iint_{K_R}
     \frac{(|U_h(x,y)|+|U(x,y)|)^{p-2} (|\btau_hu(x)|+|\btau_hu(y)|)^{\delta-1}}{|x-y|^{N+sp}}\\
     &\qquad \qquad\qquad\qquad\qquad
     \cdot |\btau_hu(x)- \btau_hu(y)|^2\big( \underbrace{\eta^p(x)+\eta^p(y)}_{=:\Theta (x,y)}\big) \,\dx\dy
\end{align*}
and 
\begin{align*}
    \mathbf I_2
    &:=
    \iint_{K_R}
     \frac{(|U_h(x,y)|+|U(x,y)|)^{p-2} (|\btau_hu(x)|+|\btau_hu(y)|)^{\delta+1}}{|x-y|^{N+sp}}\\
     &\qquad \qquad\qquad\qquad\qquad\qquad\qquad\qquad\quad
     \cdot \big| \eta^\frac{p}2(x)-\eta^\frac{p}2(y) \big|^2 \,\dx\dy.
\end{align*}
Since $\mathbf I=-2\mathbf T$ as in  \eqref{eq:I=2T}, we obtain
\begin{equation}\label{eq:I_1}
    \mathbf I_1
    \le 
    \widetilde{C}(p) 2^{2\delta}\big[ \mathbf I_2+ |\mathbf T|\big].
\end{equation}
Using the elementary inequality
\begin{equation*}
    \big( |U_h(x,y)|+|U(x,y)| \big)^{p-2}
    \ge 
    \big|U_h(x,y)-U(x,y) \big|^{p-2}
    =
    \big|\btau_hu(x) -\btau_hu(y) \big|^{p-2},
\end{equation*}
we can further estimate $\mathbf I_1$ from below. Using in turn also Lemma \ref{lem:Acerbi-Fusco} with $\gamma =\frac{\delta-1}{p}+1$, $a=\btau_hu(x)$, and $b=\btau_hu(y)$, afterwards Lemma \ref{lem:algebraic-2} with $\gamma =p$, $A=V_\frac{p+\delta-1}{p}(\btau_hu(x))$, $B=V_\frac{p+\delta-1}{p}(\btau_hu(y))$, $e=\eta(x)$, and $f=\eta (y)$, and finally the convexity of $t\mapsto t^p$ (note that we are dealing with the case $p\ge 2$) we have
\begin{align*}
    \mathbf I_1
    &\ge
    \iint_{K_R}
     \frac{ (|\btau_hu(x)|+|\btau_hu(y)|)^{\delta-1}|\btau_hu(x)- \btau_hu(y)|^p\Theta (x,y)}{|x-y|^{N+sp}} \,\dx\dy\\
     &\ge
     \frac1{\delta^p} \iint_{K_R}
     \frac{\big| V_\frac{p+\delta-1}{p}(\btau_hu(x))-V_\frac{p+\delta-1}{p}(\btau_hu(y))\big|^p \big( \eta^p(x)+\eta^p(y)\big) }{|x-y|^{N+sp}} \,\dx\dy\\
     &\ge
     \frac{2^{2-p}}{\delta^p }
     \iint_{K_R}
     \frac{\big| V_\frac{p+\delta-1}{p}(\btau_hu(x))\eta (x)-V_\frac{p+\delta-1}{p}(\btau_hu(y))\eta(y)\big|^p  }{|x-y|^{N+sp}} \,\dx\dy\\
     &\phantom{\ge\,}
     -
     \frac{2^{1-p}}{\delta^p }
     \iint_{K_R}
     \frac{\big( |\btau_hu(x)|^\frac{p+\delta-1}{p} +|\btau_hu(y)|^\frac{p+\delta-1}{p}\big)^p |\eta(x)-\eta (y)|^p}{|x-y|^{N+sp}} \,\dx\dy\\
     &\ge
     \frac{2^{2-p}}{\delta^p }
     \iint_{K_R}
     \frac{\big| V_\frac{p+\delta-1}{p}(\btau_hu(x))\eta (x)-V_\frac{p+\delta-1}{p}(\btau_hu(y))\eta(y)\big|^p  }{|x-y|^{N+sp}} \,\dx\dy\\
     &\phantom{\ge\,}
     -
     \frac{1}{\delta^p }
     \iint_{K_R}
     \frac{\big( |\btau_hu(x)|^{p+\delta-1} +|\btau_hu(y)|^{p+\delta-1}\big) |\eta(x)-\eta (y)|^p}{|x-y|^{N+sp}} \,\dx\dy\\
     &=
     \mathbf I_{1,1}-\mathbf I_{1,2},
\end{align*}
with the obvious meaning of $\mathbf I_{1,1}$ and $\mathbf I_{1,2}$. In combination with \eqref{eq:I_1} this yields
\begin{equation}\label{split-I_11}
    \mathbf I_{1,1}
    \le 
    \widetilde{C}(p) 2^{2\delta}\big[ \mathbf I_2+ |\mathbf T|\big] +
    \mathbf I_{1,2}.
\end{equation}
The integral $\mathbf I_{1,2}$ can be further estimated from above 
using the Lipschitz bound $|\eta (x)-\eta (y)|\le \|\nabla\eta\|_{L^\infty}|x-y|\le\frac{4}{R-r}|x-y|$, the symmetry of the resulting integral, and Lemma~\ref{int-sing}. This leads to 
\begin{align*}
    \mathbf I_{1,2}
    &\le
    \frac{4^p}{\delta^p (R-r)^p}
    \iint_{K_R}
     \frac{ |\btau_hu(x)|^{p+\delta-1}+|\btau_hu(y)|^{p+\delta-1}}{|x-y|^{N+(s-1)p}}\,\dx\dy\\
     &=
     \frac{2^{2p+1}}{\delta^p (R-r)^p}
    \iint_{K_R}
     \frac{ |\btau_hu(x)|^{p+\delta-1}}{|x-y|^{N+(s-1)p}}\,\dx\dy\\
     &=
     \frac{2^{2p+1}}{\delta^p (R-r)^p}
    \int_{B_R}\bigg[ 
    \underbrace{\int_{B_R}
     \frac{ 1 }{|x-y|^{N+(s-1)p}}\dy}_{\le \frac{\om_N}{(1-s)p}R^{(1-s)p}}\bigg]|\btau_hu(x)|^{p+\delta-1} \,\dx\\
     &\le
     \frac{\om_N 2^{2p+1}}{\delta^p (1-s)p}\frac{R^{(1-s)p}}{(R-r)^p}
     \int_{B_R} |\btau_hu(x)|^{p+\delta-1} \,\dx\\
     &\le
     \frac{\om_N 2^{3p-1}}{\delta^p (1-s)p}\frac{R^{(1-s)p}}{(R-r)^p}
     \|u\|_{L^{\infty}(B_R)}^{p-2}
     \int_{B_R} |\btau_hu(x)|^{\delta+1} \,\dx.
\end{align*}
In the following, we consider  the integral $\mathbf I_2$. To the difference of the cut-off functions we apply Lemma \ref{lem:Acerbi-Fusco} with $\gamma =\frac12p$ and obtain
\begin{align*}
    \big| \eta^\frac{p}2(x)-\eta^\frac{p}2(y)\big|^2
    &\le 
    C(p) \big(\eta(x)+\eta (y)\big)^{p-2}|\eta (x)-\eta (y)|^2\\
    &\le 
    \frac{C(p)}{(R-r)^2} |x-y|^2.
\end{align*}
In the last line we used the assumption $\|\nabla \eta\|_{L^\infty}\le\frac4{R-r}$.
This reduces in $\mathbf I_2$ the exponent of  $|x-y|$ from $N+sp$ to $N+sp-2$ and hence allows for an application of H\"older's inequality with exponents $\frac{p+\delta-1}{\delta+1}$ and $\frac{p+\delta-1}{p-2}$ (which is necessary only if $p>2$). For the application we decompose the exponent $N+sp-2$  in the form 
\begin{align*}
    &N+sp-2
    =
    \big(N-(1-s)p\big)\frac{\delta+1}{p+\delta-1}
    +\big( N+sp+\delta-1\big) \frac{p-2}{p+\delta-1}. 
\end{align*}
This leads to
\begin{align}\label{est:I_2}\nonumber
    \mathbf I_2
    &\le
    \frac{C(p)}{(R-r)^2}
    \iint_{K_R}
     \frac{(|U_h(x,y)|+|U(x,y)|)^{p-2} (|\btau_hu(x)|+|\btau_hu(y)|)^{\delta+1}}{|x-y|^{N+sp-2}}
     \,\dx\dy\\\nonumber
     &\le 
     \frac{C(p)}{(R-r)^2}
     \Bigg[
     \iint_{K_R} \frac{(|\btau_hu(x)|+|\btau_hu(y)|)^{p+\delta-1}}
     {|x-y|^{N-(1-s)p}}\,\dx\dy 
     \Bigg]^\frac{\delta+1}{p+\delta-1}\\
     &\qquad\qquad\qquad\cdot
    \Bigg[
     \iint_{K_R} \frac{(|u_h(x)-u_h(y)|+|u(x)-u(y)|)^{p+\delta-1}}
     {|x-y|^{N + sp +\delta-1}} \,\dx\dy 
     \Bigg]^\frac{p-2}{p+\delta-1}.
\end{align}
We estimate the integrand of the first integral using the convexity of $t\mapsto t^{p+\delta -1}$. Afterwards we use the symmetry of the resulting integrand in the arguments $x,y$ and Lemma~\ref{int-sing} with exponent $\beta=(1-s)p$. In the second integral, we eliminate the dependence on $h$ in the integrand by enlarging the domain of integration. In this way we get
\begin{align*}
    \mathbf I_2
    &\le
    \frac{2^{\delta}C(p)}{(R-r)^2}
    \Bigg[ 
    \int_{B_R}\bigg[ 
    \underbrace{\int_{B_R}\frac{1}{|x-y|^{N-(1-s)p}}\dy}_{\le \frac{\om_N}{(1-s)p}R^{(1-s)p}}\bigg]
    |\btau_hu(x)|^{p+\delta-1}\dx
    \Bigg]^\frac{\delta+1}{p+\delta-1}\\
    &\qquad\qquad\qquad\cdot
    \bigg[
     \iint_{K_{R+d}} \frac{|u(x)-u(y)|^{p+\delta-1}}
     {|x-y|^{N+ sp +\delta-1}}\dx\dy 
     \bigg]^\frac{p-2}{p+\delta-1}\\
     &\le
     \frac{2^{\delta}C(p,N) }{(R-r)^2}
     \bigg[ \frac{R^{(1-s)p}}{1-s} \int_{B_R} |\btau_hu(x)|^{p+\delta-1}\dx\bigg] ^\frac{\delta+1}{p+\delta-1}\\
     &\qquad\qquad\qquad\cdot
    \bigg[
     \iint_{K_{R+d}} \frac{|u(x)-u(y)|^{p+\delta-1}}
     {|x-y|^{N+sp +\delta-1}}\dx\dy 
     \bigg]^\frac{p-2}{p+\delta-1}.
\end{align*}
In the case $p=2$ a similar estimate holds true. Indeed, the second integral on the right-hand side in the last displayed estimate is to be interpreted as 1 in this case. Whereas the constant in the front takes the form $\frac{2^\delta C(N)}{1-s}\frac{R^2}{(R-r)^2}$, so that
\begin{align*}
        \mathbf I_2
        &\le
        \frac{2^\delta C(N)}{1-s}\frac{R^{2(1-s)}}{(R-r)^2} 
        \int_{B_R} |\btau_hu(x)|^{\delta+1}\dx.
\end{align*}
Next, we deal with the nonlocal term $\mathbf{T}$. First of all, from Lemma~\ref{Lm:tail} there exists some positive constant $C$ of the form $\widetilde C(N,p)/s$, such that for any $x\in B_{\frac12 (R+r)}$ we have
\begin{align*}
    \bigg|\int_{\R^N\setminus B_R}&
    \frac{ V_{p-1}(U_h(x,y))- V_{p-1}(U(x,y)) }{|x-y|^{N+sp}}\,\dy \bigg|\\
    &\le
    C\frac{|\btau_h u(x)|}{R^{sp}}\Big(\frac{R}{R-r}\Big)^{N+sp}
    \mathcal{T}^{p-2}
    +    
    C\frac{|h|}{R^{sp+1}} 
    \Big(\frac{R}{R-r}\Big)^{N+sp+1}
    \mathcal{T}^{p-1}.
\end{align*}
As a result of this estimate, we obtain
\begin{align*}
   | \mathbf{T}|
    &\le 
    \frac{C}{R^{sp}}\Big(\frac{R}{R-r}\Big)^{N+sp}\mathcal{T}^{p-2}\int_{B_{\frac{R+r}2}}|V_\delta (\btau_hu(x))||\btau_h u(x)|\eta^p (x)\,\dx \\
    &\quad+\frac{C|h|}{R^{sp+1}} 
    \Big(\frac{R}{R-r}\Big)^{N+sp+1}
    \mathcal{T}^{p-1}\int_{B_{\frac{R+r}2}}|V_\delta (\btau_hu(x))|\eta^p (x)\,\dx\\
    &\le 
    \frac{C}{R^{sp}}\Big(\frac{R}{R-r}\Big)^{N+sp}\mathcal{T}^{p-2}\int_{B_{R}}|\btau_h u(x)|^{\delta+1}\,\dx \\
    &\quad+\frac{C|h|}{R^{sp+1}} 
    \Big(\frac{R}{R-r}\Big)^{N+sp+1}
    \mathcal{T}^{p-1}\int_{B_{R}} |\btau_hu(x)|^\delta \,\dx.
\end{align*}
Upon inserting the corresponding estimates for the individual terms into inequality~\eqref{split-I_11}, the assertion is obtained. In the final  estimate, two terms admit  a 
dependence on $1/(1-s)$. These result from the estimates of $\mathbf I_2$ and $\mathbf I_{1,2}$. The $1/s$-dependence of the final constant $C$ is a result of the tail estimate.
\end{proof}

The second energy estimate is used in the  regime $s\in(0,\frac{p-2}p]$, which is considered in \S  \ref{sec:sub}.

\begin{proposition}[Second energy inequality]\label{prop:energy-s}
Let $p\in(2,\infty)$, $s\in(0,\frac{p-2}p]$, $\delta\in [1,\infty)$, and 
\begin{equation*}
    \sigma
    \in 
    \bigg(\max\Big\{\frac{sp-2}{p-2},0\Big\}, \,\frac{sp}{p-2}\bigg).
\end{equation*}
There exists a constant $C=C(N,p,s,\delta)$ such that whenever
$u\in W^{s,p}_{\rm loc}(\Omega)\cap L^{p-1}_{sp}(\R^N)$
is a locally bounded, local weak solution of~\eqref{PDE} in the sense of Definition~\ref{def:loc-sol} that satisfies
$$
 u\in W^{\sigma,p+\delta-1 }_{\rm loc}(\Omega),
$$
then for any $0<r<R$, $d\in (0,\tfrac14 (R-r)]$,
$B_{R+d}\equiv B_{R+d}(x_o)\Subset\Omega$, 
$\eta\in \mathfrak Z_{r,R}(x_o)$, and  any  step size $0<|h|\le d$ we have the energy inequality
\begin{align}\label{est:energy-p>2}\nonumber
    \iint_{B_R\times B_R} &
    \frac{\big|V_{\frac{p+\delta-1}{p}}(\btau_hu(x)) \eta(x) - V_{\frac{p+\delta-1}{p}}(\btau_hu(y)) \eta(y)\big|^{p}}{|x-y|^{N+sp}} \,\dx\dy \\\nonumber
    &\le 
    \frac{CR^\epsilon }{(R-r)^2}
    [u]_{W^{\sigma ,p+\delta -1}(B_{R+d})}^{p-2}
    \bigg[\frac{1}{\epsilon}
    \int_{B_R} |\btau_h u|^{p+\delta-1} \,\dx\bigg]^{\frac{\delta+1}{p+\delta-1}}\\\nonumber
    &\phantom{\le\,}+\frac{C}{(1-s)R^{sp}}\Big(\frac{R}{R-r}\Big)^{N+sp}\mathcal T^{p-2}\int_{B_{R}} |\btau_h u|^{\delta+1}\,\dx
    \\ 
    &\phantom{\le\,} +
   \frac{C |h|}{R^{sp+1}} 
    \Big(\frac{R}{R-r}\Big)^{N+sp+1}
    \mathcal T^{p-1}\int_{B_{R}} |\btau_hu|^\delta \,\dx,
\end{align}
where we abbreviated
\begin{equation*}
    \mathcal T
    :=
    \|u\|_{L^{\infty}(B_{R+d})}+\mathrm{Tail}(u;R+d),
\end{equation*}
and
\begin{equation*}
    \epsilon
    :=
    \sigma(p-2) - (sp-2) 
    \in 
    \big(\max\{0,2-sp\},2\big)
\end{equation*}
The dependence of the constant $C$ on $\delta$ is of the form 
$C=\widetilde{C}(N,p)s^{-1}8^{\delta}\delta^p$.
\end{proposition}

\begin{proof}
The argument is similar to the proof of Proposition~\ref{prop:energy}. The only difference is the estimate of $\mathbf I_2$. For this term we proceed as follows.
We start from the first line of~\eqref{est:I_2} and apply H\"older's inequality with exponents $\frac{p+\delta-1}{\delta+1}$ and $\frac{p+\delta-1}{p-2}$. However, we now decompose the exponent $N+sp-2$ of the numerator into the form 
\begin{align*}
    N+sp-2
    =
    \big( N+\sigma(p+\delta-1)\big) \frac{p-2}{p+\delta-1} + 
    \bigg(N-\epsilon\, \frac{p+\delta-1}{\delta+1}\bigg)\frac{\delta+1}{p+\delta-1}. 
\end{align*}
This leads to
\begin{align*}
    \mathbf I_2
    &\le
    \frac{C(p)}{(R-r)^2}
    \iint_{K_R}
     \frac{(|U_h(x,y)|+|U(x,y)|)^{p-2} (|\btau_hu(x)|+|\btau_hu(y)|)^{\delta+1}}{|x-y|^{N+sp-2}}
     \,\dx\dy\\
     &\le 
     \frac{C(p)}{(R-r)^2}
     \Bigg[
     \iint_{K_R} \frac{(|u_h(x)-u_h(y)|+|u(x)-u(y)|)^{p+\delta-1}}
     {|x-y|^{N + \sigma(p +\delta-1)}} \,\dx\dy 
     \Bigg]^\frac{p-2}{p+\delta-1}\\
     &\qquad\qquad\qquad\cdot
     \Bigg[
     \iint_{K_R} \frac{(|\btau_hu(x)|+|\btau_hu(y)|)^{p+\delta-1}}
     {|x-y|^{N-\epsilon\frac{p+\delta-1}{\delta+1}}}\,\dx\dy 
     \Bigg]^\frac{\delta+1}{p+\delta-1}.
\end{align*}
In the first integral, we eliminate the dependence on $h$ in the integrand by enlarging the domain of integration. We estimate the integrand of the second integral using the convexity of $t\mapsto t^{p+\delta -1}$. Afterwards we use the symmetry of the resulting integrand in the arguments $x,y$ and Lemma~\ref{int-sing}, respectively Remark~\ref{no-int-sing} with exponent $\beta=\epsilon\frac{p+\delta-1}{\delta+1}$. In this way we get
\begin{align*}
    \mathbf I_2
    &\le
    \frac{2^{\delta}C(p)}{(R-r)^2}
    [u]_{W^{\sigma ,p+\delta -1}(B_{R+d})}^{p-2} \\
    &\qquad\qquad\cdot
    \Bigg[ 
    \int_{B_R}\bigg[ 
    \underbrace{\int_{B_R}\frac{1}{|x-y|^{N-\epsilon\frac{p+\delta-1}{\delta+1}}}\dy}_{\le \frac{\om_N}{\epsilon} \frac{\delta+1}{p+\delta-1}(2R)^{\epsilon\frac{p+\delta-1}{\delta+1}}}\bigg]
    |\btau_hu(x)|^{p+\delta-1}\dx
    \Bigg]^\frac{\delta+1}{p+\delta-1}\\
     &\le
     \frac{2^{\delta}C(p,N) }{(R-r)^2}
     [u]_{W^{\sigma ,p+\delta -1}(B_{R+d})}^{p-2}
     \bigg[ \frac{R^{\epsilon\frac{p+\delta-1}{\delta+1}}}{\epsilon} \int_{B_R} |\btau_hu(x)|^{p+\delta-1}\dx\bigg] ^\frac{\delta+1}{p+\delta-1}.
\end{align*}
This essentially gives the first term on the right-hand side of \eqref{est:energy-p>2}. Whereas the other terms are exactly as in the proof of Proposition~\ref{prop:energy}
\end{proof}

\section{The case $s\in (0,\frac{p-2}{p}]$}\label{sec:sub}

In this section we consider the range $s\in (0,\frac{p-2}{p}]$, which can happen only if $p>2$. We establish the statement of Theorem~\ref{thm:Wgq} in \S\,\ref{sec:frac-smalls} and that of Theorem~\ref{thm:Hoelder-subcritical} in \S\,\ref{sec:holder-smalls}. 

\subsection{Fractional differentiability}\label{sec:frac-smalls}
As already mentioned, the aim of this subsection is to prove the almost $W^{\frac{sp}{p-2},q}_{\loc}$-regularity for local weak solutions of the fractional $p$-Laplace equation as stated in Theorem~\ref{thm:Wgq} for arbitrary $q\ge p$. The argument consists of two steps: We first fix the integrability order $q\ge p$ and set up an iteration scheme to raise the differentiability order to any number less than $\frac{sp}{p-2}$; then, we set up another iteration scheme to raise the integrability order from $p$ to any $q\ge p$. 

Lemma~\ref{lem:frac-impr} sets out the first step; it guarantees a small but quantifiable gain of fractional differentiability. 
The idea is, roughly speaking, to use the energy inequality in Proposition~\ref{prop:energy-s} and obtain estimates for $\|\btau_h (\btau_h u)\|_{L^q}$ in terms of a power of the increment $|h|$. Such estimates, in view of the inequality~\eqref{est-1st-diffquot<1} from Lemma~\ref{lem:Domokos}, result in a similar bound of $\|\btau_h u\|_{L^q}$ by the same power of $|h|$.
As a consequence, this ensures better quantifiable fractional differentiability due to the embedding properties of Nikol'skii spaces in fractional Sobolev spaces.
Subsequently, the improvement of the fractional differentiability achieved in Lemma~\ref{lem:frac-impr} is iterated in Lemma~\ref{lem:frac-iter} to approach the order $\frac{sp}{p-2}$.
The second step is to use what has been proven in the first step and iteratively raise the integrability order from $p$ to $q$; see Theorem~\ref{lem:Wgq}. 
 
\begin{lemma}\label{lem:frac-impr}
Let $p\in(2,\infty)$, $s\in(0,\frac{p-2}{p}]$, and  $q\in [p,\infty)$. Further, let 
\begin{equation}\label{def:gamma-sigma}
    \sigma
    \in 
    \bigg(\max\Big\{\frac{sp-2}{p-2},0\Big\}, \,\frac{sp}{p-2}\bigg)
    \quad\mbox{and}\quad 
    \beta:=
    \Big(1-\frac{p-2}{q}\Big)\sigma + \frac{sp}{q}.
\end{equation}
Then, whenever $u\in W^{s,p}_{\rm loc}(\Omega)\cap L^{p-1}_{sp}(\R^N)$ is a locally bounded, local weak solution of~\eqref{PDE} in the sense of Definition~\ref{def:loc-sol} that satisfies
$$
    u\in W^{\sigma,q}_{\loc}(\Om),
$$
we have 
\[
    u\in W^{\alpha,q}_{\loc}(\Om)
    \qquad\mbox{for any $\alpha\in (\sigma,\beta)$.}
\]
Furthermore, there exists a constant $C=C(N, p, s, q,\sigma,  \alpha)$, so that for any  ball $B_{R}\equiv B_{R}(x_o)\Subset \Omega$ and for any $r\in(0,R)$ we have the quantitative estimate
\begin{align*}
    [u]_{W^{\alpha,q}(B_r)}^q
    \le 
    \frac{C}{R^{\alpha q}}
    \Big(\frac{R}{R-r}\Big)^{N+2q+2} 
    \mathbf{K}_\sigma^q .
\end{align*}
Here, we  used the short-hand notation 
\begin{align}\label{def:Ksig}
    \mathbf{K}_\sigma^q
    :=
    R^{\sigma q} [u]^q_{W^{\sigma,q}(B_{R})} +
    R^{N} \big(\|u\|_{L^\infty(B_{R})} + \Tail(u;R)\big)^q .
\end{align}
\end{lemma}

\begin{proof}
We apply the energy inequality from Proposition~\ref{prop:energy-s} with
\begin{equation}\label{def:epsilon}
    \epsilon 
    :=
    \sigma(p-2) - (sp-2) 
    \in 
    \big(\max\{0,2-sp\},2\big)
\end{equation}
and $\delta =q-p +1$.
Furthermore, in the application we replace $r,R,d$ by $\tilde r=\frac17(5r+2R)$,
$\widetilde R=\frac17(r+6R)$, and $d=\frac14(\widetilde R-\tilde r)=\frac17(R-r)$. For later use we note that $\widetilde  R+d=R$,
\begin{equation}\label{rtilde} 
    \frac1{\widetilde R}
    =
    \frac{7}{r+6R}
    \le 
    \frac{7}{6R}
\end{equation}
and 
\begin{equation}\label{s-l-rtilde} 
    \frac{\widetilde R}{\widetilde R-\tilde r}
    =
    \frac{r+6R}{r+6R-(5r+2R)}
    =
    \frac{r+6R}{4(R-r)}
    \le \frac{7}{4}
    \frac{R}{R-r}.
\end{equation}
This allows us to replace $\frac{\widetilde R}{\widetilde R-\tilde r}$ by $\frac{R}{R-r}$ and $\frac1{\widetilde R}$ by $\frac1R$ when applying Proposition \ref{prop:energy-s} apart from a multiplicative constant depending only on $N$ and $p$.
By $\eta\in C^1_0(B_{\frac12 (\widetilde{R}+\tilde{r})})$ we denote a cut-off function in $\frak Z_{\tilde{r}, \widetilde R} $, cf.~Definition \ref{Def:Z}, satisfying $\eta =1$ on $B_{\tilde{r}}$ and $|\nabla\eta|\le \frac{4}{\widetilde R-\tilde r}=\frac{7}{R-r}$.
Noting that $u\in W^{\sigma, q}(B_{R})$ by assumption, the application of Proposition \ref{prop:energy-s} yields that
$$
    \mathbf{I}
    :=
    \big[\eta V_\frac{q}{p}(\btau_hu) \big]_{W^{s,p}(B_{\widetilde R})}^p
    \le
    \mathbf{I}_1 + \mathbf{I}_2 + \mathbf{I}_3,
$$
where 
\begin{align*}
    \mathbf{I}_1
    &:=
    \frac{C}{R^{sp-\sigma (p-2)}}\Big(\frac{R}{R-r}\Big)^{2}
     [u]^{p-2}_{W^{\sigma,q}(B_R)}  
     \bigg[\int_{B_{\widetilde R}}\frac{|\btau_hu|^{q}}{\epsilon} \,\dx\bigg] ^{\frac{q-(p-2)}{q}},\\
    \mathbf{I}_2
    &:=
    \frac{C}{(1-s)R^{sp}}\Big(\frac{R}{R-r}\Big)^{N+sp} \mathcal T^{p-2}\int_{B_{\widetilde R}}|\btau_h u|^{q-(p-2)}\,\dx,
    \\ 
    \mathbf{I}_3
    &:=
   \frac{C |h|}{R^{sp+1}} 
    \Big(\frac{R}{R-r}\Big)^{N+sp+1}
     \mathcal T^{p-1}\int_{B_{\widetilde R}} |\btau_hu|^{q-(p-1)}  \,\dx
\end{align*}
and 
\begin{equation*}
    \mathcal T:= \|u\|_{L^{\infty}(B_{R})}+\mathrm{Tail}(u;  R),
\end{equation*}
and the constant $C$ is of the from 
\begin{equation}\label{def:tilde-C}
    C
    =
    \widetilde{C}(N,p)\frac{8^qq^p}{s}. 
\end{equation}
In view of Lemma~\ref{lem:N-FS} and noting again that $u\in W^{\sigma, q}(B_{R})$, we obtain
\begin{align*}
    \int_{B_{\widetilde R}} |\btau_hu|^{q} \,\dx
    &\le 
    C\,|h|^{\sigma q}  
    \bigg[ (1-\sigma)[u]^q_{W^{\sigma,q}(B_R)} + 
    \frac{1}{\sigma R^{\sigma q}} \Big(\frac{R}{R-r}\Big)^q
    \int_{B_{ R}}|u|^{q} \,\dx\bigg] \\
    &\le 
    \frac{C}{\sigma}\Big(\frac{R}{R-r}\Big)^q 
    \frac{|h|^{\sigma q}}{R^{\sigma q}}
    \bigg[ R^{\sigma q}[u]^q_{W^{\sigma,q}(B_R)} + 
    \int_{B_{ R}}|u|^{q} \,\dx\bigg] \\
    &\le 
    \frac{C}{\sigma}\Big(\frac{R}{R-r}\Big)^q 
    \frac{|h|^{\sigma q}}{R^{\sigma q}} \,
    \mathbf{K}^q_\sigma,
\end{align*}
where $C=C(N,q)$. To obtain the last line we used the fact that $u$ is bounded and the definition of $\mathbf K_\sigma$.
This 
allows us to estimate $ \mathbf{I}_1$ by 
\begin{align*}
    \mathbf{I}_1
    &\le 
    \frac{C\,R^{\sigma(p-2)-sp}}{(\sigma \epsilon)^\frac{q-(p-2)}{q}} \Big(\frac{R}{R-r}\Big)^{2}
    [u]^{p-2}_{W^{\sigma,q}(B_R)}
     \bigg[\Big(\frac{R}{R-r}\Big)^q 
    \frac{|h|^{\sigma q}}{R^{\sigma q}} \,
    \mathbf{K}^q_\sigma \bigg]^{\frac{q-(p-2)}{q}}\\
    &\le 
    \frac{C}{R^{sp}} \frac{|h|^{\sigma(q-(p-2))}}{R^{\sigma(q-(p-2))}} \Big(\frac{R}{R-r}\Big)^{q-p+4}
    \mathbf{K}^q_\sigma ,
\end{align*}
where, in view of~\eqref{def:tilde-C}, the constant $C$ is of the form
\begin{equation*}
    C
    =
    \frac{\widetilde C(N,p,q)}{s(\sigma \epsilon)^{\frac{q-(p-2)}{q}}},
\end{equation*}
Similarly, using H\"older's inequality to raise the power of $|\btau_hu|$ from $q-(p-2)$ to $q$, we have
\begin{align*}
    \mathbf{I}_2
    &\le 
    \frac{C}{(1-s)R^{sp}}\Big(\frac{R}{R-r}\Big)^{N+sp} 
    R^{\frac{N(p-2)}{q}}\mathcal T^{p-2}
    \bigg[\int_{B_{\widetilde R}}|\btau_h u|^{q}\,\dx \bigg]^{\frac{q-(p-2)}{q}}
    \\ 
    &\le 
    \frac{C}{(1-s)R^{sp}}\Big(\frac{R}{R-r}\Big)^{N+sp} 
    \mathbf{K}^{p-2}_\sigma
    \bigg[\frac{1}{\sigma}\Big(\frac{R}{R-r}\Big)^q 
    \frac{|h|^{\sigma q}}{R^{\sigma q}} \,
    \mathbf{K}^q_\sigma\bigg]^{\frac{q-(p-2)}{q}} \\ 
    &\le 
    \frac{C}{(1-s)\sigma^\frac{q-(p-2)}{q}R^{sp}}
    \frac{|h|^{\sigma(q-(p-2))}}{R^{\sigma(q-(p-2))}}
    \Big(\frac{R}{R-r}\Big)^{N+q+2} 
    \mathbf{K}^q_\sigma.
\end{align*}
The argument that led to the estimate of $\mathbf{I}_2$
can also be applied to $\mathbf{I}_3$. In fact, we obtain
\begin{align*}
    \mathbf{I}_3
    &\le 
    \frac{C}{R^{sp}}
   \frac{ |h|}{R} 
    \Big(\frac{R}{R-r}\Big)^{N+sp+1}
    R^{\frac{N(p-1)}{q}}
    \mathcal T^{p-1}
    \bigg[\int_{B_{\widetilde R}} |\btau_hu|^{q} \,\dx\bigg]^{\frac{q-(p-1)}{q}}\\ 
    &\le 
    \frac{C}{R^{sp}}
    \frac{|h|}{R}\Big(\frac{R}{R-r}\Big)^{N+sp+1} 
    \mathbf{K}^{p-1}_\sigma
    \bigg[\frac{1}{\sigma}\Big(\frac{R}{R-r}\Big)^q 
    \frac{|h|^{\sigma q}}{R^{\sigma q}} \,
    \mathbf{K}^q_\sigma\bigg]^{\frac{q-(p-1)}{q}} \\ 
    &\le 
    \frac{C}{(1-s)\sigma^\frac{q-(p-1)}{q}R^{sp}}
    \frac{|h|^{\sigma(q-(p-2))}}{R^{\sigma(q-(p-2))}}
    \Big(\frac{R}{R-r}\Big)^{N+q+2} 
    \mathbf{K}^q_\sigma.
\end{align*}
Collecting the estimates above, we obtain that
\begin{align*}
    \mathbf{I}
    \le 
    \frac{C}{R^{sp}}
    \frac{|h|^{\sigma(q-(p-2))}}{R^{\sigma(q-(p-2))}}
    \Big(\frac{R}{R-r}\Big)^{N+q+2} 
    \mathbf{K}^q_\sigma,
\end{align*}
for a constant $C$  that has a similar structure as the one from the estimate for $\mathbf I_1$, which means that 
\begin{equation*}
    C
    =
    \frac{\widetilde C(N,p,q)}{s(1-s)(\sigma \epsilon)^{1-\frac{(p-2)}{q}}}.
\end{equation*}
To bound $\mathbf I$ from below, we apply Lemma~\ref{lem:N-FS} to $v=\eta V_\frac{q}{p}(\btau_hu)$ on $B_{\widetilde R}$ with $(q,\gm, d)$ replaced by $(p,s,d=\frac17(R-r))$.  Taking into account that $d\le R$, $\eta\le 1$, we have
\begin{align}\label{est:V_delta+p-1}
    \int_{B_{\widetilde R-d}} &
    \big|\eta \btau_\lambda\big(V_{\frac{q}{p}}(\btau_hu) \big)\big|^p \,\dx \nonumber\\
    &\le
    C |\lambda|^{sp}
     \bigg[(1-s)\mathbf I +
    \bigg(\frac{\widetilde R^{(1-s)p}}{d^p} +\frac{1}{sd^{sp}}\bigg)
    \int_{B_{\widetilde R}} \big|\eta V_{\frac{q}{p}}(\btau_hu) \big|^p \,\dx\bigg] \nonumber\\
    &\le 
    C|\lambda|^{sp} \bigg[(1-s)\mathbf I +
    \frac{R^{(1-s)p}}{s d^p} 
    \int_{B_{\widetilde R}} |\btau_hu|^{q} \,\dx\bigg]\nonumber \\
    &\le 
    C|\lambda|^{sp} \bigg[(1-s)\mathbf I +
    \frac{R^{(1-s)p}}{s d^p} 
    \frac{1}{\sigma}\Big(\frac{R}{R-r}\Big)^q 
    \frac{|h|^{\sigma q}}{R^{\sigma q}} \,
    \mathbf{K}^q_\sigma\bigg]\nonumber \\
    &\le 
    C\frac{|\lambda|^{sp} }{R^{sp}}
    \frac{|h|^{\sigma(q-(p-2))}}{R^{\sigma(q-(p-2))}}
    \Big(\frac{R}{R-r}\Big)^{N+q+2} 
    \mathbf{K}^q_\sigma
\end{align}
for any $|\lambda|\le d$. Since the constant from the estimate of $\mathbf I$ is in turn multiplied by $1-s$, the $(1-s)^{-1}$ dependency cancels out in the final constant. Thus, $C$ takes the form
$$
    C
    =
    \frac{\widetilde C(N,p,q)}{s(\sigma \epsilon)^{1-\frac{p-2}{q}}}.
$$
We choose $\lambda=h$ in \eqref{est:V_delta+p-1}. Next we  observe that
\begin{align*}
    \big|\btau_h\big(V_{\frac{q}{p}}(\btau_hu) \eta\big)\big|
    =
    \big|\btau_h\big(V_{\frac{q}{p}}(\btau_hu) \big)\big| 
    \ge 
    |\btau_h(\btau_h u)|^{\frac{q}{p}}
    \qquad \mbox{in $B_{r}$.}
\end{align*}
Here, we used  $\eta\equiv 1$ in $B_{\tilde r}=B_{r+2d}$ and $q\ge p$. 
Hence the left-hand side of \eqref{est:V_delta+p-1} is bounded from below by
$
\int_{B_r}|\btau_h(\btau_h u)|^{q}\,\dx.
$
Therefore, collecting these estimates in \eqref{est:V_delta+p-1} we conclude
\begin{align*} 
    \int_{B_{r}} \big| \btau_h(\btau_hu)\big|^{q}\,\dx
    \le 
    C\frac{|h|^{sp+\sigma(q-(p-2))}}{R^{sp+\sigma(q-(p-2))}}
    \Big(\frac{R}{R-r}\Big)^{N+q+2} 
    \mathbf{K}^q_\sigma 
\end{align*}
for any $0<|h|\le d$. At this point, we apply \eqref{est-1st-diffquot<1} from Lemma~\ref{lem:Domokos} with $\gm$ replaced by $\beta$ defined in
\eqref{def:gamma-sigma}
and  
$$
    M^q
    =
    \frac{C}{R^{sp+\sigma(q-(p-2))}}
    \Big(\frac{R}{R-r}\Big)^{N+q+2} 
    \mathbf{K}^q_\sigma.
$$
Note that $\beta <1$, since
$$
    \beta=\frac{sp+\sigma(q-(p-2))}{q}
    <
    \frac{sp+\frac{sp}{p-2}(q-(p-2))}{q}
    =
    \frac{sp}{p-2}\le 1.
$$
The application of Lemma~\ref{lem:Domokos} yields
\begin{align*} 
    \int_{B_{r}} | \btau_h u|^{q}\,\dx
    &\le 
    C(q)|h|^{\beta q}
    \bigg[
    \Big(\frac{M}{1-\beta}\Big)^q 
    + 
    \frac{1}{d^{q\beta}} \int_{B_R} |u|^q \,\dx \bigg]\\ 
    &\le 
    \frac{C}{(1-\beta)^q}\frac{|h|^{sp+\sigma(q-(p-2))}}{R^{sp+\sigma(q-(p-2))}}
    \Big(\frac{R}{R-r}\Big)^{N+q+2} 
    \mathbf{K}^q_\sigma \\
    &\le 
    C\frac{|h|^{sp+\sigma(q-(p-2))}}{R^{sp+\sigma(q-(p-2))}}
    \Big(\frac{R}{R-r}\Big)^{N+q+2} 
    \mathbf{K}^q_\sigma
\end{align*}
for any $0<|h|\le \frac12 d$, where $C$ takes the form
$$
     C
     =
     \frac{\widetilde C(N,p,q)}{s(1-\beta)^q(\sigma \epsilon)^{1-\frac{p-2}{q}}}.
$$
To obtain the last line we used the $L^\infty$-bound for $u$ and the definition of $\mathbf K_\sigma$. Moreover, we  replaced the exponent $\beta q$ by $sp+\sigma(q-(p-2))$. This could also have been done for the denominator $(1-\beta)^q$. However, we kept using $\beta$ in the expression of $C$ for simplicity.
Next, we apply Lemma~\ref{lem:FS-N} with 
$$
    M^q
    =
   \frac{C}{R^{sp+\sigma(q-(p-2))}}
    \Big(\frac{R}{R-r}\Big)^{N+q+2} 
    \mathbf{K}^q_\sigma 
$$
and $\beta$ as defined  in \eqref{def:gamma-sigma}. As a result, for any $\sigma <\alpha<\beta$, we have
\begin{align*} 
    [u]_{W^{\alpha,q}(B_r)}^q
    &\le 
    C(N,q)\bigg[ 
    \frac{d^{(\beta-\alpha)q}}{\beta-\alpha} M^q +
    \frac{1}{\alpha d^{\alpha q}} \int_{B_R} |u|^q \,\dx
    \bigg] \\
    &\le
    \frac{C}{R^{\alpha q}}
    \Big(\frac{R}{R-r}\Big)^{N+2q+2} 
    \mathbf{K}^q_\sigma ,
\end{align*}
where the constant $C$ takes the form
\begin{equation}\label{def:C-final}
    C=\frac{\widetilde C(N,p,q)}{s(\beta -\alpha)(1-\beta)^q\sigma (\sigma \epsilon)^{1-\frac{p-2}{q}}}.
\end{equation}
This proves the claim. 
\end{proof}

In the following lemma we iterate the improvement  in fractional differentiability obtained in Lemma~\ref{lem:frac-impr}.

\begin{lemma}\label{lem:frac-iter}
Let $p\in(2,\infty)$, $s\in(0,\frac{p-2}{p}]$, and  $q\in [p,\infty)$. Further, let 
\begin{equation*}
    \sigma
    \in 
    \bigg(\max\Big\{\frac{sp-2}{p-2},0\Big\}, \,\frac{sp}{p-2}\bigg).
\end{equation*}
Then for any locally bounded, local weak solution of~\eqref{PDE} in the sense of Definition~\ref{def:loc-sol} that satisfies
\[
    u\in W^{\sigma,q}_{\loc}(\Om),
\]
we have 
\[
    u\in W^{\gamma,q}_{\loc}(\Om)
    \quad\mbox{for any $\gamma\in \big(\sigma,\frac{sp}{p-2}\big)$.}
\]
Moreover, there exist  constants $C=C(N,p,s,q,\sigma,\gamma)$ and $\kappa=\kappa (N,p,s,q,\gamma)\ge 1$
such that for any ball $B_{R}\equiv B_{R}(x_o)\Subset \Omega$ and any $r\in(0,R)$, we have
\begin{align*}
    [u]_{W^{\gamma,q}(B_r)}^q
    \le 
    \frac{C}{R^{\gamma q}}
    \Big(\frac{R}{R-r}\Big)^{\kappa} 
    \mathbf{K}_\sigma^q ,
\end{align*}
where $\mathbf{K}_\sigma$ is defined in~\eqref{def:Ksig}.
\end{lemma}

\begin{proof}
Let $\tilde\gamma=\frac{1}{2}(\gamma + \frac{sp}{p-2})\in(\gamma,\frac{sp}{p-2})$. For $i\in\N_0$ we define the recursive sequence
$$
    \sigma_0=\sigma,
    \qquad\sigma_{i+1}=\Big(1-\frac{p-2}{q}\Big)\sigma_i + \tilde\gamma\frac{p-2}{q}.
$$ 
For the $\sigma_i$ we have the explicit representation
\begin{equation*}
    \sigma_i=\tilde\gamma -\Big( 1-\frac{p-2}{q}\Big)^i(\tilde\gamma -\sigma),
\end{equation*}
so that $\sigma_i\uparrow \tilde\gamma>\gamma$ as $i\to\infty$. 
Moreover, for $i\in\N_0$ we define the sequence of radii 
\begin{equation}\label{rho-i}    
    \rho_i:= r+\frac{1}{2^{i}}(R-r).
\end{equation}
First, we observe that
\begin{equation}\label{rho-i-1}    
    R
    \ge 
    \rho_{i-1} 
    =
    \frac{R}{2^{i-1}}+ r\Big( 1-\frac{1}{2^{i-1}}\Big)
    >
    \frac{R}{2^{i-1}},
\end{equation}
and 
\begin{equation}\label{rho-i-2}
    \frac{\rho_{i-1}}{\rho_{i-1}-\rho_{i}}
    =
    \frac{2^{i}r+2(R-r)}{R-r}
    <
    2^{i}\frac{R}{R-r}.
\end{equation}
In order to express the following estimates in a more compact form, we define
$$
    \mathcal T_i
    :=
    \|u\|_{L^\infty(B_{\rho_{i}})} + \Tail(u;\rho_{i}),\qquad \mathcal T_0\equiv\mathcal T.
$$
As a consequence of Lemma~\ref{lem:t} and \eqref{rho-i-1} we 
obtain
\begin{align}\label{est:tail-standard}
    \Tail(u;\rho_{i-1})^{p-1}
    &\le
    C(N) \Big(\frac{R}{\rho_{i-1}}\Big)^N
    \mathcal T^{p-1} 
    \le
    C(N)2^{iN}
    \mathcal T^{p-1}.
\end{align}
Now we apply Lemma~\ref{lem:frac-impr} with  $(\alpha, \beta, \sigma,r,R)$ replaced by
$\big(\sigma_{i}, \beta_i,\sigma_{i-1},\rho_{i},\rho_{i-1}\big)$, where 
$$
    \beta_i
    :=
    \Big( 1-\frac{p-2}{q}\Big)\sigma_{i-1}+\frac{sp}{q}
    >
    \sigma_i
$$
and obtain
\begin{align*}
     \rho_{i-1}^{\sigma_{i}q} 
    [u]_{W^{\sigma_{i},q}(B_{\rho_{i}})}^q 
    \le 
    C_i
    \Big(\frac{\rho_{i-1}}{\rho_{i-1}-\rho_{i}}\Big)^{N+2q+2} 
    \Big[\rho_{i-1}^{\sigma_{i-1} q} 
    [u]^q_{W^{\sigma_{i-1},q}(B_{\rho_{i-1}})} +
    \rho_{i-1}^{N} \mathcal T_{i-1}^q \Big] .
\end{align*}
The application  is permitted if $[u]_{W^{\sigma_{i-1},q}(B_{\rho_{i-1}})}<\infty$ is fulfilled.  Next, we use \eqref{rho-i}, \eqref{rho-i-1}, and \eqref{est:tail-standard} to estimate the corresponding terms in the right-hand side. Whereas for the left-hand side we use  $\rho_i<\rho_{i-1}$. In this way we get
\begin{align*}
    \rho_{i}^{\sigma_{i}q} &
    [u]_{W^{\sigma_{i},q}(B_{\rho_{i}})}^q\\
    &\le 
    C_i2^{[(N+2)q+2]i}
    \Big(\frac{R}{R-r}\Big)^{N+2q+2} 
    \Big[ \rho_{i-1}^{\sigma_{i-1} q} [u]^q_{W^{\sigma_{i-1},q}(B_{\rho_{i-1}})} +
    R^{N} \mathcal T^q \Big].
\end{align*}
Prior to an iteration based on the above estimate, we examine the dependencies of the constant $C_i$ from Lemma~\ref{lem:frac-impr}, which is presented in~\eqref{def:C-final}. To do this, we have to calculate and estimate the corresponding factors of the denominator. Let us start with
$ \beta_i-\sigma_i$ which plays the role of $\beta -\alpha$. In fact, we have
\begin{align*}
    \beta_i-\sigma_i
    &=
    \Big( 1-\frac{p-2}{q}\Big)\sigma_{i-1} +\frac{sp}{q}
    -
    \Big[ \Big( 1-\frac{p-2}{q}\Big)\sigma_{i-1}
    +\tilde\gamma \frac{p-2}{q}\Big]\\
    &=
    \frac{sp}{q}-\tilde\gamma \frac{p-2}{q}
    = \frac{p-2}{q}\Big(\frac{sp}{p-2}
    -\tilde\gamma \Big)
    =
    \frac{p-2}{2q}\Big(\frac{sp}{p-2}
    - \gamma\Big).
\end{align*}
Next we have to estimate $1-\beta_i$, playing the role of $1-\beta$.
Using the explicit expression for $\sigma_{i-1}$, the definition of $\tilde\gamma$, the assumption on $s$, i.e.~$\frac{sp}{p-2}\le 1$, we have
\begin{align*}
    1-\beta_i
    &=
    1-\bigg[\Big( 1-\frac{p-2}{q}\Big)\sigma_{i-1}+\frac{sp}{q}\bigg]\\
    &>
    1-\frac{sp}{q}-\Big( 1-\frac{p-2}{q}\Big)\tilde\gamma \\
    &=
    1-\frac{sp}{p-2} +
    \frac12\Big(\frac{sp}{p-2}-\gamma\Big)\Big( 1-\frac{p-2}{q}\Big) \\
    &\ge 
    \frac1q\Big(\frac{sp}{p-2}-\gamma\Big).
\end{align*}
Finally, using that $\sigma_i\ge \sigma$ we obtain the estimate for $\epsilon_i$, which takes over the role of $\epsilon$. In fact, we have
\begin{align*}
    \epsilon_i
    &=
    \sigma_i (p-2)-(sp-2)\ge (p-2)\Big( \sigma -\frac{sp-2}{p-2}\Big).
\end{align*}
Therefore, the constant $C_i$ from \eqref{def:C-final} can be estimated by
\begin{align*}
    C_\ast=\frac{\widetilde C(N,p,q)}{s\big(\frac{sp}{p-2}-\gamma\big)^{q+1}\sigma^2\big( \sigma -\frac{sp-2}{p-2}\big)}.
\end{align*}
Iterating the above  inequality for $[u]_{W^{\sigma_{i},q}(B_{\rho_{i}})}$ we obtain
\begin{align*}
    [u]_{W^{\sigma_i,q}(B_{\rho_i})}^q
    &\le 
    \frac{iC^i_\ast\, 2^{[(N+2)q+2]i!}}{\rho_{i}^{\sigma_i q}}
    \Big(\frac{R}{R-r}\Big)^{(N+2q+2)i} 
    \mathbf{K}_\sigma^q \\
    &\le 
    \frac{iC_\ast^i\, 2^{[(N+3)q+2]i!}}{R^{\sigma_i q}}
    \Big(\frac{R}{R-r}\Big)^{(N+2q+2)i} 
    \mathbf{K}_\sigma^q 
\end{align*}
for any $i\in\N$. By $i_o\in\N$ we denote the smallest integer such that $\sigma_{i_o}\ge \gamma$. 
More explicitly, we have
\begin{equation*}
    i_o
    := 
    \Biggl\lceil
    \frac{\ln \frac{\tilde\gamma -\sigma}{\tilde\gamma -\gamma}}{\ln \frac{q}{q-(p-2)}}\Biggr\rceil
    \le 
    \frac{\ln \frac{1}{\tilde\gamma -\gamma}}{\ln \frac{q}{2}} + 1 
    \le 
    \frac{\ln \frac{2}{\frac{sp}{p-2} -\gamma}}{\ln \frac{q}{2}} + 1 
    =
    \frac{\ln \frac{q}{\frac{sp}{p-2} -\gamma}}{\ln \frac{q}{2}}.
\end{equation*}
This fixes the dependencies of $i_o=i_o(s,p,q,\gamma)$. Note that $i_o$ blows up as $\gamma \uparrow \frac{sp}{p-2}$. For the fractional $W^{\gamma,q}$-norm of $u$ on $B_r$ the last inequality with $i=i_o$ implies
\begin{align*}
    [u]_{W^{\gamma,q}(B_{r})}^q
    &\le 
    (2R)^{(\sigma_{i_o}-\gamma)q} 
    [u]_{W^{\sigma_{i_o},q}(B_{\rho_{i_o}})}^q\\
    &
    \le 
    \frac{i_oC_*^{i_o}\, 2^{[(N+3)q+2]i_o!}}{R^{\gamma q}}
    \Big(\frac{R}{R-r}\Big)^{(N+2q+2)i_o} 
    \mathbf{K}_\sigma^q.
\end{align*}
Letting $C=i_oC^{i_o}_\ast\, 2^{[(N+3)q+2]i_o!}$
and $\kappa:= (N+2q+2)i_o$ the claim follows.
\end{proof}

Now we are able to prove Theorem~\ref{thm:Wgq}, i.e. ~the almost $W^{\frac{sp}{p-2},q}_{\loc}$-regularity for any $q\ge p$ of locally bounded, local weak solutions to the fractional $p$-Laplace. In fact, the following Theorem~\ref{lem:Wgq} implies the Theorem~\ref{thm:Wgq} as a special case when choosing $r=\frac12 R$. In the proof of Theorem~\ref{lem:Wgq} we use the fact that $u\in W^{\gamma,\theta}$ implies $u\in W^{\frac{\gamma\theta}{q},q}$ for $q>\theta\ge p$ provided $u$ is bounded. This observation allows us to increase the integrability exponent after decreasing the order of fractional differentiability.  This is the point where Lemma~\ref{lem:frac-iter} comes into play, which in turn allows to increase the order of fractional differentiability with the larger integrability exponent. 
In the case $s\in(0,\frac{2}{p}]$ we can achieve the final result in one step, while in the case $s\in(\frac{2}{p}, \frac{p-2}{p}]$ the previously described argument must be iterated in order to obtain the result.

\begin{theorem}\label{lem:Wgq}
Let $p\in(2,\infty)$, $s\in(0,\frac{p-2}{p}]$. Then, whenever $u\in W^{s,p}_{\rm loc}(\Omega)\cap L^{p-1}_{sp}(\R^N)$ is a locally bounded, local weak solution of~\eqref{PDE} in the sense of Definition~\ref{def:loc-sol}, we have 
\[
    u\in W^{\gamma,q}_{\loc}(\Om)
    \qquad\mbox{for any $q\in [p,\infty)$, and $\gamma\in \big[s,\frac{sp}{p-2}\big)$.}
\]
Moreover, there exist constants $C=C(N,p,s,q,\gamma)$ and $\kappa=\kappa(N,p,s,q,\gamma)$ such that for any ball $B_{R}\equiv B_{R}(x_o)\Subset \Omega$ and any $r\in(0,R)$, we have 
\begin{align*}
    [u]_{W^{\gamma,q}(B_r)}^q
    \le 
    \frac{C}{R^{\gamma q}}
    \Big(\frac{R}{R-r}\Big)^{\kappa} 
    \mathbf{K}^q.
\end{align*}
Here, we  used the short-hand notation 
\begin{align*}
    \mathbf{K}^q
    :=
    R^{[s-N(\frac{1}{p}-\frac{1}{q})]q}[u]_{W^{s,p}(B_R)}^q +
    R^{N} \big(\|u\|_{L^\infty(B_{R})} + 
    \Tail(u;R)\big)^q .
\end{align*}
Note that $C$ blows up as $\gm\uparrow \frac{sp}{p-2}$.
\end{theorem}

\begin{proof}
We consider some fixed ball ball $B_{R}\equiv B_{R}(x_o)\Subset \Omega$ and abbreviate 
\begin{equation*}
    \mathcal T
    :=
    \|u\|_{L^\infty(B_{R})} + \Tail(u;R).
\end{equation*}
Let us first consider {\bf the case} $s\in(0,\frac{2}{p}]$. 
Since $u\in L^\infty_{\rm loc}(\Omega) $, we have $u\in W^{\frac{sp}{q},q}_{\loc}(\Om)$ with the trivial estimate
\begin{align*}
    R^{sp}[u]_{W^{\frac{sp}{q},q}(B_R)}^q
    &\le 
    2^{q-p}R^{sp} \|u\|_{L^\infty(B_R)}^{q-p}
    [u]_{W^{s,p}(B_R)}^p \\
    &\le 
    2^{q-p} \Big[ 
    R^{sq-N(\frac{q}{p}-1)}[u]_{W^{s,p}(B_R)}^q +
    R^N \|u\|_{L^\infty(B_R)}^{q}
    \Big].
\end{align*}
Next, we apply Lemma~\ref{lem:frac-iter} with  $\sigma=\frac{sp}{q}$.
This particular choice of $\sigma$ is admissible, since $sp-2\le 0$ implies  $\max\{\frac{sp-2}{p-2},0\}=0$. The application of the lemma yields (note that $\sigma<s$ if $q>p$, while for $q=p$ the subsequent assertion with $\gamma=s$ is trivial)
\[
    u\in W^{\gamma,q}_{\loc}(\Om)
    \quad\mbox{for any $\gamma\in \big[s,\frac{sp}{p-2}\big)$.}
\]
Moreover, there exist $C=C(N,p,s,\gamma,q)$ and $\kappa =\kappa (N,p,s,\gamma,q)\ge 1$, such that for any $r\in(0,R)$ we have
\begin{align*}
    [u]_{W^{\gamma,q}(B_r)}^q
    &\le 
    \frac{C}{R^{\gamma q}}
    \Big(\frac{R}{R-r}\Big)^{\kappa} 
    \Big[
    R^{sp} [u]^q_{W^{\frac{sp}{q},q}(B_{R})} +
    R^{N} \mathcal T^q 
    \Big] \\
    &\le 
    \frac{C}{R^{\gamma q}}
    \Big(\frac{R}{R-r}\Big)^{\kappa} 
    \Big[
    R^{sq-N(\frac{q}{p}-1)}[u]_{W^{s,p}(B_R)}^q +
    R^{N} \mathcal T^q 
    \Big].
\end{align*}

\medskip
Next, we consider {\bf the case} $s\in (\frac{2}{p},\frac{p-2}{p}]$ (which can happen only if $p>4$). Note that now $\max\{\frac{sp-2}{p-2},0\}=\frac{sp-2}{p-2}$. For $i\in\N_0$ we define the sequence of radii 
\begin{equation*}
    \rho_i:= r+\frac{1}{2^{i}}(R-r),
\end{equation*}
and
$$
    \mathcal T_i
    :=
    \|u\|_{L^\infty(B_{\rho_{i}})} + \Tail(u;\rho_{i})
$$
and the sequence of exponents 
$$
    q_i
    :=
    \alpha^i p,
    \qquad
    \mbox{for some } \alpha\in \bigg(1,\frac{s(p-2)}{sp-2}\bigg).
$$
The exact choice of $\al$ will be made later. 
Note that $q_i=\alpha q_{i-1}$ for $i\in\N$ and $q_i\to\infty$ as $i\to\infty$ and that $\mathcal T_0=\mathcal T$. 
Now, we prove by an induction argument that 
\begin{equation}\label{induction-q}
    [u]_{W^{\gamma, q_i}(B_{\rho_i})}^{q_i} 
    \le 
    \frac{\widetilde C_i}{\rho_i^{\gamma q_i}} 
    \Big(\frac{R}{R-r}\Big)^{\tilde\kappa_i} 
    \mathbf{K}_i^{q_i},
    \qquad\mbox{for any $i\in\N_0$,}
\end{equation}
where 
\begin{align*}
    \mathbf{K}_i^{q_i}
    :=
    R^{s q_i-N(\frac{q_i}{p}-1)} [u]^{q_i}_{W^{s,p}(B_{R})} +
    R^{N} \mathcal T^{q_i}
\end{align*}
and the constants $\widetilde C_i, \tilde\kappa_i\ge 1$ depend on $N,p,s,\gamma,q$, and $i$ and will be specified at the end of the induction argument. 

First we consider {\bf the case} $i=0$. Observing that
$s>\frac{sp-2}{p-2}$, we can apply Lemma~\ref{lem:frac-iter} with  $(s,p,\rho_o,R)$ instead of $(\sigma,q,r,R)$. With  constants $\widetilde C_0$ and $\tilde\kappa_0\ge 1$
both depending only on $N$, $p$, $s$, $\gamma$, and $q$, we obtain that
\begin{align*}
    [u]_{W^{\gamma,q_0}(B_{\rho_0})}^{q_0}
    \le 
    \frac{\widetilde C_0}{R^{\gamma q_0}}
    \Big(\frac{R}{R-r}\Big)^{\tilde\kappa_0} 
    \mathbf{K}_0^{q_0} .
\end{align*}

For the {\bf induction step}
we assume that~\eqref{induction-q}$_i$
is satisfied. In particular, this means that $ [u]_{W^{\gamma, q_i}(B_{\rho_i})}$ is finite.
Since $u$ is locally bounded, this implies $u\in W^{\gamma\frac{q_{i}}{q_{i+1}},q_{i+1}}(B_{\rho_i})$. In fact, we have
\begin{align}\label{qi-qi1}
    \rho_i^{\gamma q_i}[u]^{q_{i+1}}_{W^{\gamma\frac{q_{i}}{q_{i+1}},q_{i+1}}(B_{\rho_i})}
    &\le 
    \rho_i^{\gamma q_i}\big(2 \|u\|_{L^\infty(B_{\rho_i})}\big)^{q_{i+1}-q_i} 
    [u]_{W^{\gamma, q_i}(B_{\rho_i})}^{q_i} \nonumber\\
    &\le 
    \rho_i^{\gamma q_{i+1}-N(\frac{q_{i+1}}{q_i}-1)}  
    [u]_{W^{\gamma, q_i}(B_{\rho_i})}^{q_{i+1}} +
    2^{q_{i+1}} R^N \|u\|_{L^\infty(B_{R})}^{q_{i+1}}.
\end{align}
Next, observe that by the ranges of $\al$ and $\gm$, we have $\gamma\frac{q_{i}}{q_{i+1}}\in( \frac{sp-2}{p-2},  \frac{sp}{p-2})$, since
$$
    \gamma\frac{q_{i}}{q_{i+1}}
    =
    \frac{\gamma}{\alpha}
    \ge 
    \frac{s}{\alpha}
    >
    \frac{sp-2}{p-2}, \quad\mbox{and}\quad\gamma\frac{q_{i}}{q_{i+1}}
    =
    \frac{\gamma}{\alpha}<\frac{sp}{p-2}.
$$
Therefore, we can apply Lemma~\ref{lem:frac-iter} with $(\sigma,q,r,R)$ replaced by $(\gamma\frac{q_{i}}{q_{i+1}},q_{i+1},\rho_{i+1},\rho_i)$.  As a result, we obtain 
\begin{align*}
    \rho_i^{\gamma q_{i+1}} &[u]_{W^{\gamma,q_{i+1}}(B_{\rho_{i+1}})}^{q_{i+1}} \\
    &\le 
    C_{i+1} 
    \Big(\frac{\rho_i}{\rho_i-\rho_{i-1}}\Big)^{\kappa_{i+1}} 
    \bigg[ 
    \rho_i^{\gamma q_i} [u]^{q_{i+1}}_{W^{\gamma\frac{q_{i}}{q_{i+1}},q_{i+1}}(B_{\rho_i})} +
    \rho_i^{N} \mathcal T_i^{q_{i+1}} 
    \bigg],
\end{align*}
where $C_{i+1}$ and $\kappa_{i+1}$ denote the corresponding constants from Lemma~\ref{lem:frac-iter}. 
We now use $\rho_{i+1}<\rho_i$, inequalities \eqref{rho-i-1}, \eqref{rho-i-2}, and~\eqref{est:tail-standard} from the proof of Lemma~\ref{lem:frac-iter} and~\eqref{qi-qi1} to find that
\begin{align*}
    \rho_{i+1}^{\gamma q_{i+1}} & [u]_{W^{\gamma,q_{i+1}}(B_{\rho_{i+1}})}^{q_{i+1}} \\
    &
    \le 
    2C_{i+1}2^{(i+1)\kappa_{i+1}}
    \Big(\frac{R}{R-r}\Big)^{\kappa_{i+1}} \\
    &\phantom{\le\,}\cdot
    \bigg[ 
    \rho_i^{\gamma q_{i+1}-N(\frac{q_{i+1}}{q_i}-1)}  
    [u]_{W^{\gamma, q_i}(B_{\rho_i})}^{q_{i+1}} + 
    2^{(i+1)Nq_{i+1}}
    R^{N} \mathcal T^{q_{i+1}} 
    \bigg] \\
    &\le 
    \widehat C_{i+1} \Big(\frac{R}{R-r}\Big)^{\kappa_{i+1}} 
    \bigg[ 
    \rho_i^{\gamma q_{i+1}-N(\frac{q_{i+1}}{q_i}-1)}  
    [u]_{W^{\gamma, q_i}(B_{\rho_i})}^{q_{i+1}} +
    R^{N} \mathcal T^{q_{i+1}} 
    \bigg].
\end{align*}
To obtain the last line we  abbreviated $\widehat C_{i+1}=2C_{i+1} 2^{(i+1)(\kappa_{i+1}+Nq_{i+1})}$. Inserting the induction assumption~\eqref{induction-q}$_i$ and using the definition of $\mathbf K_i$ we further estimate
\begin{align*}
    \rho_{i+1}^{\gamma q_{i+1}} & [u]_{W^{\gamma,q_{i+1}}(B_{\rho_{i+1}})}^{q_{i+1}} \\
    &\le 
    \widehat C_{i+1} \Big(\frac{R}{R-r}\Big)^{\kappa_{i+1}} 
    \bigg[  
    \rho_i^{-N(\frac{q_{i+1}}{q_i}-1)} 
    \widetilde C_i 
    \Big(\frac{R}{R-r}\Big)^{\tilde\kappa_i} 
    \mathbf{K}_i^{q_{i+1}} +
    R^{N} \mathcal T^{q_{i+1}} 
    \bigg] \\
    &\le 
    \widehat C_{i+1}\widetilde C_i 
    \Big(\frac{R}{R-r}\Big)^{\kappa_{i+1}+\tilde\kappa_i} 
    \bigg[ 
   \rho_i^{-N(\frac{q_{i+1}}{q_i}-1)} 
    \mathbf{K}_i^{q_{i+1}} +
    R^{N} \mathcal T^{q_{i+1}} 
    \bigg] \\
    &=
    \widehat C_{i+1}\widetilde C_i 
    \Big(\frac{R}{R-r}\Big)^{\kappa_{i+1}+\tilde\kappa_i} \\
     &\phantom{\le\,}\cdot
    \bigg[ 
    \|u\|_{L^\infty(B_{R})}^{q_{i+1}-p} 
    R^{s q_{i+1}-N(\frac{q_{i+1}}{p}-1)} [u]^{q_{i+1}}_{W^{s,p}(B_{R})} +
    2R^{N} \mathcal T^{q_{i+1}} 
    \bigg] \\
    &\le 
    \widetilde C_{i+1} \Big(\frac{R}{R-r}\Big)^{\tilde\kappa_{i+1}} 
    \mathbf{K}_{i+1}^{q_{i+1}},
\end{align*}
where we have abbreviated $\widetilde C_{i+1}=2 \widehat C_{i+1}\widetilde C_i$ and $\tilde\kappa_{i+1}=\kappa_{i+1}+\tilde\kappa_i$. This proves~\eqref{induction-q}$_{i+1}$ and finishes the induction argument.  

If $q=p$, the claim is implied by~\eqref{induction-q}$_0$, since the choice of $\alpha$ does not play any role in this case. Therefore, it remains to consider the case $q>p$. We choose $i_o\in \N$ to be the smallest integer such that 
$$
    \bigg(\frac{s(p-2)}{sp-2}\bigg)^{i_o}
    >
    \frac{q}{p}, 
$$
which means
$$
    i_o
    =
    \Biggl\lfloor
    \frac{\ln \frac{q}{p}}{\ln \frac{s(p-2)}{sp-2}}
    \Biggr\rfloor
    +
    1.
$$
Having fixed $i_o$ in this way, i.e.~in dependence on $s$, $p$, and $q$, we define $\alpha>1$ by
$$
    \alpha
    :=
    \Big(\frac{q}{p}\Big)^{\frac{1}{i_o}}
    \in 
    \bigg(1,\frac{s(p-2)}{sp-2}\bigg).
$$
In this way, we have $q_{i_o}=\alpha^{i_o} p=q$. 
At this point the claim follows from~\eqref{induction-q}$_{i_o}$. 
\end{proof}

\subsection{Higher H\"older regularity}\label{sec:holder-smalls}

In this subsection we prove the higher H\"older regularity result stated in Theorem~\ref{thm:Hoelder-subcritical}. 
It is a straightforward consequence of the higher fractional differentiability from Theorem~\ref{lem:Wgq} and Morrey's embedding for fractional Sobolev spaces from Lemma~\ref{Lem:morrey}.

\begin{proof}[\textbf{\upshape Proof of Theorem~\ref{thm:Hoelder-subcritical}}]
Let $\tilde\gamma=\frac12(\gamma+\frac{sp}{p-2})$ and $q=\frac{N}{\tilde\gamma-\gamma}$. The Morrey type embedding Lemma~\ref{Lem:morrey} applied with $\tilde\gamma$ instead of $\gamma$ and subsequently Theorem~\ref{lem:Wgq} yields 
\begin{align*}
    [w]_{C^{0,\gm}(B_{\frac{1}{2}R})}
    &=
    [w]_{C^{0,\tilde\gm-\frac{N}{q}}(B_{\frac{1}{2}R})}
    \le 
    C[w]_{W^{\tilde\gm,q}(B_{\frac{1}{2}R})} \\
    &\le 
    \frac{C}{R^{\tilde\gamma}} 
    \Big[ R^{s-N(\frac{1}{p}-\frac{1}{q})}[u]_{W^{s,p}(B_R)} +
    R^\frac{N}{q}\big( \|u\|_{L^\infty (B_R)} +\Tail (u;R)\big)\Big] \\
    &= 
    \frac{C}{R^{\gamma}} 
    \Big[ R^{s-\frac{N}{p}}[u]_{W^{s,p}(B_R)}
    +
    \|u\|_{L^\infty (B_R)} +\Tail (u;R)\Big],
\end{align*}
which proves the claim.
\end{proof}

\section{The  case $s\in (\frac{p-2}{p},1)$}\label{sec:super}
In this section we deal with the range $s\in (\frac{p-2}{p},1)$. First we prove in \S~\ref{sec:W1p} that weak solutions admit a gradient $\nabla u$ in $L^p_{\rm loc}$ and thus establish Theorem~\ref{thm:W1p}. In \S~\ref{sec:W1q} the integrability of the gradient is improved. Indeed, we show that the gradient $\nabla u$ belongs to $L^q_{\rm loc}$ for every $q\ge p$, which is exactly the claim of Theorem~\ref{thm:W1q}. Finally, \S~\ref{sec:Holder-s>} deals with the almost Lipschitz regularity result from Theorem~\ref{thm:Hoelder s>} and the improved fractional differentiability from Theorem~\ref{*thm:beta-q}. The role of $s$ is traced throughout this section; all estimates are stable as $s\uparrow1$.

\subsection{$W^{1,p}$-regularity}\label{sec:W1p}

In this section, we improve the regularity of a locally bounded, local weak solution of the fractional $(s,p)$-Laplacian from the initial fractional $W^{s,p}_{\rm loc}(\Omega)$-regularity to the Sobolev regularity $W^{1,p}_{\rm loc}(\Omega)$. This is achieved by establishing the difference quotient $\btau_h u/|h|$ is uniformly bounded in $L^p$; see Lemma~\ref{lem:diff-quot-p} below. We realize it via an iteration scheme. The starting point is the basic estimate for finite differences of second order $\btau_h (\btau_h u)$ from Lemma \ref{lem:Nikol-est}, which follows by applying the energy inequality from Proposition~\ref{prop:energy} with the choice $\delta =1$. In view of Lemma \ref{lem:Domokos} in the case $\gm<1$, such an estimate can be recast into the mnemonic relation
\[
A_{i}\approx |h|^{sp}\big[|h|^2 + A_{i-1}^{\frac2p}\big]\quad\text{where}\> A_i=\|\btau_h(\btau_h u)\|^p_{L^p(B_i)}.
\]
Then, we can start with $A_0=1$ and iterate to obtain $A_1\approx |h|^{sp}$, $A_2=|h|^{sp+2s}$, etc. Heuristically, we have $A_\infty\approx |h|^{p\frac{sp}{p-2}}$.
This argument can be repeated finitely many times until the power of $|h|$ exceeds $p$, which occurs only if $s\in(\frac{p-2}{p},1)$. 
Then, we apply Lemma~\ref{lem:Domokos} in the case $\gm>1$ and conclude that $\btau_h u/|h|$ is uniformly bounded in $L^p$, from which the weak differentiability of $u$ follows in Theorem~\ref{cor:W1p} below. Note that Theorem~\ref{thm:W1p} follows as a special case of Theorem~\ref{cor:W1p} by choosing $r=\frac12 R$.

\begin{lemma}\label{lem:Nikol-est}
Let $p\in [2,\infty)$ and $s\in(\frac{p-2}{p},1)$. Then, there exists a constant $C\ge 1$ of the form $\widetilde{C}(N,p)/s$ such that whenever $u\in W^{s,p}_{\rm loc}(\Omega)\cap L^{p-1}_{sp}(\R^N)$ is a locally bounded local weak solution of~\eqref{PDE} in the sense of Definition~\ref{def:loc-sol}, $B_{R}\equiv B_{R}(x_o)\Subset \Omega$, $r\in(0,R)$ and $d=\frac1{7}(R-r)$ we have 
\begin{align*}
    \int_{B_r}\!\! |\btau_h (\btau_h u)|^p \,\dx 
    &\le 
    C\frac{|h|^{sp}}{R^{sp}} \Big(\frac{R}{R-r}\Big)^{N+sp+1} 
     \mathbf K^{p-2} 
    \Bigg[\frac{|h|^2}{R^2} \mathbf{K}^2 +
    \bigg[\int_{B_{R-d}}\!\!|\btau_hu|^p\,\dx\bigg]^\frac2p\Bigg] 
\end{align*}
for any $h\in\R^N\setminus\{0\}$ with $|h|\le d$, where 
\begin{align*}
    \mathbf{K}^p
    :=
    R^{sp}(1-s)[u]^p_{W^{s,p}(B_{R})} +
    R^{N} \big(\|u\|_{L^\infty(B_{R})} + \Tail(u;R)\big)^p .
\end{align*}
\end{lemma}

\begin{remark}\label{rem:N+sp+1}\upshape
The statement of Lemma~\ref{lem:Nikol-est} continues to hold for $s\in(0,\frac{p-2}{p}]$ with the larger exponent $N+p+1$ of $\frac{R}{R-r}$ instead of $N+sp+1$. However, we do not use the inequality in this  range. 
\end{remark}

\begin{proof} 
Apply the energy inequality from Proposition \ref{prop:energy} with $\delta=1$ and 
$\tilde r=\frac17(5r+2R)$,
$\widetilde R=\frac17(r+6R)$ instead of $R$, $r$. Then,
$d=\frac14(\widetilde R-\tilde r)=\frac17(R-r)$.
Using inequalities~\eqref{rtilde} and~\eqref{s-l-rtilde} from the proof of Lemma~\ref{lem:frac-impr} allows us to replace $\frac{\widetilde R}{\widetilde R-\tilde r}$ by $\frac{R}{R-r}$ and $\frac1{\widetilde R}$ by $\frac1R$ when applying Proposition \ref{prop:energy} apart from a multiplicative constant depending only on $N$ and $p$.
Such an application yields 
\begin{align*}
    \mathbf I
    &:=
    [\eta\btau_hu]_{W^{s,p}(B_{\widetilde R})}^p
    \equiv
    \iint_{B_{K_{\widetilde R}}}
    \frac{\big|\btau_hu(x)\eta(x)- \btau_hu(y)\eta(y)\big|^p}{|x-y|^{N+sp}}\dx\dy \\
    &\le
    \frac{C}{R^{2s}}\Big( \frac{R}{R-r}\Big)^2[u]_{W^{s,p}(B_R)}^{p-2}
    \bigg[
    \int_{B_{\widetilde R}}\frac{|\btau_hu|^p}{1-s}\,\dx \bigg]^\frac2p\\
    &\phantom{\le\,}+
    \frac{C}{R^{sp}}\Big(\frac{R}{R-r}\Big)^{N+sp}
    \mathcal T^{p-2}
   \int_{ B_{\widetilde R}}\frac{|\btau_h u|^{2}}{1-s}\,\dx\\
    &\phantom{\le\,}
    +\frac{C |h|}{R^{sp+1}} 
    \Big(\frac{R}{R-r}\Big)^{N+sp+1}
    \mathcal T^{p-1}\int_{B_{\widetilde R}} \frac{|\btau_hu|}{1-s} \,\dx\\
    &=:
    \mathbf{I}_1+\mathbf{I}_2+\mathbf{I}_3
\end{align*}
with a constant $C= \widetilde{C}(N,p)/s$ and the abbreviation 
\begin{equation*}
    \mathcal T:= \|u\|_{L^{\infty}(B_{R})}+\mathrm{Tail}(u;  R).
\end{equation*}
Here, $\eta\in C^1_0(B_{\frac12 (\widetilde{R}+\tilde{r})})$ denotes the usual cut-off function in $\frak Z_{\tilde{r}, \widetilde R} $, cf.~Definition \ref{Def:Z}, satisfying $\eta =1$ on $B_{\tilde{r}}$ and $|\nabla\eta|\le \frac{4}{\widetilde R-\tilde r}=\frac{7}{R-r}$.

The first term is estimated by observing that 
$N+sp+1>N+1\ge 2$ holds true, such that one can enlarge the power of $\frac{R}{R-r}$ from $2$ to $N+sp+1$. As a result, we have
\begin{align*}
    \mathbf{I}_1\le \frac{C}{(1-s)R^{sp}}\Big( \frac{R}{R-r}\Big)^{N+sp+1}
    \bigg[\int_{B_{\widetilde R}}|\btau_hu|^p\,\dx\bigg]^\frac2p\big(R^{sp}
    (1-s)[u]^p_{W^{s,p}(B_R)}\big)^{\frac{p-2}p}.
\end{align*}
To the integral in the second term we apply H\"older's inequality, and enlarge the power of $\frac{R}{R-r}$ from $N+sp$ to $N+sp+1$. Then,
\begin{align*}
    \mathbf{I}_2
    &\le
    \frac{C}{R^{sp}}\Big(\frac{R}{R-r}\Big)^{N+sp+1}
    \mathcal T^{p-2} \frac{R^{N(1-\frac2p)}}{1-s}\bigg[ \int_{B_{\widetilde R}}|\btau_hu|^p\,\dx\bigg]^\frac2p.
\end{align*}
The third term is treated similarly as
\begin{align*}
    \mathbf{I}_3
    &\le
    \frac{C |h|}{(1-s)R^{sp+1}} 
    \Big(\frac{R}{R-r}\Big)^{N+sp+1}
    \mathcal T^{p-1} 
    \bigg[ \int_{B_{\widetilde R}} |\btau_hu|^p \,\dx\bigg]^{\frac1p} R^{N(1-\frac1p)}.
\end{align*}
Combining these estimates we obtain that
\begin{align*}
    \mathbf I
    &\le
    \frac{C}{(1-s)R^{sp}}\Big( \frac{R}{R-r}\Big)^{N+sp+1}
    \bigg[\int_{B_{\widetilde R}}|\btau_hu|^p\,\dx\bigg]^\frac2p\\
    &\qquad
    \cdot
    \bigg[
    \big(R^{sp}(1-s)[u]^p_{W^{s,p}(B_R)}\big)^{\frac{p-2}p} 
    +
    R^{N(1-\frac2p)}\mathcal T^{p-2}
    \bigg]\\
    &\phantom{\le\,}+
    \frac{C |h|}{(1-s)R^{sp+1}} 
    \Big(\frac{R}{R-r}\Big)^{N+sp+1}
    \bigg[ \int_{B_{\widetilde R}} |\btau_hu|^p \,\dx\bigg]^{\frac1p} 
    R^{N(1-\frac1p)}
    \mathcal T^{p-1} \\
    &\le 
    \frac{C}{(1-s)R^{sp}}\Big( \frac{R}{R-r}\Big)^{N+sp+1}
    \mathbf{K}^{p-2}
    \bigg[\int_{B_{\widetilde R}}|\btau_hu|^p\,\dx\bigg]^\frac2p\\
    &\phantom{\le\,} +
    \frac{C }{(1-s)R^{sp}} 
    \Big(\frac{R}{R-r}\Big)^{N+sp+1}
    \mathbf K^{p-2}\frac{\mathbf K |h|}{R}
    \bigg[ \int_{B_{\widetilde R}} |\btau_hu|^p \,\dx\bigg]^{\frac1p} \\
    &\le
    \frac{C}{(1-s)R^{sp}}\Big( \frac{R}{R-r}\Big)^{N+sp+1}
    \mathbf{K}^{p-2} 
    \Bigg[\frac{|h|^2}{R^2} \mathbf{K}^2 +
    \bigg[\int_{B_{\widetilde R}}|\btau_hu|^p\,\dx\bigg]^\frac2p\Bigg],
\end{align*}
with $C=\widetilde{C}(N,p)/s$. In turn we applied Young's inequality to get the last line.

The above estimate of $\mathbf I$ allows us to bound the $L^p$-norm of $\btau_\lambda( \eta \btau_h u)$. In fact, we apply  Lemma 
\ref{lem:N-FS} with $v$, $q$, $\gamma$, $R$, $d$ replaced by $\eta\btau_hu$, $p$, $s$, $\widetilde R-d$, $ d=\frac1{7}(R-r)$. The application yields for any $|h|\le d$ and any $|\lambda|\le d$ with a constant $C=C(N,p)$ that
\begin{align*}
    \int_{B_{\widetilde R-d}}&\big|\btau_\lambda\big( \eta \btau_h u\big)\big|^p\,\dx\\
    &\le
    C |\lambda|^{sp}
    \bigg[(1-s)
    [\eta\btau_hu]_{W^{s,p}(B_{\widetilde R})}^p
    +
    \frac{\widetilde{R}^{(1-s)p}+h_o^{(1-s)p}}{h_o^p}\| \eta\btau_hu\|^p_{L^p(B_{\widetilde R})}
    \bigg].
\end{align*}
To proceed, observe that
\begin{align*}
    \frac{\widetilde R^{(1-s)p}+d^{(1-s)p}}{d^p}
    &\le 
\frac{C(p)}{R^{sp}}\Big(\frac{R}{R-r}\Big)^p,
\end{align*}
and the elementary inequality
\begin{align*}
     \int_{B_{\widetilde R}}|\btau_h u|^p\,\dx
     &\le
     C(N,p) R^{N(1-\frac2p)}\|u\|_{L^\infty (B_R)}^{p-2}
    \bigg[\int_{B_{\widetilde R}}|\btau_hu|^p\,\dx\bigg]^\frac2p.
\end{align*}
We use these observations together with the above bound of $\mathbf{I}$ to estimate
\begin{align*}
    &\int_{B_{\widetilde R-d}} 
    \big|\btau_\lambda\big( \eta \btau_h u\big)\big|^p\,\dx \nonumber\\
    &\quad \le
    C\frac{|\lambda|^{sp}}{R^{sp}} 
    \bigg[
    (1-s)
    R^{sp}\mathbf I
    +
    \Big(\frac{R}{R-r}\Big)^p
    \int_{B_{\widetilde R}}|\btau_h u|^p\,\dx
    \bigg] \nonumber\\
    &\quad \le
    C\frac{|\lambda|^{sp} }{R^{sp}} 
    \Bigg[
    (1-s)R^{sp}\mathbf I
    +
    \Big(\frac{R}{R-r}\Big)^{p}R^{N(1-\frac2p)}\|u\|_{L^\infty (B_R)}^{p-2}
    \bigg[\int_{B_{\widetilde R}} |\btau_hu|^p\dx\bigg]^\frac2p\Bigg] \nonumber\\   
    &\quad \le
     C\frac{|\lambda|^{sp} }{R^{sp}}\Big(\frac{R}{R-r}\Big)^{N+sp+1}
     \mathbf{K}^{p-2}
     \Bigg[\frac{|h|^2}{R^2} \mathbf{K}^2 + \bigg[\int_{B_{\widetilde R}}|\btau_hu|^p\,\dx\bigg]^\frac2p\Bigg]
\end{align*}
for any $0<|\lambda|\le d$. The constant $C$ has the form $C=\widetilde{C}(N,p)/s$.
Note that $\widetilde R-d=R-2d$. Choosing $\lambda=h$  we arrive at
\begin{align*}
    \int_{B_{ R-2d}}&\big|\btau_h\big( \eta \btau_h u\big)\big|^p\,\dx \\
    &\le
    C\frac{|h|^{sp}}{R^{sp}}\Big(\frac{R}{R-r}\Big)^{N+sp+1} \mathbf{K}^{p-2}
    \Bigg[\frac{|h|^2}{R^2} \mathbf{K}^2 + \bigg[\int_{B_{\widetilde R}}|\btau_hu|^p\,\dx\bigg]^\frac2p\Bigg]
\end{align*}
for a constant $C=\widetilde{C}(N,p)/s$. Due to the choice of $\eta\in C^1_0(B_{\frac12(\widetilde R+\tilde r)})$, precisely $\eta =1$ in $B_{\tilde r}= B_{r+2d}$, we have $\btau_h(\eta\btau_hu)=\btau_h(\btau_hu)$ in $B_{r+2d}$. From the preceding inequality we therefore get
\begin{align*}
    \int_{B_r}|\btau_h(\btau_hu)|^p\,\dx
    &=
    \int_{B_r}|\btau_h(\eta\btau_hu)|^p\,\dx
    \le
    \int_{B_{R-2d}}|\btau_h(\eta\btau_hu)|^p\,\dx\\
    &\le
    C\frac{|h|^{sp}}{R^{sp}}\Big(\frac{R}{R-r}\Big)^{N+sp+1} \mathbf{K}^{p-2}
    \Bigg[\frac{|h|^2}{R^2} \mathbf{K}^2 + \bigg[\int_{B_{\widetilde R}}|\btau_hu|^p\,\dx\bigg]^\frac2p\Bigg].
\end{align*}
This is the desired estimate for the second order finite differences.
\end{proof}

In the next lemma, we iterate the estimates obtained in Lemma~\ref{lem:Nikol-est} to increase the power of the increment $|h|$.

\begin{lemma}\label{lem:diff-quot-p}
Let $p\in[2,\infty)$, $s\in (\frac{p-2}{p},1)$. Then, there exists a  constant $C\ge 1$ depending on $N$, $p$, and $s$ such that whenever $u\in W^{s,p}_{\rm loc}(\Omega)\cap L^{p-1}_{sp}(\R^N)$ is a locally bounded, local weak solution of~\eqref{PDE} in the sense of Definition~\ref{def:loc-sol}, $B_{R}\equiv B_{R}(x_o)\Subset \Omega$, and $r\in(0,R)$ we have
\begin{align}\label{est:tau_hu^p}
    \int_{B_r} |\btau_h u|^p \,\dx 
    &\le 
    C\frac{|h|^{p}}{R^{p}} \Big(\frac{R}{R-r}\Big)^{\frac{3}{s}(N+sp+1
    )} 
     \mathbf{K}^{p} 
\end{align}
for any $h\in\R^N\setminus\{0\}$ with $|h|\le R-r$, where 
\begin{align}\label{def:K}
    \mathbf{K}^p
    :=
    R^{sp} (1-s)[u]^p_{W^{s,p}(B_{R})} +
    R^{N} \big(\|u\|_{L^\infty(B_{R})} + \Tail(u;R)\big)^p .
\end{align}
Moreover, the constant $C$ remains stable as $s\uparrow1$. 
\end{lemma}

\begin{proof}
For $i\in \N_0$ we define sequences
\begin{equation*}
    \rho_i:= r+\frac{1}{2^{i+1}}(R-r)
    \quad\mbox{and}\quad
    s_i:=s\sum_{j=0}^i\Big(\frac{2}{p}\Big)^j=\frac{sp}{p-2}\bigg[ 1-\Big(\frac{2}{p}\Big)^{i+1}\bigg].
\end{equation*}
Note that $s_o=s$  and $s_i\uparrow \frac{sp}{p-2}$ as $i\to\infty$. 
If $p=2$, the definition of $s_i$ reduces to $s_i=(i+1)s$ and the term $\frac{sp}{p-2}$ must be interpreted as $\infty$.
Moreover, we have
\begin{equation}\label{eq:recursive-s_i}
    s_{i}=s_{i-1}+(\tfrac{2}{p})^{i}s
    \quad\mbox{and}\quad
    s_{i}
    =
    s+\tfrac2p s_{i-1}.
\end{equation}
Fix $i\in\N_0$. Applying Lemma \ref{lem:Nikol-est}
with $r$, $R$, and $d$ replaced by $\rho_i$, $\rho_{i-1}$, and $d_i=\frac1{7}(\rho_{i-1}-\rho_{i})=\frac{1}{7\cdot 2^{i+1}}(R-r)$,  we obtain
\begin{align*}
    &\int_{B_{\rho_{i}}} |\btau_h (\btau_h u)|^p \,\dx \nonumber\\
    &\qquad\le 
    \frac{C}{\rho_{i-1}^{sp}} \Big(\frac{\rho_{i-1}}{\rho_{i-1}-\rho_{i}}\Big)^{N+sp+1} 
    |h|^{sp} \mathbf{K}_{i-1}^{p-2} 
    \Bigg[\frac{|h|^2}{\rho_{i-1}^2} \mathbf{K}_{i-1}^2 +
    \bigg[\int_{B_{\rho_{i-1}-d_i}}|\btau_hu|^p\,\dx\bigg]^\frac2p\Bigg] 
\end{align*}
for any $h\in\R^N\setminus\{0\}$ with $|h|\le d_i$, where 
\begin{align*}
    \mathbf{K}_{i-1}^p
    :=
    \rho_{i-1} ^{sp}(1-s)[u]^p_{W^{s,p}(B_{\rho_{i-1}})} +
    \rho_{i-1}^{N} \big(\|u\|_{L^\infty(B_{\rho_{i-1}})} + 
    \Tail(u;\rho_{i-1})\big)^p 
\end{align*}
and $C=\widetilde{C}(N,p)/s$. 
Similarly to \eqref{rho-i-1}, \eqref{rho-i-2}, and~\eqref{est:tail-standard} from the proof of Lemma~\ref{lem:frac-iter} we have
\begin{equation*}    
    R
    \ge 
    \rho_{i-1} 
    >
    \frac{R}{2^{i}},
    \qquad
    \frac{\rho_{i-1}}{\rho_{i-1}-\rho_{i}}
    <
    2^{i+1}\frac{R}{R-r},
\end{equation*}
and 
\begin{align*}
    \Tail(u;\rho_{i-1})
    \le
    C(N,p)2^{\frac{iN}{p-1}}
    \big( \Tail (u;R)+\|u\|_{L^\infty(B_R)}\big).
\end{align*}
From this we derive that $\mathbf{K}_{i-1}\le C(N,p)2^{iN} \mathbf{K}$, where $\mathbf{K}$ is the quantity from~\eqref{def:K}. Moreover, we enlarge the domain of integration in the integral on the right-hand side from $B_{\rho_{i-1}-d_i}$ to $B_{\rho_{i-1}}$. 
With these remarks, the above inequality for second finite differences becomes
\begin{align}\label{est:it-i-2}\nonumber
    \int_{B_{\rho_{i}}} &|\btau_h (\btau_h u)|^p \,\dx \\
    &\le 
    \widetilde{C}_i
    \frac{|h|^{sp}}{R^{sp}} \Big(\frac{R}{R-r}\Big)^{N+sp+1} 
     \mathbf{K}^{p-2}
    \Bigg[\frac{|h|^2}{R^2} \mathbf{K}^2 + 
    \bigg[\int_{B_{\rho_{i-1}}} |\btau_h u|^p \,\dx \bigg]^{\frac{2}{p}} \Bigg]
\end{align}
for any $h\in\R^N\setminus\{0\}$ with $|h|\le d_i$, where $\widetilde{C}_i=2^{i(Np+N+2p+3)}\widetilde{C}(N,p)/s$.

To proceed further we consider $\alpha\in(0,1]$ and $i_o\in \N_0$ (both to be fixed later), such that $\alpha s_{i_o}<1$. 
See~Remark~\ref{Rmk:alpha} for the reason of introducing the parameter $\al$.
Now, we show by induction that for any $i\in\{0,1,\dots,i_o\}$ (then $\alpha s_i\le \alpha s_{i_o}<1$) there holds 
\begin{align}\label{est:it-i-3}
    \int_{B_{\rho_{i}}} |\btau_h u|^p \,\dx 
    &\le 
    C_i
    \frac{|h|^{\alpha s_ip}}{R^{\alpha s_ip}} 
    \Big(\frac{R}{R-r}\Big)^{(N+sp+1)\frac{s_i}{s}} 
     \mathbf{K}^{p},\qquad \forall\, 0<|h|\le d_i, 
\end{align}
for a constant 
$C_i=\frac{2^{(Np+N+2p+3)i!}}{[s(1-\alpha s_{i_o})^p]^{i}}\, C(N,p)$. 

First consider {\bf the case} $i=0$. In this case we have $\rho_o= r+\frac12(R-r)=\frac12 (R+r)$. Applying Lemma \ref{lem:N-FS} with $(v,q,\gamma, d,R)$ replaced by $(u, p, s, \frac12 (R-r), \rho_o)$, and noting that this  implies $\rho_o+\frac12 (R-r)=R$,
we have
\begin{align*}
     \int_{B_{\rho_o}}& |\btau_h u|^p \,\dx \\
     &\le
     C\,|h|^{sp} 
     \bigg[
     (1-s)[u]_{W^{s,p}(B_R)}^p+\bigg( \frac{R^{(1-s)p}}{(R-r)^{p}}
     +\frac1{s(R-r)^{sp}}\bigg) \|u\|^p_{L^p(B_R)}
     \bigg]\\
     &\le
     C\frac{|h|^{sp}}{R^{sp}} 
     \bigg[ R^{sp}(1-s)
     [u]_{W^{s,p}(B_R)}^p+\bigg( \frac{R^{p}}{(R-r)^{p}}
     +\frac{R^{sp}}{(R-r)^{sp}}\bigg) \|u\|^p_{L^p(B_R)}
     \bigg] \\
      &\le
     C\frac{|h|^{sp}}{R^{sp}} |h|^{sp}
     \Big(\frac{R}{R-r}\Big)^p \bigg[ R^{sp}(1-s)
     [u]_{W^{s,p}(B_R)}^p+  \big(R^\frac{N}{p}\|u\|_{L^\infty(B_R)}\big)^p
     \bigg] \\
     &\le
     C\frac{|h|^{s_op}}{R^{s_op}} \Big(\frac{R}{R-r}\Big)^{N+sp+1} 
     \mathbf{K}^p\\
     &\le
     C\frac{|h|^{\alpha s_op}}{R^{\alpha s_op}} \Big(\frac{R}{R-r}\Big)^{N+sp+1} 
     \mathbf{K}^p
\end{align*}
for any $0<|h|\le\tfrac12 (R-r)$. In turn we used 
the $L^\infty$-bound for $u$, the definition of $\mathbf{K}$, $s_o=s$, and $|h|/R\le \frac12$. The constant $C$ has the form $\widetilde C(N,p)/s$. 

For the {\bf induction step} we assume that \eqref{est:it-i-3}$_{i-1}$ holds true. Using in \eqref{est:it-i-2} the assumption \eqref{est:it-i-3}$_{i-1}$ in order to bound the $L^p$-norm of $\btau_hu$ on $B_{\rho_{i-1}}$ we obtain for any $0<|h|\le d_{i-1}$ that
\begin{align}\label{tauh-tauh}
    &\int_{B_{\rho_{i}}} |\btau_h (\btau_h u)|^p \,\dx \nonumber\\
    &\qquad
    \le 
    \widetilde{C}_i
    \frac{ |h|^{sp}}{R^{sp}} 
    \Big(\frac{R}{R-r}\Big)^{N+sp+1} 
    \mathbf{K}^{p-2}\nonumber\\
    &\qquad\qquad\cdot
    \Bigg[ \frac{|h|^2}{R^2} \mathbf{K}^2 +
    \bigg[
     \frac{C_{i-1}}{R^{\alpha s_{i-1}p}} \Big(\frac{R}{R-r}\Big)^{(N+sp+1)\frac{s_{i-1}}{s}} 
    |h|^{\alpha s_{i-1}p} \mathbf{K}^{p}
    \bigg]^\frac{2}{p}\Bigg] \nonumber\\
    &\qquad\le
    2\widetilde{C}_i C_{i-1}^{\frac2p}
    \frac{|h|^{p(s+\frac{2}{p}\alpha s_{i-1})}}{R^{p(s+\frac{2}{p}\alpha s_{i-1})}} \Big(\frac{R}{R-r}\Big)^{(N+sp+1)(1+\frac{2}{p}\frac{s_{i-1}}{s})} 
    \mathbf{K}^{p} \nonumber\\
    &\qquad\le
    2\widetilde{C}_i C_{i-1}^{\frac2p}
    \frac{|h|^{\alpha ps_i}}{R^{\alpha ps_i}} \Big(\frac{R}{R-r}\Big)^{(N+sp+1)\frac{s_i}{s}} 
    \mathbf{K}^{p}.
\end{align}
Here, to obtain the last line we used \eqref{eq:recursive-s_i}$_1$,  $\alpha s_i=\alpha(s+\frac{2}{p}s_{i-1})< s+ \frac{2}{p}\alpha s_{i-1}$, and $\alpha s_{i-1}<1$.
This allows us to replace the power $p(s+\frac{2}{p}\alpha s_{i-1})$ of $\frac{|h|}{R}$ and the power $(N+sp+1)(1+\frac{2}{p}\frac{s_{i-1}}{s})$ of $\frac{R}{R-r}$ in the second-to-last line by $\al p s_i$ and by $(N+sp+1)\frac{s_i}{s}$ respectively, as shown in the last line.
The estimate \eqref{tauh-tauh} plays the role of the assumption \eqref{ass:Domokos}$_{\sigma<1}$ in Lemma \ref{lem:Domokos}
and permits us to apply \eqref{est-1st-diffquot<1}   with $\big(p, \alpha s_i, \rho_i, \rho_{i-1}, d_i\big)  $ instead of $\big(q, \gamma, r,R, d\big)$ and with
$$
    M^p
    := 
    \frac{\widetilde{C}_i C_{i-1}^{\frac2p}}{R^{\alpha ps_i}} \Big(\frac{R}{R-r}\Big)^{(N+sp+1)\frac{s_i}{s}} 
    \mathbf{K}^{p}.
$$
Using again the $L^\infty$-bound for $u$, the definition of $d_i$, 
the fact $\rho_{i-1}\le R$, the definition of $\mathbf{K}$ from \eqref{def:K}, and the fact that $1-\alpha s_{i}\ge 1 -\alpha s_{i_o}$,
we have 
\begin{align*}
    \int_{B_{\rho_i}} & |\btau_h u|^p\,\dx\\
    &\le
    C\,|h|^{\alpha ps_i}\bigg[ 
    \frac{1}{(1-\alpha s_i)^p}
    \underbrace{\frac{\widetilde{C}_i C_{i-1}^{\frac2p}}{ R^{\alpha ps_i}} \Big(\frac{R}{R-r}\Big)^{(N+sp+1)\frac{s_i}{s}} 
    \mathbf{K}^{p}}_{\equiv M^p} +\frac{1}{d_i^{\alpha ps_i}}\int_{B_{\rho_{i-1}}}|u|^p\,\dx
    \bigg]\\
    &\le
    \frac{C\,\widetilde{C}_i C_{i-1}}{(1 -\alpha s_{i_o})^p} 
    |h|^{\alpha ps_i}\bigg[ 
    \frac{1}{R^{\alpha ps_i}} \Big(\frac{R}{R-r}\Big)^{(N+sp+1)\frac{s_i}{s}} 
    \mathbf{K}^{p} +\frac{\rho_{i-1}^N}{d_i^{\alpha ps_i}}\|u\|_{L^\infty(B_R)}^p\bigg]\\
    &\le
    C_i \frac{|h|^{\alpha ps_i}}{R^{\alpha ps_i}}
    \bigg[ 
    \Big(\frac{R}{R-r}\Big)^{(N+sp+1)\frac{s_i}{s}} 
    \mathbf{K}^{p} + 
    \Big( \frac{R}{R-r}\Big)^{\alpha ps_i}\underbrace{R^N\|u\|_{L^\infty(B_R)}^p}_{\le \mathbf{K}^p}\bigg]\\
    &\le
    C_i
    \frac{|h|^{\alpha ps_i}}{R^{\alpha ps_i}} \Big(\frac{R}{R-r}\Big)^{(N+sp+1)\frac{s_i}{s}} 
    \mathbf{K}^{p} 
\end{align*}
for any $0<|h|\le  d_i$ and with $C_i$ as in \eqref{est:it-i-3}. 
This finishes the induction step and proves \eqref{est:it-i-3} for any $i\in\{0,1,\dots,i_o\}$.

\medskip
Now, we come to the proof of inequality~\eqref{est:tau_hu^p}. 
To this aim we choose $i_o\in\N$ such that $s_{i_o}\ge 1$ and $s_{i_o-1}< 1$. Note that $i_o$ depends only on $p$ and $s$ and can be determined explicitly, i.e.
$$
    i_o=\Biggl\lceil\frac{\ln \left(1-\frac{p-2}{sp}\right)}{\ln\frac2p}-1\Biggr\rceil.
$$
Recall, in the case $p=2$, the definition of $s_i$ reduces to $s_i=(i+1)s$, and meanwhile $i_o=\lceil\frac1s -1\rceil$. Note also that in all cases $i_o\to1$ as $s\uparrow1$.
To proceed further we let
\begin{equation}\label{def:alpha}
    \alpha
    :=
    \frac{p(2-s)}{(p+2)s_{i_o}}
\end{equation}
and note that $\alpha\in (\frac{p}{2(p+2)},1)$, since on the one hand $s_{i_o}\ge 1$ and $\frac{p(2-s)}{p+2}<1$, and on the other hand $s_{i_o}=s_{i_o-1}+(\frac2p)^{i_o}s< 2$. Moreover, $\alpha$ is chosen in such a way that 
\begin{equation}\label{est:s-alpha_i}
    \alpha s_{i_o}
    =
    \frac{p(2-s)}{p+2}
    <
    1.
\end{equation}
Observe that
$$
    1-\alpha s_{i_o}= \frac{sp-(p-2)}{p+2}>0
$$
holds true, because $s>\frac{p-2}{p}$ by assumption.
See~Remark~\ref{Rmk:alpha} for the reason of introducing the parameter $\al$. 
We now apply inequality~\eqref{est:it-i-3} with $i=i_o$.
As a result, we have 
\begin{align*}
    \int_{B_{\rho_{i_o}}} |\btau_h u|^p \,\dx 
    \le 
    C_{i_o}
    \frac{ |h|^{\alpha s_{i_o}p}}{R^{\alpha s_{i_o}p}} \Big(\frac{R}{R-r}\Big)^{\frac{s_{i_o}}{s}(N+sp+1) } 
    \mathbf{K}^{p},\qquad \forall\, 0<|h|\le d_{i_o},
\end{align*}
where $C_{i_o}=\frac{2^{(Np+N+2p+3)i_o!}}{[s(1-\alpha s_{i_o})^p]^{i_o}}\, C(N,p)$. Since $\alpha s_{i_o}<1$ we still have to improve the exponent of $|h|$. 
To this end, we first apply 
\eqref{est:it-i-2} on $B_{\rho_{i_o+1}}$ and subsequently use the preceding inequality to estimate the integral on the right. 
For the appearing exponents we  
recall \eqref{eq:recursive-s_i} and \eqref{est:s-alpha_i}. The latter
implies  that
\begin{equation}\label{def:sigma}
        sp+2\alpha s_{i_o}=\sigma p >p,\quad\mbox{where}\quad
    \sigma:=\frac{sp+4}{p+2}>1.
\end{equation}
Here, we used again the lower bound $s>\frac{p-2}{p}$. Note that
$$
    \sigma -1=\frac{sp -(p-2)}{p+2}=1-\alpha s_{i_o}.
$$
In this way we acquire similarly as in \eqref{tauh-tauh} that 
\begin{align*}
    \int_{B_{\rho_{i_o+1}}} & |\btau_h (\btau_h u)|^p \,\dx\\ 
    &\le 
    \widetilde{C}_{i_o+1}
    \frac{|h|^{sp}}{R^{sp}} \Big(\frac{R}{R-r}\Big)^{N+sp+1} 
     \mathbf{K}^{p-2}
    \Bigg[\frac{|h|^2}{R^2} \mathbf{K}^2 +
    \bigg[\int_{B_{\rho_{i_o}}} |\tau_h u|^p \,\dx \bigg]^{\frac{2}{p}} \Bigg]
    \\
    &\le
    \widetilde{C}_{i_o+1}
     \frac{ |h|^{sp}}{R^{sp}} \Big(\frac{R}{R-r}\Big)^{N+sp+1} 
    \mathbf{K}^{p-2}\\
    &\qquad\qquad\cdot
    \Bigg[ \frac{|h|^2}{R^2} \mathbf{K}^2 +
    \bigg[ \frac{C_{i_o}}{R^{\alpha s_{i_o}p}} \Big(\frac{R}{R-r}\Big)^{(N+sp+1)\frac{s_{i_o}}{s}} 
    |h|^{\alpha s_{i_o}p} \mathbf{K}^{p}\bigg]^\frac2p\Bigg]  \\
    &\le
    C\, \widetilde{C}_{i_o+1} C_{i_o}
     \frac{|h|^{sp+2\alpha s_{i_o}}}{R^{sp+2\alpha s_{i_o}}} \Big(\frac{R}{R-r}\Big)^{(N+sp+1)\frac{s_{i_o+1}}{s}} \mathbf{K}^{p}
    \\
    &=
    C\, \widetilde{C}_{i_o+1} C_{i_o}
     \frac{|h|^{\sigma p}}{R^{\sigma p}} \Big(\frac{R}{R-r}\Big)^{(N+sp+1)\frac{s_{i_o+1}}{s}} \mathbf{K}^{p} 
\end{align*}
for any $0<|h|\le  d_{i_o+1}$ and where 
$\widetilde{C}_{i_o+1}$ can be written in the form $\widetilde{C}_{i_o+1}=2^{i_o(Np+N+2p+3)}\widetilde{C}(N,p)/s$. 
The above estimate plays the role of the assumption \eqref{ass:Domokos}$_{\sigma>1}$ in Lemma \ref{lem:Domokos}, and
we are now in the position to apply \eqref{est-1st-diffquot>1}  with 
$(p,\rho_{i_o+1}, \rho_{i_o}, d_{i_o+1},\sigma=\frac{sp+4}{p+2})$ instead of $(q,r,R, d,\gamma)$, and with
$$
    M^p:= 
    \frac{C\, \widetilde{C}_{i_o+1} C_{i_o}}{R^{\sigma p}} \Big(\frac{R}{R-r}\Big)^{(N+sp+1)\frac{s_{i_o+1}}{s}} \mathbf{K}^{p}.
$$
The application of \eqref{est-1st-diffquot>1} gives for any $h\in \R^N\setminus\{0\}$ with $|h|\le d_{i_o+1}$ that
\begin{align*}
     &\int_{B_{\rho_{i_o+1}}} |\btau_h  u|^p \,\dx\\
    &\qquad\le
     C|h|^p\bigg[ 
      \frac{\widetilde{C}_{i_o+1} C_{i_o} d_{i_o+1}^{(\sig-1) p}}{(\sigma-1)^pR^{\sig p}} \Big(\frac{R}{R-r}\Big)^{(N+sp+1)\frac{s_{i_o+1}}{s}} \mathbf{K}^{p}
      +
      \frac{1}{d_{i_o+1}^p}\int_{B_{\rho_{i_o}}}|u|^p\dx\bigg]\\
      &\qquad =
     C\frac{|h|^p}{R^p}\bigg[ \frac{\widetilde{C}_{i_o+1} C_{i_o} d_{i_o+1}^{(\sig-1) p}}{(1-\alpha s_{i_o})^pR^{(\sig-1) p}}
       \Big(\frac{R}{R-r}\Big)^{(N+sp+1)\frac{s_{i_o+1}}{s}} \mathbf{K}^{p}
      +
      \frac{R^p}{d_{i_o+1}^p}\int_{B_{\rho_{i_o}}}|u|^p\dx\bigg]\\
        &\qquad \le
     \frac{C\,\widetilde{C}_{i_o+1} C_{i_o}}{(1-\alpha s_{i_o})^p}
     \frac{|h|^p}{R^p}\bigg[ 
     \widetilde{C}^{p} \Big(\frac{R}{R-r}\Big)^{(N+sp+1)\frac{s_{i_o+1}}{s}} \mathbf{K}^{p}
      +
      \Big(\frac{R}{R-r}\Big)^{p} \underbrace{R^N\|u\|_{L^{\infty}(B_{R})}^p}_{\le \mathbf{K}^p}\bigg]\\
      &\qquad\le
      \frac{C\,\widetilde{C}_{i_o+1} C_{i_o}}{(1-\alpha s_{i_o})^p}
      \frac{|h|^p}{R^p}\Big(\frac{R}{R-r}\Big)^{(N+sp+1)\frac{s_{i_o+1}}{s}} \mathbf{K}^{p}\\
      &\qquad\le
      \frac{C\,\widetilde{C}_{i_o+1} C_{i_o}}{(1-\alpha s_{i_o})^p }\frac{|h|^p}{R^p}
      \Big(\frac{R}{R-r}\Big)^{\frac{3}{s}(N+sp+1)} \mathbf{K}^{p},
\end{align*}
where $C=C(p)$ and to obtain the second-to-last line we used $(N+sp+1)\frac{s_{i_o+1}}{s}>p$ to combine the two terms with $\frac{R}{R-r}$. In the last line we used 
\begin{align*}
    (N+sp+1)\frac{s_{i_o+1}}{s}
    =
    (N+sp+1)\bigg(1+\frac2p+\Big(\frac2p\Big)^2 \frac{s_{i_o-1}}{s}\bigg)
    \le 
    \frac{3}{s}(N+sp+1) .
\end{align*}
Finally, we discuss the dependence of the constants with respect to $s$. In view of the definitions of $\widetilde{C}_{i_o+1}$, $C_{i_o}$, and $\alpha$, in particular that 
$$
    1 -\alpha s_{i_o}
    \ge 
    1-\frac{p(2-s)}{p+2}=\frac{sp-(p-2)}{p+2}>0
$$
holds true (cf.~\eqref{est:s-alpha_i}),
we have 
\begin{align*}
    \frac{\widetilde{C}_{i_o+1} C_{i_o}}{(1-\alpha s_{i_o})^p}
    &\le 
    \frac{2^{2(Np+N+2p+3)i_o!}}{[s(1-\alpha s_{i_o})^p]^{i_o+1}}
    \,C(N,p) \\
    &\le 
    \frac{2^{2(Np+N+2p+3)i_o!}}{s^{i_o+1}[(sp-(p-2)]^{i_o+1}}
    \,C(N,p).
\end{align*}
Hence, in principle, the constant depends on $N$, $p$ and $i_o$, but since $i_o$ depends only on $s$ and $p$,
the dependence is on $N$, $p$ and $s$. 
Recall that $i_o\to1$ as $s\uparrow1$ and $i_o\to\infty$ as $s\downarrow\frac{p-2}{p}$. 
Hence, the constant $\widetilde C$ is stable as $s\uparrow1$ while it blows up in the limit $s\downarrow\frac{p-2}{p}$.
More precisely, the $s$-dependence of the constant has three components; on the one hand there is the dependence on $s$ via $i_o$ in the exponent. On the other hand there is a contribution that behaves as $1/s$ and one that behaves as $(sp-(p-2))^{-1}$. This is exactly the part in which the condition on $s$ is embodied. At this point, we explicitly emphasize that  $\alpha s_{i_o}$ is fixed in such a way that its distance to $1$
is precisely $\frac{sp -(p-2)}{p+2}$; see Remark~\ref{Rmk:alpha}.
Since $\rho_{i_o+1}>r$, the claim \eqref{est:tau_hu^p} follows from the preceding estimate for $0<|h|\le d_{i_o+1}$. 

For $|h|\in (d_{i_o+1}, R-r]$ the $L^\infty$-bound for $u$, the lower bound for $|h|$, and
the fact that $p<N+sp+1$ easily imply
\begin{align*}
    \int_{B_{\rho_{i_o+1}}} |\btau_h  u|^p \,\dx
    &\le 
    C(N,p)R^N\| u\|_{L^\infty (B_R)}
    \le
    C(N,p) \mathbf{K}^p \frac{|h|^p}{d_{i_o+1}^p}\\
    &\le 
    C(N,p,i_o)\frac{|h|^p}{R^p}\Big(\frac{R}{R-r}\Big)^p\mathbf{K}^p\\
    &\le
    C(N,p,i_o)\frac{|h|^p}{R^p}\Big(\frac{R}{R-r}\Big)^{\frac3s (N+sp+1)}\mathbf{K}^p.
\end{align*}
This proves \eqref{est:tau_hu^p} also in this case.
\end{proof}

\begin{remark}\label{Rmk:alpha}\upshape
The purpose of introducing $\alpha$ in \eqref{def:alpha} is to stabilize our estimates as $s\uparrow1$. Indeed, if we had applied Lemma~\ref{lem:Domokos} with $\gm=s_i$ at each step, we would have ended up with a constant depending on $(1-s_{i_o-1})^{-1}$ at step $i_o-1$  and on $(s_{i_o}-1)^{-1}$ at step $i_o$.
This would have resulted in a discontinuous dependence of the constants on $s$; see Figure~\ref{fig:1}. 

\vspace{1cm}

\begin{figure}[h]
\begin{center}
    \includegraphics[scale=0.3]{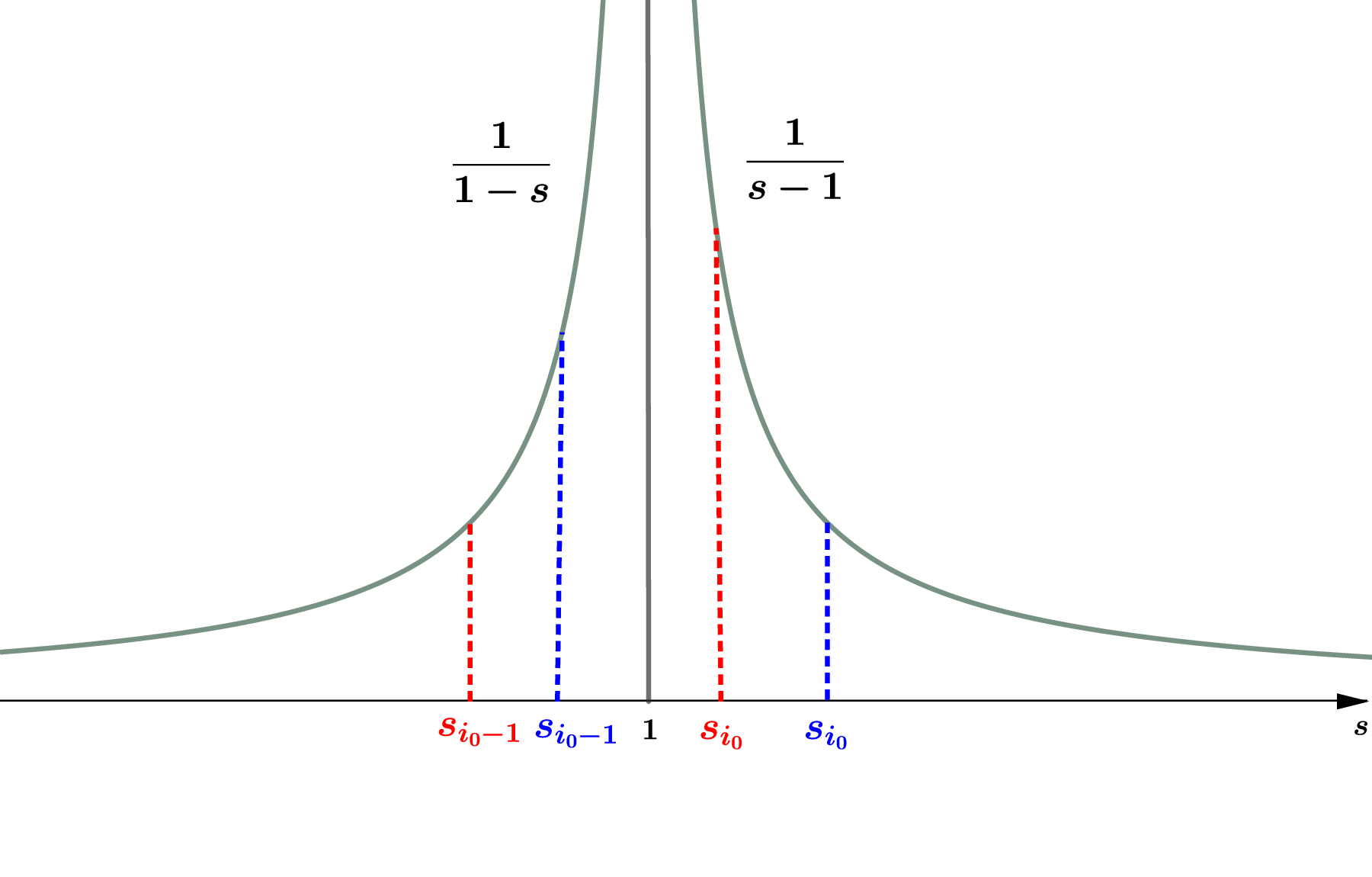}
    \end{center}
     \vspace{-1.30truecm}
    \caption{Behaviour of the constants without stabilization}\label{fig:1}
\end{figure}

\vspace{1cm}

In order to stabilize the constants, each $s_i$ is multiplied by  $\alpha$ so that the last step prior to reaching weak differentiability of $u$ can be quantified as in \eqref{est:s-alpha_i}. Note that, as long as $i\le i_o$, we apply Lemma \ref{lem:Domokos} with $\sigma =\alpha s_i$. This process increases the fractional differentiability of $u$. In the final iteration step from $i_o$ to $i_o+1$ we jump over the critical value of $1$ in Lemma \ref{lem:Domokos} obtaining the weak differentiability of $u$ in $W^{1,p}$. More precisely, we reach $\sigma=s+\frac2p \alpha s_{i_o}=\frac{sp+4}{p+2} =
2-\alpha s_{i_o}>1$ in the estimate for the second order finite differences of $u$; see \eqref{def:sigma}. The particular choice of $\alpha$ ensures on the one hand that $1-\alpha s_i\ge 1-\alpha s_{i_o}= \frac{sp-(p-2)}{p+2}$ for any $i\le i_o$, and on the other hand that also
$\sigma -1=1-\alpha s_{i_o}= \frac{sp-(p-2)}{p+2}$; see Figure~\ref{fig:2}. This strategy leads to quantitative constants which are continuous and stable as $s\uparrow1$ and which only blow up as $s\downarrow \frac{p-2}{p}$.
\begin{figure}[h]
\begin{center}
\includegraphics[scale=0.30]{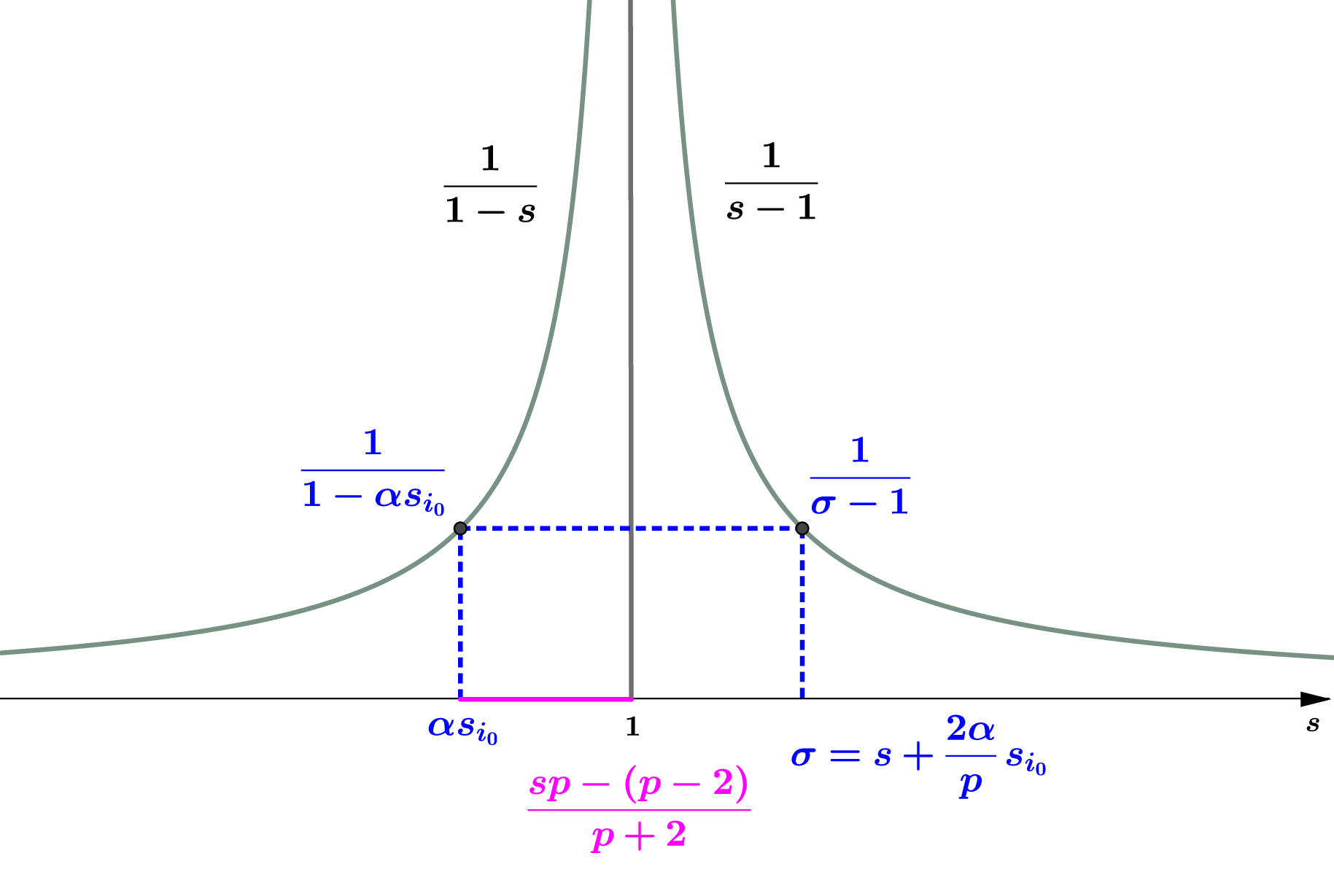}
    \end{center}
     \vspace{-0.70truecm}
    \caption{Behaviour of the constants with stabilization}\label{fig:2}
\end{figure}
\end{remark}

\newpage

Combination of Lemma \ref{lem:diff-quot-p} with the standard estimate for difference quotients from Lemma \ref{lem:diff-quot-1} immediately yields that locally bound local weak solutions of the fractional $p$-Laplace equation are of class $W^{1,p}_{\rm loc}(\Omega)$. As already mentioned, Theorem~\ref{thm:W1p} follows from Corollary~\ref{cor:W1p} by choosing $r=\frac12 R$. 

\begin{corollary}\label{cor:W1p}
Let $p\ge 2$ and $s\in(\frac{p-2}{p},1)$. Then, for any locally bounded, local weak solution $u\in W^{s,p}_{\rm loc}(\Omega)\cap L^{p-1}_{sp}(\R^N)$ of~\eqref{PDE} in the sense of Definition~\ref{def:loc-sol} we have
$$
    u\in W^{1,p}_{\rm loc}(\Omega).
$$
Moreover, there exists a constant $C=C(N,p,s)$ such that for any $B_{R}\equiv B_{R}(x_o)\Subset \Omega$ and $r\in(0,R)$ the quantitative $W^{1,p}$-estimate
\begin{align*}
    \int_{B_r} |\nabla u|^p \,\dx 
    &\le 
    \frac{C}{R^{p}} \Big(\frac{R}{R-r}\Big)^{\frac3s(N+sp+1)} 
    \mathbf{K}^{p},
\end{align*}
holds true, where
\begin{align*}
    \mathbf{K}^p
    :=
    R^{sp}(1-s)[u]^p_{W^{s,p}(B_{R})} +
    R^{N} \big(\|u\|_{L^\infty(B_{R})} + \Tail(u;R)\big)^p .
\end{align*}
Moreover, the constant $C$ is stable as $s\uparrow 1$.
\end{corollary}

\begin{remark}\upshape
In the case $p>2$, Corollary~\ref{cor:W1p}
improves \cite[Corollary~2.8]{Brasco-Lindgren} in the sense that the condition on $s$ can be weakened from $s\in (\frac{p-1}{p},1)$ to $s\in (\frac{p-2}{p},1)$. This  fits perfectly with the case $p=2$ in the limit $p\downarrow 2$.
Indeed, for $p=2$ we have for any $s\in (0,1)$ that $u\in W^{1,2}_{\rm loc}(\Omega)$. 
For linear equations with more general kernels we refer to \cite{Diening-Kim-Lee-Nowak, Cozzi-2}; see also \cite{Brasco-Lindgren}.
\hfill $\Box$
\end{remark}

\subsection{Higher integrability of the gradient}\label{sec:W1q}
The iteration scheme we set up in the previous section improves the regularity from $W^{s,p}_{\rm loc}$ to $W^{1,p}_{\rm loc}$, while the iteration argument in this section will improve the $W^{1,p}_{\rm loc}$-regularity to $W^{1,q}_{\rm loc}$. The starting point is an improved estimate for the second-order finite differences under the assumption that $u\in W^{1,q}(B_R)$ for some $q\ge p$. This assumption allows us to apply the energy estimate from Proposition~\ref{prop:energy} with $\delta =q+p-1$. This in turn leads to an estimate of $[\eta V_{\frac{q}{p}}(\btau_hu)]_{W^{s,p}(B_R)}^p$ in the terms of $|h|^{q-(p-2)}$. 
With  Lemma \ref{lem:N-FS}, i.e.~the embedding of the fractional Sobolev space into a certain Nikol'skii space, we deduce that the integral of the first-order finite difference $|\btau_\lambda(\eta V_{\frac{q}{p}}(\btau_hu))|^p$ is estimated in terms of the product $|\lambda|^{sp}|h|^{q-(p-2)}$, which leads directly to the assertion.

\begin{lemma}\label{lem:second-diff-theta}
Let $p\in[2,\infty)$, $s\in(\frac{p-2}{p},1)$, and $q\ge p$. There exists a constant $C$ that depends on $N$, $p$, and $q$, so that whenever $u\in W^{s,p}_{\rm loc}(\Omega)\cap L^{p-1}_{sp}(\R^N)$ is a locally bounded, local weak solution of~\eqref{PDE} in the sense of Definition~\ref{def:loc-sol} that satisfies
$$
u\in W^{1,q}(B_R)
$$
on some ball $B_{R}\equiv B_{R}(x_o)\Subset \Omega$, then for any $r\in(0,R)$ and any $h\in\R^N\setminus\{0\}$ with $|h|\le d=\frac1{7}(R-r)$, we have
\begin{align}\label{est:tau_h(tau_hu)}
    \int_{B_{r}}\big| \btau_h(\btau_hu)\big|^{q}\,\dx
    &\le
    \frac{C}{s}\frac{|h|^{q+sp-(p-2)}}{R^{q+sp-(p-2)}}
    \Big(\frac{R}{R-r}\Big)^{N+sp+1}
    \mathbf{K}^{q}.
\end{align}
Here, we  used the short-hand notation 
\begin{align}\label{def:K-2diff}
    \mathbf{K}^{q}
    :=
    R^{q}\int_{B_{R}}|\nabla u|^{q}\,\dx +
    R^{N}\big(\|u\|_{L^{\infty}(B_{R})}+\mathrm{Tail}(u;R)\big)^{q}.
\end{align}
\end{lemma}

\begin{proof}
Let $\delta:= 1-p+q\ge 1$. Note that $sp+\delta-1=q-(1-s)p$. We apply the energy inequality from
Proposition~\ref{prop:energy} with $r,R,d$ replaced by $\tilde r=\frac17(5r+2R)$,
$\widetilde R=\frac17(r+6R)$, and $d=\frac14(\widetilde R-\tilde r)=\frac17(R-r)$.
Using inequalities~\eqref{rtilde} and~\eqref{s-l-rtilde} from the proof of Lemma~\ref{lem:frac-impr} allows us to replace $\frac{\widetilde R}{\widetilde R-\tilde r}$ by $\frac{R}{R-r}$ and $\frac1{\widetilde R}$ by $\frac1R$ when applying Proposition \ref{prop:energy} apart from a multiplicative constant depending only on $N$ and $p$.
With the abbreviations
\begin{align*}
    \mathbf I
    &:=
    \iint_{K_{\widetilde R}} 
    \frac{\big|V_{\frac{q}{p}}(\btau_hu(x)) \eta(x) - V_{\frac{q}{p}}(\btau_hu(y)) \eta(y)\big|^{p}}{|x-y|^{N+sp}} \,\dx\dy 
    =
    \big[\eta V_\frac{q}{p}(\btau_hu) \big]_{W^{s,p}(B_{\widetilde R})}^p
\end{align*}
and
\begin{align*}
    \mathcal T
    =
    \|u\|_{L^\infty(B_R)}+\Tail( u;R) 
\end{align*}
the application of Proposition \ref{prop:energy} yields
\begin{align*}
    \mathbf I 
    &\le 
    \frac{C}{(R-r)^2}\bigg[\iint_{K_{R}} \frac{|u(x)-u(y)|^{q}}{|x-y|^{N+q-(1-s)p}} \,\dx\dy \bigg]^{\frac{p-2}{q}}
    \bigg[\frac{R^{(1-s)p}}{1-s}
    \int_{B_{\widetilde R}} |\btau_h u|^{q} \,\dx\bigg]^{\frac{q-(p-2)}{q}}\\
    &\quad+
    \frac{C}{R^{sp}}\Big(\frac{R}{R-r}\Big)^{N+sp}
    \mathcal T^{p-2}\int_{B_{\widetilde R}}\frac{|\btau_h u|^{q-(p-2)}}{1-s}\,\dx\\ 
    &\quad +
   \frac{C |h|}{R^{sp+1}} 
    \Big(\frac{R}{R-r}\Big)^{N+sp+1}
    \mathcal T^{p-1}\int_{B_{\widetilde R}} |\btau_hu|^{q-(p-1)} \,\dx \\
    &=:
    \mathbf I_1 + \mathbf I_2 + \mathbf I_3
\end{align*}
for any $0<|h|\le d=\frac17(R-r)$.
Here, $\eta\in C^1_0(B_{\frac12 (\widetilde{R}+\tilde{r})})$ denotes the usual cut-off function in $\frak Z_{\tilde{r}, \widetilde R} $, cf.~Definition \ref{Def:Z}, satisfying $\eta =1$ on $B_{\tilde{r}}$ and $|\nabla\eta|\le \frac{4}{\widetilde R-\tilde r}=\frac{7}{R-r}$.
Note that $C$ 
has the structure $\widetilde C(N,p) 8^\delta \delta^ps^{-1}\equiv \widetilde C(N,p,q)s^{-1}$.
In the preceding inequality we can replace $\widetilde R$ and $\tilde r$ by the corresponding quantities $R$ and $r$ apart from a multiplicative constant,
since $R>\widetilde R>\tfrac67 R$,  and 
$\widetilde R-\tilde r\ge \tfrac47 (R-r)$.
Let us estimate $\mathbf I_1$ using the assumption $u\in W^{1,q}(B_R)$. In fact, applying the embedding Lemma~\ref{lem:FS-S} with exponent $q$,  and with $\gamma$ replaced by $ 1-\frac{(1-s)p}{q}$  we obtain
\begin{align*}
     \iint_{K_{ R}}
     \frac{|u(x)-u(y)|^{q}}{|x-y|^{N+q-(1-s)p}}
     \,\dx\dy
     &\le 
     \frac{C R^{(1-s)p}}{1-s}\int_{B_{ R}}|\nabla u|^{q}\, \dx
\end{align*}
for a constant $C=C(N,p,q)$.
Whereas the integral of  $|\btau_hu|$ is bounded by the corresponding integral of $|\nabla u|$ according to Lemma \ref{lem:diff-quot-2}. That is,
\begin{align*}
     \int_{B_{\widetilde R}} |\btau_h u|^{q} \,\dx\le|h|^{q}\int_{B_{ R}}|\nabla u|^{q}\, \dx.
\end{align*}
Hence, we have
\begin{align*}
    \mathbf I_1 
    &\le 
    \frac{C}{s(1-s)R^{sp}} \Big(\frac{R}{R-r}\Big)^2 |h|^{q-(p-2)} R^{p-2}
    \int_{B_{R}} |\nabla u|^{q} \,\dx \\
    &=
    \frac{C}{s(1-s)R^{sp}} \Big(\frac{R}{R-r}\Big)^2 \frac{|h|^{q-(p-2)} }{R^{q -(p-2)}}
    R^q \int_{B_{R}} |\nabla u|^{q} \,\dx\\
    &\le 
    \frac{C}{s(1-s)R^{sp}} \Big(\frac{R}{R-r}\Big)^2 \frac{|h|^{q-(p-2)} }{R^{q -(p-2)}}\mathbf{K}^q,
\end{align*}
where $C=C(N,p,q)$.
For $\mathbf I_2$ we first apply H\"older's inequatilty in order to raise the exponent of $|\btau_hu|$ from $q -(p-2)$ to $q$, then Lemma \ref{lem:diff-quot-2} and finally the definition of $\mathbf{K}$. This amounts to  
\begin{align*}
    \mathbf I_2 
    &\le 
    \frac{C}{s(1-s)R^{sp}}\Big(\frac{R}{R-r}\Big)^{N+sp}
    R^{\frac{N(p-2)}{q}} 
    \mathcal T^{p-2}
    \bigg[\int_{B_{\widetilde R}}|\btau_h u|^{q}\,\dx\bigg]^{\frac{q-(p-2)}{q}} \\
    &\le 
    \frac{C}{s(1-s)R^{sp}}\Big(\frac{R}{R-r}\Big)^{N+sp} \frac{|h|^{q-(p-2)}}{R^{q-(p-2)}}
    \big[R^N\mathcal T^{q}\big]^\frac{p-2}{q}
    \bigg[R^q\int_{B_{R}}|\nabla u|^{q}\,\dx\bigg]^{\frac{q-(p-2)}{q}} \\
    &\le 
    \frac{C}{s(1-s)R^{sp}}\Big(\frac{R}{R-r}\Big)^{N+sp} 
    \frac{|h|^{q-(p-2)}}{R^{q-(p-2)}} \mathbf{K}^q,
\end{align*}
again with a constant $C$ depending only on $N$, $p$, and $q$.
Here, to obtain the last line we also used the definition of $\mathbf{K}$. The last integral  $\mathbf I_3$ is treated similarly.  Firstly, we apply H\"older's inequatilty in order to raise the exponent of $|\btau_hu|$ from $q -(p-1)$ to $q$, then Lemma \ref{lem:diff-quot-2} and finally the definition of $\mathbf{K}$. This leads to
\begin{align*}
    \mathbf I_3
    &\le 
    \frac{C}{sR^{sp}}\Big(\frac{R}{R-r}\Big)^{N+sp+1}\frac{|h|}{R}
    \big[R^N\mathcal T^q \big]^\frac{p-1}{q}
    \bigg[\int_{B_{\widetilde R}}|\btau_h u|^{q}\,\dx\bigg]^{\frac{q -(p-1)}{q}} \\
    &\le 
    \frac{C}{sR^{sp}}\Big(\frac{R}{R-r}\Big)^{N+sp+1} \frac{|h|^{q-(p-2)}}{R^{q-(p-2)}}
    \big[R^N\mathcal T^q \big]^\frac{p-1}{q}
    \bigg[R^q\int_{B_{R}}|\nabla u|^{q}\,\dx\bigg]^{\frac{q -(p-1)}{q}} \\
    &\le 
    \frac{C}{sR^{sp}}\Big(\frac{R}{R-r}\Big)^{N+sp} \frac{|h|^{q-(p-2)}}{R^{q-(p-2)}} \mathbf{K}^q,
\end{align*}
where $C=C(N,p,q)$.
Inserting these estimates above, we have shown that 
\begin{align}\label{est:I}
    \mathbf I
    &\le
    \frac{C(N,p,q)}{s(1-s)R^{sp}}
    \Big(\frac{R}{R-r}\Big)^{N+sp+1} 
    \frac{|h|^{q-(p-2)}}{R^{q-(p-2)}}
    \mathbf{K}^{q}\qquad\forall \, 0<|h|\le d.
\end{align}
Now, we apply Lemma~\ref{lem:N-FS} to $v=\eta V_\frac{q}{p}(\btau_hu)$ on $B_{\widetilde R}$ with $(p, s, d=\frac17(R-r))$ instead of $(q, \gamma, d)$. 
Note that in the application $\widetilde R$ plays the role of $R+d$. Taking into account that $d\le R$, $\eta\le 1$,  Lemma \ref{lem:diff-quot-2}, and \eqref{est:I} we find that
\begin{align*}
    \int_{B_{\widetilde R-d}} &
    \big|\btau_\lambda\big(V_{\frac{q}{p}}(\btau_hu) \eta\big)\big|^p \,\dx \\
    &\le
    C |\lambda|^{sp}
     \bigg[(1-s)\mathbf I +
    \bigg(\frac{\widetilde R^{(1-s)p}}{d^p} +\frac{1}{sd^{sp}}\bigg)
    \int_{B_{\widetilde R}} \big|V_{\frac{q}{p}}(\btau_hu) \eta\big|^p \,\dx\bigg]\\
    &\le 
    C|\lambda|^{sp} \bigg[(1-s)\mathbf I +
    \frac{R^{(1-s)p}}{sd^p} 
    \int_{B_{\widetilde R}} \big|V_{\frac{q}{p}}(\btau_hu) \eta\big|^p \,\dx\bigg]  \\
    &\le 
    C|\lambda|^{sp} \bigg[(1-s)\mathbf I +
    \frac{R^{(1-s)p}}{sd^p} 
    \int_{B_{\widetilde R}} |\btau_hu|^{q} \,\dx\bigg] \\
    &\le 
    C|\lambda|^{sp} \bigg[(1-s)\mathbf I +
    \frac{|h|^{q}}{sR^{sp}} \Big(\frac{R}{R-r}\Big)^p 
    \int_{B_R} |\nabla u|^{q} \,\dx\bigg] \\
    &\le
    C|\lambda|^{sp} \bigg[
    \frac{C}{sR^{sp}}\Big(\frac{R}{R-r}\Big)^{N+sp+1} \frac{|h|^{q-(p-2)}}{R^{q-(p-2)}}
    \mathbf{K}^{q} +
    \frac{1}{sR^{sp}} \Big(\frac{R}{R-r}\Big)^p \Big(\frac{|h|}{R}\Big)^{q}\mathbf{K}^{q} \bigg]\\
    &\le
    \frac{C}{s}
    \frac{|\lambda|^{sp}}{R^{sp}}\Big(\frac{R}{R-r}\Big)^{N+sp+1}\frac{|h|^{q-(p-2)}}{R^{q-(p-2)}}
    \mathbf{K}^{q} 
\end{align*}
for any $|\lambda|\le d$ and with a constant $C=C(N,p,q)$.
Now we choose $\lambda=h$, and observe that, since $\eta\equiv 1$ in $B_{\tilde r}=B_{r+2d}$ and $\frac{q}{p}\ge 1$ there holds 
\begin{align*}
    \big|\btau_h\big(V_{\frac{q}{p}}(\btau_hu) \eta\big)\big|
    =
    \big|\btau_h\big(V_{\frac{q}{p}}(\btau_hu) \big)\big| 
    \ge 
    |\btau_h(\btau_h u)|^{\frac{q}{p}}
    \qquad \mbox{in $B_{r}$.}
\end{align*}
Therefore, with a constant $C=C(N,p,q)$ we get
\begin{align*}
    \int_{B_{r}}\big| \btau_h(\btau_hu)\big|^{q}\,\dx
    &\le
    \frac{C}{s}\frac{|h|^{q+sp-(p-2)}}{R^{q+sp-(p-2)}}
    \Big(\frac{R}{R-r}\Big)^{N+sp+1}
    \mathbf{K}^{q},
\end{align*}
for any $0<|h|\le d$. This proves \eqref{est:tau_h(tau_hu)}. 
\end{proof}

At this point it is worthwhile to note that the exponent of the increment $|h|$ in Lemma~\ref{lem:second-diff-theta} is larger than the integrability exponent $q$. This allows to apply Lemma~\ref{lem:2nd-Ni-FS} to conclude that the gradient admits a certain fractional differentiability.

\begin{proposition}\label{prop:beta-q}
Let $p\in[2,\infty)$, $s\in(\frac{p-2}{p},1)$, and $q\in[p,\infty)$. Then,  for any locally bounded  local weak solution $u\in W^{s,p}_{\rm loc}(\Omega)\cap L^{p-1}_{sp}(\R^N)$ of~\eqref{PDE} in the sense of Definition~\ref{def:loc-sol}, satisfying
\[
    u\in W^{1,q}_{\rm loc}(\Omega),
\]
we have
$$
    \nabla u\in W^{\alpha,q}_{\rm loc}(\Omega)\qquad 
    \mbox{for any $\alpha\in(0,\beta)$, where $\beta:=\tfrac{sp-(p-2)}{q}$.}
$$
Moreover, there exists a constant $C=C(N,p,s,q,\alpha)$ such that for any ball $B_{R}\equiv B_{R}(x_o)\Subset \Omega$ and for any $r\in(0,R)$, we have
\begin{align*}
    [\nabla u]_{W^{\alpha, q}( B_{r})}^{q}
    \le
    \frac{C}{R^{(1+\alpha) q}} 
    \Big(\frac{R}{R-r}\Big)^{N+q+1}
    \mathbf{K}^{q},
\end{align*}
where $\mathbf{K}$ is defined in~\eqref{def:K-2diff} and the constant is of the form $C=\frac{\widetilde{C}(N,p,q)}{s\alpha\beta^{q} (\beta-\alpha)(1-\alpha)^{q}}$.
\end{proposition}

\begin{proof}
We apply Lemma~\ref{lem:second-diff-theta} with $r$ replaced by $\tilde r=\frac{1}{11}(7r+4R)$ and leave the larger radius $R$ unchanged in the application. Taking into account that $R-\tilde r=\frac{7}{11}(R-r)$, we obtain
\begin{align*}
    \int_{B_{\tilde r}}\big| \btau_h(\btau_hu)\big|^{q}\,\dx
    &\le
    \frac{C}{s}\frac{|h|^{q+sp-(p-2)}}{R^{q+sp-(p-2)}}
    \Big(\frac{R}{R-\tilde r}\Big)^{N+sp+1}
    \mathbf{K}^{q} \\
    &\le
    \frac{C}{s}\frac{|h|^{q(1+\beta)}}{R^{q(1+\beta)}}
    \Big(\frac{R}{R-r}\Big)^{N+sp+1}
    \mathbf{K}^{q} ,
\end{align*}
for any $h\in\R^N$ with $0<|h|\le d=\frac17(R-\tilde r)=\frac{1}{11}(R-r)$ and where $\mathbf{K}$ is defined in~\eqref{def:K-2diff} and the constant $C$ depends only on $N,p,q$. We now fix some $\alpha\in(0,\beta)$. 
In the preceding inequality we take advantage of the inequality $|h|\le R$ in order to reduce the power of $|h|$ from $q(1+\beta)$ to $q(1+\tilde\beta)$, where $\tilde\beta=\frac12(\alpha+\beta)$. We thus obtain
\begin{align*}
    \int_{B_{\tilde r}}\big| \btau_h(\btau_hu)\big|^{q}\,\dx
    &\le
    \frac{C}{s}\frac{|h|^{q(1+\tilde\beta)}}{R^{q(1+\tilde\beta)}}
    \Big(\frac{R}{R-r}\Big)^{N+sp+1}
    \mathbf{K}^{q},
\end{align*}
for any $h\in\R^N$ with $0<|h|\le d=\frac17(R-\tilde r)=\frac{1}{11}(R-r)$.
The above estimate  plays the role of assumption \eqref{ass:W^beta,q-second-diff} in Lemma \ref{lem:2nd-Ni-FS}, which we apply on $B_{\tilde r}$ with $d=\frac1{11}(R-r)$; note that $\tilde r=r+4d$. The quantity
$$
    \frac{C}{s R^{q(1+\tilde\beta)}}
    \Big(\frac{R}{R-r}\Big)^{N+sp+1}
    \mathbf{K}^{q} 
$$
plays the role of $M^q$ in \eqref{ass:W^beta,q-second-diff}. The application is allowed, since $u\in W^{1,q}(B_{r+6d})$ by assumption. Note that $d=\frac1{11}(R-r)$, and $\tilde r =r+4d$. In particular, we have $\nabla u\in W^{\alpha,q}( B_{r})$ with the quantitative estimate 
\begin{align*}
    &[\nabla u]_{W^{\alpha, q}(B_{r})}^{q}\\
    &\quad\le
    \frac{C\,d^{q(\tilde\beta-\alpha)}}{(\tilde\beta-\alpha)\tilde\beta^{q}(1-\tilde\beta)^{q}}
    \bigg[
    \frac{C}{sR^{q(1+\tilde\beta)}} \Big(\frac{R}{R-r}\Big)^{N+sp+1}\mathbf{K}^{q}
    +
    \frac{\tilde r^{q}}{\alpha d^{q(1+\tilde\beta)}}
    \int_{B_{\tilde r}}|\nabla u|^{q}\,\dx\bigg]\\
    &\quad\le
     \frac{C}{s(\beta-\alpha)\beta^{q}(1-\alpha)^{q}}
     \bigg[
     \frac{d^{q(\tilde\beta-\alpha)}}{R^{q(1+\tilde\beta)}} \Big(\frac{R}{R-r}\Big)^{N+sp+1}\mathbf{K}^{q}
     +
     \frac{\tilde r^{q}}{\alpha d^{q(1+\alpha)}}
    \int_{B_{R}}|\nabla u|^{q}\,\dx
     \bigg]\\
     &\quad\le
     \frac{C}{s(\beta-\alpha)\beta^{q}(1-\alpha)^{q}}
     \bigg[
     \frac{ 1}{R^{q(1+\alpha)}} \Big(\frac{R}{R-r}\Big)^{N+sp+1}\mathbf{K}^{q}
     +
     \frac{R^{q}}{\alpha d^{q(1+\alpha)}}
    \int_{B_{R}}|\nabla u|^{q}\,\dx
     \bigg]\\
     &\quad\le
     \frac{C}{s\alpha (\beta-\alpha)\beta^{q}(1-\alpha)^{q}R^{q(1+\alpha)}}
     \bigg[
      \Big(\frac{R}{R-r}\Big)^{N+sp+1}
     +
     \Big( \frac{R}{R-r}\Big)^{q(1+\alpha)}
     \bigg]\mathbf{K}^q\\
     &\quad\le
     \frac{C}{s\alpha (\beta-\alpha)\beta^{q}(1-\alpha)^{q}R^{q(1+\alpha)}}
      \Big(\frac{R}{R-r}\Big)^{N+q+1}\mathbf{K}^{q}, 
\end{align*}
where $C=C(N,p,q)$. From the second to third line we used $\tilde\beta-\alpha=\frac12(\beta-\alpha)$ and $1-\tilde\beta\ge\frac12(1-\alpha)$. 
To obtain the last line we estimated the exponents of $R/(R-r)$ by
\begin{align*}
    N+sp +1
    <
    N+q+1  ,
\end{align*}
and
\begin{align*}
    q(1+\alpha)
    <
    q+q\beta =q +sp-(p-2)
    \le 
    q+2<N+q+1.
\end{align*}
This proves the claim.
\end{proof}

The gain in fractional differentiability from the preceding proposition can be exploited via the Sobolev embedding for fractional Sobolev spaces to increase the integrability exponent of the gradient. 

\begin{lemma}\label{lem:increase-exp}
Let $p\in[2,\infty)$, $s\in(\frac{p-2}{p},1)$, and $q\ge p$. There exists a constant $C$ depending on $N$, $p$, $s$ and $q$, such that whenever $u\in W^{s,p}_{\rm loc}(\Omega)\cap L^{p-1}_{sp}(\R^N)$ is a local weak solution of~\eqref{PDE} in the sense of Definition~\ref{def:loc-sol}, satisfying
\[
    u\in W^{1,q}_{\rm loc}(\Omega),
\]
we have 
$$
    \nabla u
    \in 
    L^{\frac{Nq}{N-\alpha q}}_{\rm loc}(\Omega),
    \qquad \mbox{where $\alpha=\frac{sp-(p-2)}{2q}$.}
$$
Moreover, there exists a constant $C=C(N,p,s,q)$, so that for any ball $B_{R}\equiv B_{R}(x_o)\Subset \Omega$ and for any $r\in(0,R)$ we have
\begin{align*}
    \bigg[\mint_{B_{r}}| \nabla u|^{\frac{Nq}{N-\alpha q}}\,\dx\bigg]^{\frac{N-\alpha q}{N}}
    &\le
    C\Big(\frac{R}{r}\Big)^N \Big(\frac{R}{R-r}\Big)^{N+q+1}
   \mathbf M^{q},
\end{align*}
where 
\begin{align*}
   \mathbf M
    :=
    \bigg[\mint_{B_{R}}|\nabla u|^{q}\,\dx\bigg]^\frac{1}{q} +
    \frac{1}{R}\big(\|u\|_{L^{\infty}(B_{R})}+\mathrm{Tail}(u;R)\big).
\end{align*}
Moreover, the constant $C$ is stable as $s\uparrow1$ and blows up as $s\downarrow \frac{p-2}{p}$.
\end{lemma}

\begin{proof}
We first observe that 
$\alpha= \frac{sp-(p-2)}{2q}\in(0,\beta)$, where $\beta=\tfrac{sp-(p-2)}{q}$. 
Therefore, Proposition~\ref{prop:beta-q} ensures that $\nabla u\in W^{\alpha,q}_{\rm loc}(\Omega)$ together with the quantitative estimate (note that for the quantity $\mathbf{K}$ defined in~\eqref{def:K-2diff} we have $\mathbf{K}\le C(N) R^{N+q}\mathbf{M}$)
\begin{align*}
    [\nabla u]_{W^{\alpha, q}( B_{r})}^{q}
    &\le
    \frac{C}{R^{(1+\alpha) q}} 
    \Big(\frac{R}{R-r}\Big)^{N+q+1}
    \mathbf{K}^{q} \\
    &\le
    \frac{C_1 R^N}{R^{\alpha q}} 
    \Big(\frac{R}{R-r}\Big)^{N+q+1}
    \mathbf{M}^{q},
\end{align*}
for a constant $C_1$ of the form 
\begin{equation*}
    C_1
    =
    \frac{C(N,p,q)}{s\alpha\beta^{q} (\beta-\alpha)(1-\alpha)^{q}}
    =
    \frac{C(N,p,q)}{s (sp-(p-2))^{q+2}}.
\end{equation*}
In order to specify the dependency of $C_1$ we used the fact that $\alpha\le\frac12s<\frac12$ and hence $1-\alpha>\frac12$. 
By the Sobolev embedding for fractional Sobolev spaces from Lemma~\ref{lem:frac-Sob-2} we conclude that $\nabla u\in L^{\frac{Nq}{N-\alpha q}}(B_{r})$ together with the quantitative estimate
\begin{align*}
    \bigg[\mint_{B_{r}}&| \nabla u|^{\frac{Nq}{N-\alpha q}}\,\dx\bigg]^{\frac{N-\alpha q}{N}}\\
    &\le 
    2^{q-1}\bigg[
    C_2^q r^{\alpha q-N} [\nabla u]_{W^{\alpha, q}( B_{r})}^{q} +
    \mint_{B_r} |\nabla u|^q \,\dx \bigg] \\
    &\le
    2^{q-1}\bigg[ 
    C_2^qC_1 \Big(\frac{R}{r}\Big)^N \Big(\frac{R}{R-r}\Big)^{N+q+1} +
    \Big(\frac{R}{r}\Big)^N \mint_{B_R} |\nabla u|^q \,\dx 
    \bigg] 
    \\
    &\le
    2^{q-1}\big(C_2^qC_1+1\big) 
    \Big(\frac{R}{r}\Big)^N \Big(\frac{R}{R-r}\Big)^{N+q+1}  \mathbf M^{q},
\end{align*}
where $C_2=C_2(N,q,\alpha)$ denotes the constant from Lemma~\ref{lem:frac-Sob-2}.
The application is allowed since $\alpha q = 1-\frac12 p(1-s)<N$. This proves the claimed inequality.
\end{proof}

\begin{remark}\upshape
A few words concerning stability of the constants $C_1$ and $C_2$ are in order. For $C_2$, Lemma~\ref{lem:frac-Sob-2} and the subsequent remark, in conjunction with the specific choice of $\beta$, show that
\begin{align*}
    C_2^q
    &=
    \frac{C(N)(1-\alpha)}{(N-\alpha q)^{q-1}}
    =
    C(N)\frac{1-\frac{1}{2q}(sp -(p-2))}{(N-\frac{1}{2}(sp -(p-2)))^{q-1}}\\
    &=
    \frac{C(N)}{q}\frac{q-1+\frac12(1-s)p}{(N-1+\frac12(1-s)p)^{q-1}},
\end{align*}
which implies
\begin{equation*}
    \lim_{s\uparrow 1} C_2^q=\frac{C(N)(q-1)}{q(N-1)^{q-1}}.
\end{equation*}
Moreover, we have (recall the definition of $C_1$)
\begin{align*}
    \lim_{s\uparrow 1} C_1=\frac{C(N,p,q)}{2^{q+2}}.
\end{align*}
This proves, that the constant remains stable as $s\uparrow 1$. Finally, the definition of $C_1$ 
implies that $C_1$ blows up at $s\downarrow \frac{p-2}{p}$. 
\hfill $\Box$
\end{remark}

Lemma~\ref{lem:increase-exp} allows to set up a Moser-type iteration scheme that improves for any given $q\in[p,\infty)$ the regularity of a locally bounded weak solution of the fractional $p$-Laplacian from $W^{1,p}_{\rm loc}(\Omega)$ to $W^{1,q}_{\rm loc}(\Omega)$.

\begin{proposition}[$W^{1,q}$-gradient regularity]\label{lem:W1q}
Let $p\in[2,\infty)$ and  $s\in(\frac{p-2}{p},1)$. Then,  for any locally bounded local weak solution $u\in W^{s,p}_{\rm loc}(\Omega)\cap L^{p-1}_{sp}(\R^N)$ of~\eqref{PDE} in the sense of Definition~\ref{def:loc-sol}  we have
$$
    u\in W^{1,q}_{\rm loc}(\Omega)\quad \mbox{for any $q\in [p,\infty)$.}
$$
Moreover, there exists a constant  $C=C(N,p,s,q)$, such that on any ball $B_{R}\equiv B_{R}(x_o)\Subset \Omega$ the quantitative $L^q$-gradient estimate 
\begin{align*}
    \bigg[\mint_{B_{R/2}}| \nabla u|^{q}\,\dx\bigg]^{\frac{1}{q}}
    &\le
    C\Bigg[\bigg[\mint_{B_{R}}|\nabla u|^{p}\,\dx\bigg]^\frac{1}{p} +
    \frac{1}{R}\big(\|u\|_{L^{\infty}(B_{R})}+\mathrm{Tail}(u;R)\big)\Bigg]
\end{align*}
holds true. Moreover, the constant $C$ is stable in the limit  $s\uparrow 1$ and blows up as $s\downarrow \frac{p-2}{p}$.
\end{proposition}

\begin{proof}
Based on the   quantitative higher integrability lemma we set up an iteration argument. To this end, we define a sequence $(q_i)_{i\in\N_0}$ of exponents, a sequence $(\rho_i)_{i\in\N_0}$ of radii, and a sequence of shrinking concentric balls $(B_i)_{i\in\N_0}$ by
\begin{equation*}
    \left\{
    \begin{array}{c}
    \displaystyle
    q_o:=p,
    \qquad
    q_{i}=\frac{Nq_{i-1}}{N-\frac12[sp-(p-2)]}
    =
    \bigg(\frac{N}{N-\frac12[sp-(p-2)]}\bigg)^{i}p,\\[12pt]
    \displaystyle
    \rho_i
    :=
    \frac{1}2 \Big(R+\frac{R}{2^{i}}\Big),\qquad B_i=B_{\rho_i}.
     \end{array}
    \right.
\end{equation*}
Clearly $q_i\to \infty$ as $i\to\infty$. For $i\in \N$ we apply Lemma \ref{lem:increase-exp} with $r=\rho_i$, $R=\rho_{i-1}$, 
$q=q_{i-1}$, 
and
$\beta= \beta_{i-1}=\frac12\frac{sp-(p-2)}{q_{i-1}}$. Note that $\beta_{i-1} q_{i-1} =1 -\frac12p (1-s)<N$.
This amounts to
\begin{align}\label{est:int-theta_i}
    \bigg[
    \mint_{B_i}|\nabla u|^{q_i}\,\dx
    \bigg]^\frac1{q_i}
    &\le
    C_i
    \Big(\frac{\rho_{i-1}}{\rho_i}\Big)^{\frac{N}{q_{i-1}}}
    \Big(\frac{\rho_{i-1}}{\rho_{i-1}-\rho_{i}}\Big)^{\frac{N}{q_{i-1}}
    +1}\mathbf M_{i-1},
\end{align} 
for a constant $C_i=C_i(N,p,\beta_{i-1},q_{i-1})$.  Concerning the stability of $C_i$ we refer to Remark \ref{stability-C_1-C_2}. To obtain 
\eqref{est:int-theta_i}$_i$  we abbreviated 
\begin{align*}
   \mathbf M_{i-1}
    &:=
    \bigg[\mint_{B_{i-1}}|\nabla u|^{q_{i-1}}\,\dx   \bigg]^\frac1{q_{i-1}}
    +
    \frac1{\rho_{i-1}}
    \mathcal T_{i-1}
   .
\end{align*}
where
$$
    \mathcal T_i
    :=
    \|u\|_{L^\infty(B_{\rho_{i}})} + \Tail(u;\rho_{i}),\qquad \mathcal T_0\equiv\mathcal T.
$$
To proceed further, we estimate the numerical factors and the tail term. In fact, we have
$$
    \frac{\rho_{i-1}}{\rho_i}
    =
    \frac{R+\frac1{2^{i-1}}R}{R+\frac1{2^{i}}R}<2,
    \quad\quad
    \frac{R}{\rho_{i-1}}
    =
    \frac{2R}{R+\frac1{2^{i-1}}R}
    \le 2,
$$
and
$$
    \frac{\rho_{i-1}}{\rho_{i-1}-\rho_{i}}
    =
    \frac{R+\frac{1}{2^{i-1}}R}{\frac{1}{2^{i-1}}R-\frac{1}{2^{i}}R}
    =
    2^{i}\big(1+\tfrac1{2^{i-1}}\big)
    \le
    2^{i+1}.
$$
Moreover, by Lemma~\ref{lem:t} we have
\begin{align*}
    \Tail (u, \rho_{i-1})^{p-1}
    &\le
    \Big(\frac{R}{\rho_{i-1}}\Big)^{N} 
    \big(\Tail (u, R)^{p-1} +
    \| u\|_{L^\infty (B_R)}\big)^{p-1}
    \le
    C(N)\mathcal T^{p-1}.
\end{align*}
Using this in \eqref{est:int-theta_i}$_i$ we obtain for any $i\in\N$ that
\begin{align*}
    \bigg[
    \mint_{B_i}|\nabla u|^{q_i}\,\dx
    \bigg]^\frac1{q_i}
    &\le
    \underbrace{C(N,p)2^{\frac{iN}{p}}C_i}_{=:\widetilde C_i}\Bigg[
    \bigg[\mint_{B_{i-1}}|\nabla u|^{q_{i-1}}\,\dx   \bigg]^\frac1{q_{i-1}}
    +
    \frac1{R}\mathcal T\Bigg].
\end{align*}
Iterating this inequality results in
 \begin{align*}   
    \bigg[
    \mint_{B_i}|\nabla u|^{q_i}\,\dx
    \bigg]^\frac1{q_i} 
    &\le 
    C\Bigg[
    \bigg[\mint_{B_{R}}|\nabla u|^{p}\,\dx   \bigg]^\frac1{p}
    +
    \frac{1}{R}\mathcal T\Bigg],
\end{align*} 
where the constant $C$ is given by $C=\prod_{j=1}^i \widetilde C_j$.
Since $q_i\to\infty$ as $i\to\infty$ there exists $i_o\in\N$, such that $q_{i_o}\ge q$ and $q_{i_o-1}<q$. This defines $i_o\in\N$ as a function of $N$, $p$, $s$ and $q$. 
More precisely, we have
\begin{align*}
    i_o
    =  \Biggl\lceil\frac{\ln \frac{p}{q}}{\ln \frac{N}{N-\frac12 (sp-(p-2))}}-1\Biggr\rceil,
\end{align*}
so that the number $i_o$ of iterations to reach the $L^q$-integrability 
is bounded by
\begin{align*}
    \frac{\ln \frac{p}{q}}{\ln \frac{N}{N-\frac12 (sp-(p-2))}}\to 
    \frac{\ln \frac{p}{q}}{\ln \frac{N}{N-1}}\quad\mbox{as $s\uparrow 1$.}
\end{align*}
Enlarging the domain of integration from $B_{\frac12 R}$ to $B_{i_o}$ and using H\"older's inequality, we finally get
\begin{align*}
    \bigg[
    \mint_{B_{\frac12 R}}|\nabla u|^{q}\,\dx
    \bigg]^\frac1{q}
    &\le
    2^\frac{N}{q}\bigg[
    \mint_{B_{i_o}}|\nabla u|^{q_{i_o}}\,\dx
    \bigg]^\frac1{q_{i_o}}\le 
    C\Bigg[
    \bigg[\mint_{B_{R}}|\nabla u|^{p}\,\dx   \bigg]^\frac1{p}
    +
    \frac1{R}\mathcal T\Bigg],
\end{align*} 
where $C=C(N,p,s,q)$. This proves the claim.
\end{proof}

\begin{remark}\label{stability-C_1-C_2}\upshape 
The precise value of $C_i$ is given by $C_i=2^{q_{i-1}}\big(C_2^{q_{i-1}}C_1+1\big)$, where
\begin{equation*}
    C_1=\frac{C(N,p,q_{i-1})}{s (sp-(p-2))^{q_{i-1}+2}},
\end{equation*} 
and
\begin{align*}
    C_2^{q_{i-1}}
    &= 
    \frac{C(N)}{q_{i-1}}\frac{q_{i-1}-1+\frac12(1-s)p}{(N-1+\frac12(1-s)p)^{q_{i-1}-1}} \\
    &\le
     \frac{C(N)}{(N-1 +\frac12 p(1-s))^{q_{i-1}-1}}
     \le
      \frac{C(N)}{(N-1 )^{q_{i-1}-1}}.
\end{align*}
Moreover, we have
\begin{align*}
    \lim_{s\uparrow 1}q_{i-1}
    &=
    \lim_{s\uparrow 1} \bigg(\frac{N}{N-\frac12[sp-(p-2)]}\bigg)^{i-1}p
    =\Big(\frac{N}{N-1}\Big)^{i-1}p
\end{align*}
Hence, both constants $C_1$  and $C_2^{q_{i-1}}$ remain stable at $s\uparrow 1$.  Therefore also
$C_i$ is  stable as $s\uparrow 1$. As already pointed out, $i_o$ stabilizes as $s\uparrow 1$, so that also the factor $2^\frac{iN}{p}\le 2^\frac{i_oN}{p}$ in the constant $\widetilde C_i$ remains stable.
\hfill $\Box$
\end{remark}

At this point Theorem~\ref{thm:W1q} can be achieved by combining Theorem~\ref{cor:W1p} and Proposition~\ref{lem:W1q}.

\begin{proof}[\textbf{\upshape Proof of Theorem~\ref{thm:W1q}}]
First we apply Theorem~\ref{cor:W1p} on the balls $B_{\frac12 R}$ and $B_{\frac34 R}$, which is possible after slightly changing the radii. Subsequently we use Theorem~\ref{cor:W1p} to estimate the $L^p$ norm of $\nabla u$ and Lemma~\ref{lem:t} to increase $\tfrac34 R$ in the tail-term to $R$. In this way, we obtain
\begin{align*}
    \|\nabla u\|_{L^q(B_{R/2})}
    &\le
    C R^{\frac{N}{q}} 
    \Big[R^{-\frac{N}{p}} \|\nabla u\|_{L^p(B_{\frac34 R})} +
    \tfrac{1}{R}\Big(\|u\|_{L^{\infty}(B_{R})}+\mathrm{Tail}(u;\tfrac34 R)\Big)\Big] \\
    &\le
    C R^{\frac{N}{q}-1} 
    \Big[R^{s-\frac{N}{p}} (1-s)^{\frac1p}[u]_{W^{s,p}(B_{R})} +
    \|u\|_{L^{\infty}(B_{R})}+\mathrm{Tail}(u;R)\Big].
\end{align*}
This finishes the proof of Theorem~\ref{thm:W1q}.
\end{proof}

\subsection{Almost Lipschitz continuity and improved fractional differentiability}\label{sec:Holder-s>}

By Morrey's embedding the $W^{1,q}$-regularity result obtained in Theorem~\ref{thm:W1q} immediately implies that solutions are Hölder continuous for any Hölder exponent $\gamma\in(0,1)$. This is exactly the content of Theorem~\ref{thm:Hoelder s>}.

\begin{proof}[\textbf{\upshape Proof of Theorem~\ref{thm:Hoelder s>}}]
From Theorem~\ref{lem:W1q} we know $u\in W^{1,q}_{\rm loc}(\Omega)$ for any $q\ge p$. Therefore, by Morrey's embedding, Lemma~\ref{Lem:morrey-classic}, we conclude that $u\in C^{0,\gamma}(\Omega)$ for any $\gamma\in(0,1)$. Now, fix some $\gamma\in(0,1)$ and consider a ball $B_R\equiv B_R(x_o)\Subset\Omega$. Applying in turn Lemma~\ref{Lem:morrey-classic} with the choice $q=\frac{N}{1-\gamma}$ and Theorem~\ref{lem:W1q}, we obtain the quantitative estimate 
\begin{align*}
    [u]_{C^{0,\gamma}(B_{\frac12 R})}
    &=
    [u]_{C^{0,1-\frac{N}{q}}(B_{\frac12 R})}
    \le 
    C\|\nabla u\|_{L^{q}(B_{\frac12 R})} \\
    &\le
    C R^{\frac{N}{q}-1} 
    \Big[R^{s-\frac{N}{p}} (1-s)^{\frac1p}[u]_{W^{s,p}(B_{R})} +
    \|u\|_{L^{\infty}(B_{R})}+\mathrm{Tail}(u;R)\Big].
\end{align*}
Recalling the choice of $q$, we conclude the claimed inequality.
\end{proof}

Theorem~\ref{thm:W1q} ensures that weak solutions admit a gradient in $L^q$ for any $q\ge p$. This result can still be improved, in the sense that the gradient is fractional differentiable to a certain power. 

\begin{proposition}[$W^{\beta,q}$-gradient regularity]\label{thm:beta-q}
Let $p\in[2,\infty)$ and $s\in(\frac{p-2}{p},1)$. Then,  for any locally bounded,  local weak solution $u\in W^{s,p}_{\rm loc}(\Omega)\cap L^{p-1}_{sp}(\R^N)$ of~\eqref{PDE} in the sense of Definition~\ref{def:loc-sol}  we have
$$
    \nabla u\in W^{\alpha,q}_{\rm loc}(\Omega)\qquad 
    \mbox{for any $q\in [p,\infty)$ and $\alpha\in\big(0,\tfrac{sp-(p-2)}{q}\big)$.}
$$
Moreover, there exists a constant  $C=C(N,p,s,q,\alpha)$ such that for any ball $B_{R}\equiv B_{R}(x_o)\Subset \Omega$ and any $r\in(0,R)$, we have
\begin{align*}
    [\nabla u]_{W^{\alpha, q}( B_{r})}^{q}
    \le
    \frac{C}{R^{(1+\alpha) q}} 
    \Big(\frac{R}{R-r}\Big)^{N+q+1}
    \mathbf{K}^{q},
\end{align*}
where $\mathbf{K}$ is defined in~\eqref{def:K-2diff} and the constant is of the form $C=\frac{C(N,p,q)}{s\alpha\beta^q (\beta-\alpha)(1-\alpha)^{q}}$. 
\end{proposition}

\begin{proof}
Throughout the proof we abbreviate $\beta:=\tfrac{sp-(p-2)}{q}$ and fix some $\alpha\in(0,\beta)$. 
By Theorem~\ref{thm:W1q} we know that $u\in W^{1,q}_{\rm loc}(\Omega)$ for any $q\ge p$. This allows us to apply Proposition~\ref{prop:beta-q} on any pair of concentric balls $B_r\Subset B_R\Subset \Omega$ to conclude that $\nabla u\in W^{\alpha,q}( B_{r})$ with the quantitative estimate 
\begin{align*}
    [\nabla u]_{W^{\alpha, q}(B_{r})}^{q}
    \le
    \frac{C(N,p,q)}{s\alpha\beta^q (\beta-\alpha)(1-\alpha)^{q}R^{q(1+\beta)}}
    \Big(\frac{R}{R-r}\Big)^{N+q+1}\mathbf{K}^{q}.
\end{align*}
This proves the claim.
\end{proof}

\begin{remark}\label{Rmk:stability}\upshape
At this point, it is certainly appropriate to compare the $W^{\alpha,q}$-estimate with known results from the local case in order to better understand the significance of the statement. We start with the special case $q=p$. In this case 
the constraint on $\alpha$ reduces to $0<\alpha < \beta:=s-\frac{p-2}{p}$. For the limit we get
$$
    \lim_{s\uparrow 1} \Big( s-\frac{p-2}{p} \Big)
    =
    \frac{2}{p}.
$$
This implies that the constant in the $W^{\alpha,p}$-estimate remains stable in the limit $s\uparrow 1$. Indeed, we have
$$
    \lim_{s\uparrow 1} \frac{1}{\alpha\beta^{p} (\beta-\alpha)(1-\alpha)^{p}}
    =
    \Big(\frac{2}{p}\Big)^p 
    \frac{1}{\alpha (\frac2p -\alpha)(1-\alpha)^{p}},
$$
and the formal limiting condition for the order of fractional differentiability is $0<\alpha <\frac{2}{p}$. This, however, is exactly the condition that appears in the context of weak solutions to the local $p$-Laplacian for $p\ge 2$. In fact, a classical result  for $p$-harmonic functions
\cite{Bojarski:1987, Uhlenbeck, Uraltseva}
ensures   that $|\nabla u|^\frac{p-2}2\nabla u\in W^{1,2}_{\rm loc}(\Omega , \R^N)$.  This higher differentiability can be converted into fractional differentiability  by a standard argument, namely,  $\nabla u\in W^{\alpha,p}_{\rm loc} (\Omega,\R^N)$ for every $0<\alpha<\frac{2}{p}$. This is exactly the outcome that is formally obtained from Theorem~\ref{thm:beta-q} in the limit $s\uparrow 1$.

Next, we consider the case $q>p$. Here, Theorem \ref{thm:beta-q} ensures
that $\nabla u\in W^{\alpha ,q}_{\rm loc}(\Omega ,\R^N)$ for any $0<\alpha <\frac{p}{q}\big( s-\frac{p-2}{p}\big)$. In the limit $s\uparrow 1$ we have
\begin{align*}
    \lim_{s\uparrow 1}
    \frac{p}{q}\Big( s-\frac{p-2}{p}\Big)
    &=
    \frac{2}{q},
\end{align*}
so that the constant in the $W^{\alpha,p}$-estimate remains stable in the limit $s\uparrow 1$. In fact, we have
\begin{align*}
      \lim_{s\uparrow 1} 
      \frac{1}{\alpha\beta^{q} (\beta-\alpha)(1-\alpha)^{q}}
      &=
      \Big(\frac{2}{q}\Big)^q
    \frac{1}{\alpha (\frac2q -\alpha)(1-\alpha)^{q}}.
\end{align*}
The formal limiting condition for the order of fractional differentiability  is $0<\alpha <\frac2{q}$. This exactly coincides  with the condition that appears in the context of weak solutions to the local $p$-Laplacian equation \cite{Dong:2020, Sarsa:2020}. There it was shown that $|\nabla u|^\frac{q-2}{2}\nabla u\in W^{1,2}_{\rm loc}(\Omega ,\R^N)$  for any $q\ge p$.  More precisely, the result from \cite{Sarsa:2020}  ensures for  $N\ge 2$, $p>1$, and $\sigma >-1-\frac{p-1}{N-1}$ that for any local weak $p$-harmonic function $u\in W^{1,p}(\Omega)$ we have
$
    |\nabla u|^\frac{p-2+\sigma}{2}\nabla u\in W^{1,2}_{\rm loc}(\Omega ,\R^N)
$.
This weak differentiability can be converted into fractional differentiability, that is,  $\nabla u\in W^{\alpha,q}_{\rm loc} (\Omega,\R^N)$ for every $0<\alpha<\frac{2}{q}$. Observe that this is again the outcome that is formally obtained  from Theorem \ref{thm:beta-q} in the limit $s\uparrow 1$.

In the local case the assertions concerning the fractional differentiability 
are a consequence of the weak differentiability of $|\nabla u|^\frac{q-2}2\nabla u$
and the elementary inequality 
$$
    \big| |A|^\frac{q-2}2A-|B|^\frac{q-2}2B\big|^2
    \ge 
    \tfrac{1}{C(q)} |A-B|^q, \quad \mbox{for $q\ge 2$, $A,B\in \R^N$,}
$$
which itself follows from Lemma \ref{lem:Acerbi-Fusco}. 
In fact, we have
\begin{align*}
    [u]_{W^{\alpha,q}(B_R)}^q
    &\le
    C(q) 
    \iint_{B_R\times B_R}
    \frac{\big| |\nabla u(x)|^\frac{q-2}2\nabla u(x)
    -|\nabla u(y)|^\frac{q-2}2\nabla u(y)\big|^2}{|x-y|^{N+2 \frac{q\alpha}2 }}\,\dx\dy.
\end{align*}
The right-hand side is finite provided we have $0<\frac12 q\alpha <1$. This follows by using  Lemma~\ref{lem:FS-S} and 
$|\nabla u|^\frac{q-2}2\nabla u\in W^{1,2}_{\rm loc}(\Omega)$. But this means that $\alpha$ has to satisfy  $0<\alpha <\frac2{q}$. 
\end{remark}

At this point Theorem~\ref{*thm:beta-q} is obtained by joining the quantitative estimates from Theorem~\ref{thm:W1q} and Proposition~\ref{thm:beta-q}.

\begin{proof}[\textbf{\upshape Proof of Theorem~\ref{*thm:beta-q}}]
The qualitative statement asserting that $\nabla u\in W^{\alpha,q}_{\rm loc}(\Omega)$ for any $q\in [p,\infty)$ and $\alpha\in(0,\beta)$ with $\beta=\tfrac{sp-(p-2)}{q}$ has already been established in Proposition~\ref{thm:beta-q}. The quantitative estimate follows by joining the ones from Proposition~\ref{thm:beta-q} and Theorem~\ref{thm:W1q} and using Lemma~\ref{lem:t} to increase $\tfrac34 R$ in the tail-term to $R$
\begin{align*}
    &[\nabla u]_{W^{\alpha, q}(B_{\frac12 R})}^{q} \\
    &\quad\le
    \frac{C_1}{R^{(1+\alpha) q}} 
    \bigg[R^{q}\int_{B_{\frac34 R}}|\nabla u|^{q}\,\dx +
    R^{N}\big(\|u\|_{L^{\infty}(B_{R})}+\mathrm{Tail}(u;\tfrac34 R)\big)^{q} \bigg] \\
    &\quad\le 
    \frac{C_1 C_2 R^N}{R^{(1+\alpha) q}} 
    \Big[R^{sq-\frac{Nq}{p}} (1-s)^{\frac{q}{p}}[u]_{W^{s,p}(B_{R})}^q +
    \big(\|u\|_{L^{\infty}(B_{R})}+\mathrm{Tail}(u;R)\big)^{q} \Big],
\end{align*}
Where $C_1$ denotes the constant from Theorem~\ref{thm:W1q}, while $C_2$ denotes the one from Proposition~\ref{thm:beta-q}. Tracing back the dependency of the constants, we observe that $C_2$ is stable in the limit  $s\uparrow 1$ and blows up as $s\downarrow \frac{p-2}{p}$ and $C_1$ is of the form 
$$
    C_1
    =
    \frac{C(N,p,q)}{s\alpha\beta^q (\beta-\alpha)(1-\alpha)^{q}}.
$$
Hence, the resulting constant is stable as $s\uparrow 1$ and blows up as $s\downarrow \frac{p-2}{p}$ and $\alpha\uparrow \beta$. 
This finishes the proof of Theorem~\ref{*thm:beta-q}. 
\end{proof}




\end{document}